\documentclass[a4paper]{amsart}

\usepackage{amssymb}
\usepackage{amsfonts}
\usepackage{graphicx, enumerate, verbatim}

\renewcommand{\bar}[1]{\overline{#1}}

\usepackage{amsmath,amssymb,amsfonts,amsthm}
\newtheorem{theorem}{Theorem}[section]

\newtheorem{lemma}[theorem]{Lemma}
\newtheorem{proposition}[theorem]{Proposition}

\newtheorem{fact}[theorem]{Fact}

\theoremstyle{definition}

\newtheorem{remark}[theorem]{Remark}
\newtheorem{example}[theorem]{Example}

\newtheorem{problem}{Problem}

\newtheorem{notation}[theorem]{Notation}

\def\Aut{{\sf Aut}}
\def\Sym{{\sf Sym}}

\newcommand{\cal}{\mathcal}

\def\Sup{\mathop{\rm Sup}\nolimits}

\def\diam{\mathop{\rm diam}\nolimits}

\usepackage{epic,eepic,ecltree} 

\input xy
\xyoption{all}

\usepackage{xy} 
\newcommand{\dedge}[1]{\ar@{--}[#1]}
\newcommand{\edge}[1]{\ar@{-}[#1]}
\newcommand{\arcc}[1]{\ar@{->}[#1]}
\newcommand{\darcc}[1]{\ar@{=>}[#1]}
\newcommand{\barcc}[1]{\ar@{<-}[#1]}
\newcommand{\bdarcc}[1]{\ar@{<=}[#1]}
\newcommand{\lulab}[1]{\ar@{}[l]_<<{#1}}
\newcommand{\rulab}[1]{\ar@{}[r]^<<{#1}}
\newcommand{\ldlab}[1]{\ar@{}[l]^<<{#1}}
\newcommand{\rdlab}[1]{\ar@{}[r]_<<{#1}}
\newcommand{\node}{*+[o][F-]{ }}

\newcommand{\arrowdraw}[2]{\draw [postaction={decorate}] (#1)--(#2);}
\newcommand{\edgedraw}[2]{\draw (#1)--(#2);}

\usepackage{tikz}
\usetikzlibrary{decorations.markings}
\usetikzlibrary{arrows,matrix}
\usepgflibrary{arrows}

\begin{document}
\title[Set-homogeneous digraphs]{Set-homogeneous directed graphs \\ \today}


\maketitle

\begin{center}
		
    ROBERT GRAY\footnote{Supported by an EPSRC Postdoctoral Fellowship EP/E043194/1 held at the University of St Andrews, Scotland. Partially supported by FCT and FEDER, project POCTI-ISFL-1-143 of Centro de \'{A}lgebra da Universidade de Lisboa, and by the project PTDC/MAT/69514/2006.}

    \medskip

    Centro de \'{A}lgebra da Universidade de Lisboa, \\ 
    Av. Prof. Gama Pinto 2,  1649-003 Lisboa,  Portugal.

    \medskip

    \texttt{rdgray@fc.ul.pt}

    \bigskip

    \bigskip

    DUGALD MACPHERSON (corresponding author)\footnote{Supported by EPSRC grants EP/D048249/1 and EP/H00677X/1}

    \medskip

    School of Mathematics, \ University of Leeds, \\
    Leeds LS2 9JT, \ England, tel +441133435166, fax +441133435090.

    \medskip

    \texttt{pmthdm@maths.leeds.ac.uk} \\

    \bigskip

    \bigskip

    CHERYL E. PRAEGER\footnote{Supported by Australian Research Council Federation Fellowship FF0776186.}, \quad GORDON F. ROYLE\footnote{GFR acknowledges support from the Australian Research Council Discovery Grants Scheme (DP0984540).}

    \medskip

Centre for the Mathematics of Symmetry and Computation, \\   
School of Mathematics and Statistics, \ University of Western Australia, \\
    Nedlands, \ WA 6009, \ Australia.

    \medskip

    \texttt{praeger@maths.uwa.edu.au}, \quad \texttt{gordon@maths.uwa.edu.au} \\
\end{center}

\begin{abstract}
A directed graph  is {\em set-homogeneous} if, whenever $U$ and $V$ are isomorphic finite subdigraphs, there is an
 automorphism $g$ of the digraph with $U^g=V$.
Here,  extending work of Lachlan on finite homogeneous digraphs, we classify finite set-homogeneous digraphs, where we allow
 some pairs of vertices to have arcs in both directions. 
Under the assumption that such pairs of vertices are {\em not} allowed, we obtain initial results on countably infinite
 set-homogeneous digraphs, 
classifying those which are not 2-homogeneous.
\end{abstract}

\section{Introduction}

The notion of \emph{homogeneous} structure (sometimes called \emph{ultrahomogeneous}) dates back to the pioneering work of 
Fra{\"{\i}}ss{\'e}; see \cite{ Fraisse1}, or originally \cite{Fraisse2}. 
A relational structure $M$ is called \emph{homogeneous} if any isomorphism between 
finite induced substructures of $M$ extends to an automorphism of $M$. 
Homogeneity is a strong symmetry condition and, as a result of this, for several natural classes of relational structures those that are homogeneous have been classified. 
The finite homogeneous graphs were determined by Gardiner in \cite{gardiner}, and the countable homogeneous graphs were 
classified by Lachlan and Woodrow in \cite{Lachlan1}. The countable homogeneous posets were described by Schmerl in 
\cite{Schmerl1}, and the corresponding classification for  tournaments was given by  
Lachlan \cite{Lachlan2}.  Generalising the results for posets and tournaments, Cherlin \cite{Cherlin1}
classified the homogeneous countable digraphs in a major piece of work. 
Homogeneous structures continue to receive a lot of attention in the literature; see for example \cite{Hubicka1, Hubicka2008, Kechris1, SauerNew}.  

There are various natural ways in which the condition of homogeneity can be weakened. 
In this article we consider one such weakening called \emph{set-homogeneity}: a relational structure $M$ is {\em set-homogeneous}
if, whenever $U$ and $V$ are isomorphic finite induced substructures of $M$, there is $g\in \Aut(M)$ with $U^g=V$.  
The notion of a set-homogeneous structure was introduced by Fra{\"{\i}}ss{\'e} in \cite{Fraisse1}, although unpublished observations had been made on it earlier by  Fra{\"{\i}}ss{\'e} and Pouzet. 

In this paper we investigate the extent to which homogeneity is stronger than set-homogeneity. For instance,    
it was shown by Ronse \cite{Ronse1978} that any finite set-homogeneous graph is in fact homogeneous. 
Enomoto gave a short and very elegant direct proof of this fact in \cite{enomoto}. 
(We shall rehearse Enomoto's proof below 
when showing that every finite set-homogeneous tournament is homogeneous.) 
On the other hand, for infinite graphs homogeneity is known to be strictly stronger than set-homogeneity; see \cite{dgms}. 
Here, we consider how set-homogeneity compares with homogeneity for other natural families of relational structures, in particular  
 for directed graphs.
Lachlan classified the finite homogeneous digraphs in \cite{lachlan}. 
In fact,
 he considers two notions of digraph, differing according to whether or not one allows a pair of vertices with arcs
 in each direction. 
As we shall see (Theorem~\ref{maintheorem}(ii)), for finite digraphs (where one allows at most one arc between any pair of vertices), set-homogeneity is close to being equivalent to homogeneity in the sense that there is just one digraph -- the directed pentagon -- that is set-homogeneous but not homogeneous. On the other hand, for the more general notion of digraph where one does allow arcs in both directions between pairs of vertices, there are (up to complementation) three infinite families -- the digraphs $J_n, K_n[D_5], D_5[K_n]$ -- and several small sporadic examples of finite digraphs which are set-homogeneous but not homogeneous
(Theorem~\ref{maintheorem}(iii)).
We also obtain partial results on countably infinite set-homogeneous digraphs (Theorem~\ref{infnot2hom}).

\bigskip\noindent
\textbf{Digraph definitions:}\quad Before stating the main theorems, we list the relevant classes of graphs and digraphs. We view graphs as structures with a
 single symmetric irreflexive binary relation, denoted $\sim$.
An {\em antisymmetric digraph}, or {\em a-digraph}, has a single binary relation $\rightarrow$ which is irreflexive and 
antisymmetric.
A {\em symmetric digraph}, or {\em s-digraph}, is endowed with a symmetric relation $\sim$ and an 
antisymmetric relation $\rightarrow$, both irreflexive. We refer to pairs $\{u,v\}$ with $u\sim v$ as {\em edges}, and to pairs $(u,v)$ with $u\to v$ as {\em arcs}. An edge $\{u,v\}$ may be viewed as a pair of vertices with an arc in each direction. 
Both a graph and an a-digraph can be viewed as an s-digraph (with $\to$ or $\sim$ empty, respectively), so our main theorem is a classification 
of set-homogeneous s-digraphs. Some definitions are phrased just for s-digraphs. 

\begin{itemize}
 \item[(a)]  If $U$ and $V$ are s-digraphs, then, as in \cite{lachlan}, $U\times V$ denotes their \emph{product}, and $U[V]$ their \emph{compositional 
product}. Both have domain the cartesian product
$U\times V$. In the product $U\times V$, we have $(u_1,v_1)\rightarrow (u_2,v_2)$ if and only if $u_1\rightarrow u_2$ and 
$v_1\rightarrow v_2$, and likewise for $\sim$. On the other hand,
$U[V]$ consists of $|U|$ copies of $V$: the $\rightarrow$ relation consists of the pairs  $(u,v_1)\rightarrow (u,v_2)$ where
$v_1\rightarrow v_2$ in $V$, and $(u_1,v_1)\rightarrow (u_2,v_2)$ where $u_1\rightarrow u_2$ in $U$. Likewise, the $\sim$ relation in $U[V]$ consists of the pairs $(u,v_1)\sim (u,v_2)$ where  $v_1\sim v_2$ in $V$, and $(u_1,v_1)\sim (u_2,v_2)$ where  $u_1\sim u_2$ in $U$.

\item[(b)] For each $n$, $K_n$ denotes the complete graph on $n$ vertices, $K_{m,n}$ the complete bipartite graph with parts of size $m$ and $n$, $C_n$ the undirected cycle on $n$ vertices, and 
$D_n$ the directed cycle on $n$ vertices.
So we may view $C_n$ as having vertex set $\{0,\ldots,n-1\}$ with edge set $\{\{i,j\}: j\equiv i+1 ~({\rm mod}~ n)\}$, and $D_n$, 
on the same vertex set,
just has arcs $i\rightarrow i+1 ~({\rm mod}~ n)$. 

\item[(c)] We define three additional s-digraphs $H_0,H_1,H_2$ constructed in \cite{lachlan} and 
depicted in Figures~\ref{FIG-h0}, \ref{FIG-h1} and \ref{FIG-H2-FULL} below, and described there.
A further $s$-digraph, denoted $H_3$, is constructed in Section~\ref{sec_27vertex}; it has 27 vertices.

\item[(d)] We  use $P_3$ to denote a tournament on 3 vertices which is totally ordered by $\rightarrow$.
Let $E_6$ be a copy of $D_6$ with vertex set $\{0,1,2,3,4,5\}$, and, in addition, edges 
$i \sim i+3 ~~({\rm mod}~ 6)$. Let
 $E_7$ be the corresponding expansion of $D_7$ on $\{0,1,2,3,4,5,6\}$ with edges $i \sim i+3 ~({\rm mod}~ 7)$. Also, 
let $F_6$ be the s-digraph with vertex set $\{x_1,y_1,z_1,x_2,y_2,z_2\}$ such that each $x_iy_iz_i$ is a copy of  $D_3$ with arcs $x_i\rightarrow y_i$, $y_i\rightarrow z_i$ and $z_i\rightarrow x_i$ and with edges $x_1\sim x_2$, $y_1 \sim y_2$ and $z_1 \sim z_2$. 
The s-digraph $J_n$ (for $n>1$) is defined as follows. Let $B_0$, $B_1$, $B_2$ be disjoint independent sets each of size $n$. 
Let $f_{01}: B_0\rightarrow B_1$, $f_{02}: B_0 \rightarrow B_2$, and $f_{12}: B_1\rightarrow B_2$ be
 bijections, with $f_{02}=f_{12}\circ f_{01}$. If $x\in B_i$ and $y \in B_{i+1~(\mod 3)}$, there is an 
arc $x\rightarrow y$ except in the case when $x,y$ are matched by the $f_{ij}$, in which case they are $\sim$-related.

\item[(e)] Finally, if $M$ is an s-digraph, then $\bar{M}$, the {\em complement}, is the s-digraph with the same vertex set, such that
$u\sim v$ in $\bar{M}$ if and only if they are unrelated in $M$, and $u\rightarrow v$ in $\bar{M}$ if and only if $v\rightarrow u$ in $M$.
Thus, for example, $J_2 \cong \overline{E_6}$.

\item[(f)] We let ${\cal L}$, ${\cal A}$, and ${\cal S}$ denote respectively the collections of finite set-homogeneous graphs, finite set-homogeneous
a-digraphs, and finite set-homogeneous s-digraphs. 
\end{itemize}

Our main theorem is part (iii) of the following, with part (ii) as a special case. Part (i) is due to Gardiner \cite{gardiner} (for homogeneous graphs), Ronse \cite{Ronse1978} and  Enomoto \cite{enomoto}.  

\begin{theorem}\label{maintheorem}
Let $M$ be a finite  symmetric digraph. Then, with the above notation
\begin{enumerate}[(i)]

\item $M\in {\cal L}$ if and only if $M$ or $\bar{M}$ is one of: $C_5$, $K_3\times K_3$, $K_m[\bar{K}_n]$ (for $1\leq m,n\in {\mathbb N}$);

\item $M\in {\cal A}$ if and only if $M$ is one of: $D_1$, $D_3$, $D_4$, $D_5$, $H_0$, $\bar{K}_n$, $\bar{K}_n[D_3]$, or
$D_3[\bar{K}_n]$, for some $n\in {\mathbb N}$ with $1\leq n$;

\item $M \in {\cal S}$ if and only if either $M$ or $\bar{M}$ is isomorphic to an s-digraph of one of the
 following forms:
$K_n[A], A[K_n]$, $L$, $D_3[L]$, $L[D_3]$, $H_1$, $H_2$, $H_3$, $E_6$, $E_7$, $F_6$, $J_n$, where 
$n\in {\mathbb N}$ with $1\leq n$, $A\in {\cal A}$ 
and $L\in {\cal L}$.
\end{enumerate}
\end{theorem}

%
%
%
%

\begin{figure}
\begin{center}
\begin{tikzpicture}
[decoration={ 
markings,
mark=
at position 0.5 
with 
{ 
\arrow[scale=1.5]{stealth} 
} 
} 
] 
\tikzstyle{vertex}=[circle,draw=black, fill=white, inner sep = 0.75mm]

\matrix[row sep=5mm,column sep = 2cm]{

%
%

\node (v01) [vertex,label={90:{\small $(0,1)$}}] at (90:3cm) {};
\node (v11) [vertex,label={135:{\small $(1,1)$}}] at (135:3cm) {};
\node (v12) [vertex,label={180:{\small $(1,-1)$}}] at (180:3cm) {};
\node (v20) [vertex,label={225:{\small $(-1,0)$}}] at (225:3cm) {};
\node (v02) [vertex,label={270:{\small $(0,-1)$}}] at (270:3cm) {};
\node (v22) [vertex,label={315:{\small $(-1,-1)$}}] at (315:3cm) {};
\node (v21) [vertex,label={0:{\small $(-1,1)$}}] at (0:3cm) {};
\node (v10) [vertex,label={45:{\small $(1,0)$}}] at (45:3cm) {};

\draw [postaction={decorate}] (v11)--(v01);
\arrowdraw{v01}{v22}
\arrowdraw{v22}{v02}
\arrowdraw{v02}{v11}

\arrowdraw{v10}{v01}
\arrowdraw{v01}{v20}
\arrowdraw{v20}{v02}
\arrowdraw{v02}{v10}

\arrowdraw{v12}{v01}
\arrowdraw{v01}{v21}
\arrowdraw{v21}{v02}
\arrowdraw{v02}{v12}

\arrowdraw{v10}{v11}
\arrowdraw{v11}{v20}
\arrowdraw{v20}{v22}
\arrowdraw{v22}{v10}

\arrowdraw{v11}{v12}
\arrowdraw{v12}{v22}
\arrowdraw{v22}{v21}
\arrowdraw{v21}{v11}

\arrowdraw{v20}{v12}
\arrowdraw{v12}{v10}
\arrowdraw{v10}{v21}
\arrowdraw{v21}{v20}
\\
$H_0$ \\
};

\end{tikzpicture}
\end{center}
\caption{The digraph $H_0$. This digraph has the following description. 
The vertices are the $8$ non-zero vectors in $GF(3)^2$ and there is an arc from $u$ to $v$ if $\left| u^T \ v^T \right| = 1$ 
(where $\left| u^T \ v^T \right|$ is the determinant of the $2 \times 2$ matrix with columns $u^T$ and $v^T$).}
\label{FIG-h0}
\end{figure}
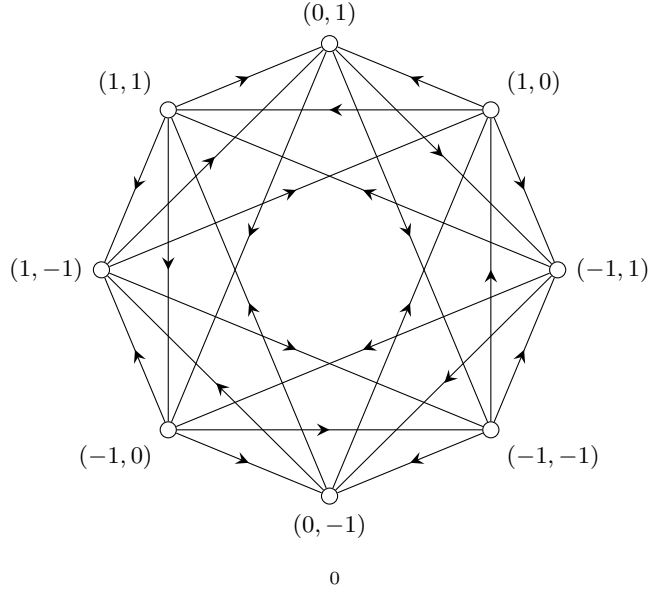

%
%
%
%

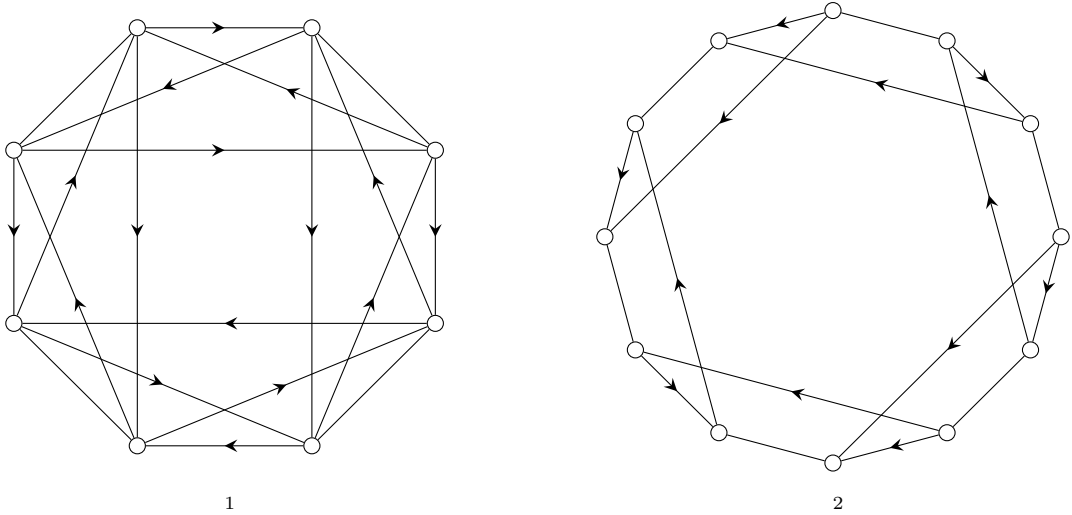
\begin{figure}

\begin{center}

\begin{tikzpicture}
[decoration={ 
markings,
mark=
at position 0.5 
with 
{ 
\arrow[scale=1.5]{stealth} 
} 
} 
] 
\tikzstyle{vertex}=[circle,draw=black, fill=white, inner sep = 0.75mm]

%
%

\matrix[row sep=5mm,column sep = 2cm]{

%
%
%
%

%
%

\node (v01) [vertex] at (90+22.5:3cm)  {};
\node (v11) [vertex] at (135+22.5:3cm) {};
\node (v12) [vertex] at (180+22.5:3cm) {};
\node (v20) [vertex] at (225+22.5:3cm) {};
\node (v02) [vertex] at (270+22.5:3cm) {};
\node (v22) [vertex] at (315+22.5:3cm) {};
\node (v21) [vertex] at (0+22.5:3cm)   {};
\node (v10) [vertex] at (45+22.5:3cm)  {};

\edgedraw{v11}{v01}
\edgedraw{v10}{v21}
\edgedraw{v12}{v20}
\edgedraw{v02}{v22}

\arrowdraw{v01}{v10}
\arrowdraw{v01}{v20}

\arrowdraw{v10}{v11}
\arrowdraw{v10}{v02}

\arrowdraw{v21}{v22}
\arrowdraw{v21}{v01}

\arrowdraw{v22}{v10}
\arrowdraw{v22}{v12}

\arrowdraw{v02}{v21}
\arrowdraw{v02}{v20}

\arrowdraw{v20}{v22}
\arrowdraw{v20}{v11}

\arrowdraw{v12}{v01}
\arrowdraw{v12}{v02}

\arrowdraw{v11}{v21}
\arrowdraw{v11}{v12}

&

\node (v1) [vertex] at (0:3cm)  {};
\node (v2) [vertex] at (30:3cm)  {};
\node (v3) [vertex] at (60:3cm)  {};
\node (v4) [vertex] at (90:3cm)  {};
\node (v5) [vertex] at (120:3cm)  {};
\node (v6) [vertex] at (150:3cm)  {};
\node (v7) [vertex] at (180:3cm)  {};
\node (v8) [vertex] at (210:3cm)  {};
\node (v9) [vertex] at (240:3cm)  {};
\node (v10) [vertex] at (270:3cm)  {};
\node (v11) [vertex] at (300:3cm)  {};
\node (v12) [vertex] at (330:3cm)  {};

\edgedraw{v1}{v2}
\edgedraw{v3}{v4}
\edgedraw{v5}{v6}
\edgedraw{v7}{v8}
\edgedraw{v9}{v10}
\edgedraw{v11}{v12}

\arrowdraw{v1}{v12}
\arrowdraw{v1}{v10}

\arrowdraw{v2}{v5}

\arrowdraw{v3}{v2}

\arrowdraw{v4}{v5}
\arrowdraw{v4}{v7}

\arrowdraw{v6}{v7}

\arrowdraw{v8}{v9}

\arrowdraw{v9}{v6}

\arrowdraw{v11}{v8}
\arrowdraw{v11}{v10}

\arrowdraw{v12}{v3}
\\
$H_1$ & $H_2$ \\
};

\end{tikzpicture}

\end{center}
\caption{The digraphs $H_1$ and $H_2$. As in \cite{lachlan}, in the diagram for $H_2$ we have omitted 
most of the arcs. The remaining arcs are obtained by carrying out the following process. In $H_2$ each 
vertex $v$ has a unique mate $v'$ to which it is joined by an undirected edge (as shown in the diagram). Now if $v \rightarrow w$
 then $w \rightarrow v'$ where $v'$ is the mate of $v$. Similarly, if $w \rightarrow v$ then 
$v' \rightarrow w$. This leads to the insertion of another $36$ arcs. The resulting digraph, with all the arcs included, is depicted in Figure~\ref{FIG-H2-FULL}.} 
\label{FIG-h1}
\end{figure}

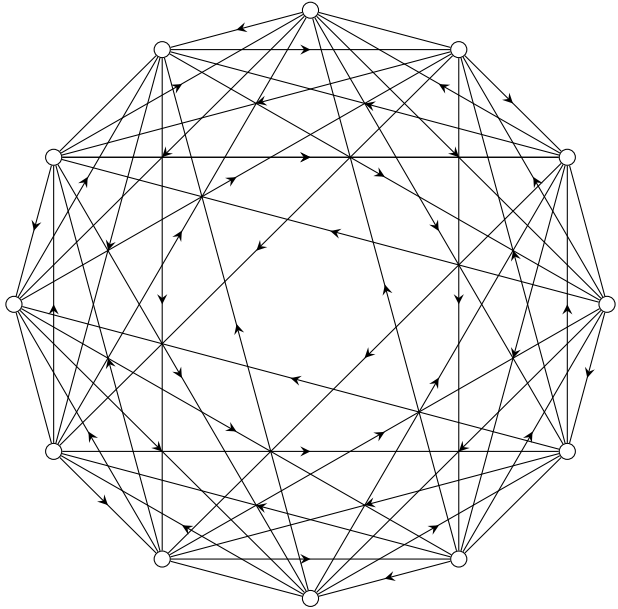
\begin{figure}
\begin{center}
\begin{tikzpicture}
[scale=.6,
decoration={ 
markings,
mark=
at position 0.5 
with 
{ 
\arrow[scale=1.2]{stealth} 
} 
} 
] 
\tikzstyle{vertex}=[circle,draw=black, fill=white, inner sep = 0.75mm]

%
%

\node (v1) [vertex,label={0:{}}] at (0:6.5cm)  {};
\node (v2) [vertex,label={30:{}}] at (30:6.5cm)  {};
\node (v3) [vertex,label={60:{}}] at (60:6.5cm)  {};
\node (v4) [vertex,label={90:{}}] at (90:6.5cm)  {};
\node (v5) [vertex,label={120:{}}] at (120:6.5cm)  {};
\node (v6) [vertex,label={150:{}}] at (150:6.5cm)  {};
\node (v7) [vertex,label={180:{}}] at (180:6.5cm)  {};
\node (v8) [vertex,label={210:{}}] at (210:6.5cm)  {};
\node (v9) [vertex,label={240:{}}] at (240:6.5cm)  {};
\node (v10) [vertex,label={270:{}}] at (270:6.5cm)  {};
\node (v11) [vertex,label={300:{}}] at (300:6.5cm)  {};
\node (v12) [vertex,label={330:{}}] at (330:6.5cm)  {};

\edgedraw{v1}{v2}
\edgedraw{v3}{v4}
\edgedraw{v5}{v6}
\edgedraw{v7}{v8}
\edgedraw{v9}{v10}
\edgedraw{v11}{v12}

\arrowdraw{v1}{v12}
\arrowdraw{v1}{v10}

\arrowdraw{v2}{v5}

\arrowdraw{v3}{v2}

\arrowdraw{v4}{v5}
\arrowdraw{v4}{v7}

\arrowdraw{v6}{v7}

\arrowdraw{v8}{v9}

\arrowdraw{v9}{v6}

\arrowdraw{v11}{v8}
\arrowdraw{v11}{v10}

\arrowdraw{v12}{v3}


\arrowdraw{v11}{v1}		\arrowdraw{v2}{v4}		\arrowdraw{v3}{v11}
\arrowdraw{v9}{v1}		\arrowdraw{v6}{v2}		\arrowdraw{v7}{v3}
\arrowdraw{v5}{v1}    \arrowdraw{v10}{v2}   \arrowdraw{v5}{v3}
\arrowdraw{v1}{v3}    \arrowdraw{v12}{v2}																

\arrowdraw{v4}{v12}		\arrowdraw{v5}{v9}		\arrowdraw{v6}{v2}
\arrowdraw{v6}{v4}		\arrowdraw{v7}{v5}		\arrowdraw{v6}{v10}
\arrowdraw{v8}{v4}													\arrowdraw{v8}{v6}

\arrowdraw{v9}{v7}		\arrowdraw{v8}{v12}		\arrowdraw{v9}{v11}
\arrowdraw{v7}{v11}		\arrowdraw{v10}{v8}		

\arrowdraw{v10}{v12}

\arrowdraw{v4}{v1}			\arrowdraw{v3}{v6}			\arrowdraw{v10}{v5}
\arrowdraw{v1}{v6}			\arrowdraw{v3}{v8}			\arrowdraw{v7}{v10}
\arrowdraw{v2}{v9}			\arrowdraw{v11}{v4}			\arrowdraw{v12}{v7}
\arrowdraw{v2}{v11}			\arrowdraw{v5}{v8}

\arrowdraw{v12}{v9}

\end{tikzpicture}
\end{center}
\caption{The digraph $H_2$ with all of the arcs included.}
\label{FIG-H2-FULL}
\end{figure}

We note that, for the classes ${\cal A}$ and ${\cal S}$, set-homogeneity is strictly weaker than
 homogeneity -- see Example~\ref{ex:c5}. 
Thus, the short argument of Enomoto for graphs, in combination with the Lachlan classification of 
finite homogeneous a-digraphs and s-digraphs,
 is not directly applicable for us. Our proof yields a proof of the Lachlan theorem \cite{lachlan}. It is expressed differently, though 
some of the ideas are similar, and some arguments are taken from \cite{lachlan}.

\begin{example}\label{ex:c5} 

\

\begin{enumerate}[(i)]
\item 
The directed cycle $D_5$ on $\{a_0,a_1,a_2,a_3,a_4\}$, with $a_i\rightarrow a_j$ if and only if
 $j=i+1 ~~({\rm mod}~ 5)$, is an a-digraph which is  set-homogeneous but 
not homogeneous. To see this, observe that the two sets $W=\{a_0,a_2\}$ and $
W'=\{a_0,a_3\}$ induce isomorphic digraphs but there is no automorphism of $D_5$
extending the isomorphism $h:W\rightarrow W'$ given by $h:a_0\rightarrow a_0,\ 
a_2\rightarrow a_3$.
\item
The s-digraphs $E_6,E_7,F_6, H_3, K_n[D_5], D_5[K_n]$ and $J_n$ (for $n\geq 2$) are  set-homogeneous but not homogeneous.
\end{enumerate}
\end{example}

\begin{remark} \label{weakcomp}
\
\begin{enumerate}
 \item[(i)]  By inspection of the lists in Theorem~\ref{maintheorem}, we see that the only finite set-homogeneous s-digraphs with vertex-primitive automorphism group
are, up to complementation, $K_n, C_5, D_3$, $D_5$, $E_7, K_3 \times K_3$.

\item[(ii)] The a-digraph $H_0$ is isomorphic to the one obtained from it by reversing arcs. One such isomorphism is the following map $g$ (in the notation of Figure 1). $(1,-1)^g=(1,-1)$; $(-1,1)^g=(-1,1)$; $((0,-1), (-1,-1))^g=((-1,-1), (0,-1))$; $((1,1), (0,1))^g=((0,1),(1,1))$; $((1,0), (-1,0))^g=((-1,0), (1,0))$.

\item[(iii)] Under our definition above, one obtains the complement $\bar{M}$ of an $s$-digraph $M$ by reversing arcs, and interchanging the relation $\sim$ with the relation `unrelated'. We could have defined an alternative notion of {\em weak complement} $M^c$ of $M$, whereby $u\sim v$ in $M^c$ if and only if $u\sim v$ in $M$, but $u \to v$ in $M^c$ if and only if $v \to u$ in $M$. However, it can be checked that each $s$-digraph in Theorem~\ref{maintheorem} is isomorphic to its weak complement. (An isomorphism in the case of $H_0$ is given in part (b).)

\end{enumerate}
\end{remark}
A relational structure $M$ will be said to be {\em $k$-homogeneous} if any isomorphism between induced substructures of $M$ of size $k$ extends to an automorphism. We say that $M$ is {\em $k$-set-homogeneous} if, whenever $U,V$ are isomorphic induced substructures of $M$ of size $k$, there is $g\in \Aut(M)$ with $U^g=V$.

In Sections 6 and 7, we investigate countably infinite set-homogeneous a-digraphs, and classify those which 
are not 2-homogeneous. As a particular case, Droste \cite[Proposition 8.13]{droste}, using different terminology, showed that every countably infinite set-homogeneous partial order is homogeneous.
A complete classification of the countably infinite set-homogeneous a-digraphs seems difficult.
 Our methods are rather similar to those of \cite{dgms}, 
where countably infinite set-homogeneous graphs which are not 3-homogeneous were classified.
We define  a-digraphs $T(4)$ and $R_n$ (for $n\geq2$), proving that they are all set-homogeneous but not 2-homogeneous, and that the automorphism group of $T(4)$ is vertex-primitive  while those for the $R_n$ are imprimitive on vertices (see Lemma 6.1). 
Our main result, below, proved in Section 6, is that $T(4)$ and the $R_n$ are the only examples which are not 2-homogeneous.

\begin{theorem} \label{infnot2hom}
Let $M$ be a countably infinite set-homogeneous digraph which is not $2$-homogeneous. Then $M$ is isomorphic to 
$T(4)$ or to $R_n$ for some $n\geq 2$.
\end{theorem}

We recall one standard piece of permutation group terminology. If $G$ is a transitive permutation group on a set $\Omega$,
 then an 
{\em orbital} of $G$ is  a $G$-orbit $\Lambda$ on $\Omega^2$. The orbital {\em paired} to $\Lambda$ is
$\Lambda^*:=\{(y,x):(x,y)\in \Lambda\}$. We may view $\Lambda$ as the arc set of a directed graph with vertex set $\Omega$
 on which
$G$ acts arc-transitively as a group of automorphisms. We write $\Lambda(\alpha):=\{x:(\alpha,x)\in \Lambda\}$. By easy 
counting arguments,
if $\Omega$ is finite then $|\Lambda(\alpha)|=|\Lambda^*(\alpha)|$. 
The $G_\alpha$-orbits  $\Lambda(\alpha)$ and $\Lambda^*(\alpha)$ are referred to as 
{\em paired suborbits} of $G$. 
We write $\Lambda \circ \Lambda$ 
for
$\{(\alpha,\beta): \exists \gamma((\alpha,\gamma),(\gamma,\beta) \in \Lambda)\}$. Also,
$\Lambda\circ \Lambda(\alpha):=\{\beta:(\alpha,\beta)\in \Lambda\circ \Lambda\}$.

We say that a permutation group $G$ on $\Omega$ is {\em $k$-homogeneous} if it is transitive on the collection of unordered
 $k$-subsets of $\Omega$, and that
 it is {\em $k$-transitive} if it is transitive on the ordered $k$-tuples of distinct points of $\Omega$. 
It is {\em highly homogeneous} (respectively, {\em highly transitive}) if it is $k$-homogeneous (respectively, $k$-transitive)
 for all $k$. 
We freely use without explicit reference the following result, which follows from Theorem 2 of \cite{livwag}.

\begin{fact} \label{livwag}
Any highly homogeneous finite permutation group, other than the cyclic group of order $3$ acting regularly,
 is $2$-transitive.
\end{fact}

\begin{notation} \label{notat}
In Section 1--5, 
$M$  always denotes a finite set-homogeneous s-digraph (possibly an a-digraph), and $G=\Aut(M)$. We describe below three orbitals $\Gamma,\Delta, \Lambda$, but slightly abusing convention, we allow the possibility that some of these are empty. We denote by $\Gamma$ the
 $G$-orbital such that
$\alpha \rightarrow \beta$ if and only if $(\alpha,\beta) \in \Gamma$. (The symbol $\to$ may occasionally be used also for functions $f:A\to B$, but no confusion should arise.) If $M$ has undirected edges, then there 
is another specified orbital $\Delta$ such that
 if $(\alpha,\beta)\in \Delta$ then $\alpha \sim \beta$. The orbital $\Delta$ may or may not be self-paired (often a key case division in proofs -- note that set-homogeneity implies that the pairs $\{\alpha,\beta\}$ with $\alpha \sim \beta$ form a single $G$-orbit). We say $\alpha,\beta$ are {\em $\sim$-related} if $(\alpha,\beta)\in \Delta \cup \Delta^*$. We reserve the words  {\em edge} for $\sim$-related pairs, and  {\em arc} for pairs $(x,y)$ with $x\to y$. We use $\Lambda$
to denote an orbital on independent pairs (also called `unrelated' pairs), which again, may or may not be self-paired, and write $\alpha \parallel \beta$ if and only if $(\alpha, \beta) \in \Lambda \cup \Lambda^*$. So $\alpha\parallel \beta$ means that $\alpha,\beta$ are distinct and are unrelated, that is, not related by either $\Gamma$ or $\Delta$ (again, set-homogeneity implies that such unordered pairs form a single $G$-orbit).
A set $S$ of vertices in an s-digraph or a-digraph is {\em independent} if it carries an induced digraph structure with no $\Delta$-edges and no $\Gamma$-arcs. If $u,v$ are vertices of an s-digraph or a-digraph, then $d(u,v)$ denotes the number of arcs in a shortest directed path from  $u$ to $v$, and takes value $\infty$ if there is no such path.
A {\em $2$-arc} in a digraph with arc set $\Gamma$ is a sequence $(x,y,z)$ such that $(x,y)\in \Gamma$ and $(y,z)\in \Gamma$.
If $A$ is a subset of the vertex set of the s-digraph $M$, and $x$ is a vertex, we say $x$ {\em dominates} $A$ if $x\to a$ for all $a\in A$. Likewise, we say $x$ {\em dominates} $a$ if $x\to a$.
We use the symbols $\supset$ and $\subset$ for {\em strict} set containment.

Where a permutation group $G$ on $X$ is imprimitive, and there is a specific system of imprimitivity (or proper non-trivial {\em $G$-congruence}) clear from the context, we occasionally
 refer
to the classes of this congruence as {\em blocks}.

We denote by ${\mathbb Z}_n$ the cyclic group of order $n$.
\end{notation}

The proof of Theorem~\ref{maintheorem} is intricate, but the basic ideas are very simple. 
It is quite straightforward to verify that each of the digraphs listed in the theorem is set-homogeneous (in the case of $H_3$, for convenience, this verification may be made computationally; see Section~\ref{sec_27vertex}). Thus, Sections 3-5 are concerned with proving the converse result, that every finite set-homogeneous digraph appears in our list. 
We run in parallel
the proofs
of (ii) and (iii), at least the initial steps, but the proof of (ii) is much simpler, since there 
are fewer examples. First, we reduce to the case when $M$ is $\Gamma$-connected (Lemma~\ref{disconnect}). Under
 this assumption, we classify examples in the special case where there are distinct $\alpha,\beta\in M$ with
$\Gamma(\alpha)=\Gamma(\beta)$ (Lemma~\ref{samenbours}). Next, it is easily seen (Lemma~\ref{nhood}) that the subdigraphs induced on 
$\Gamma(\alpha)$ and $\Gamma^*(\alpha)$ are themselves set-homogeneous. In Lemma~\ref{isom} (for a-digraphs) and Section 4 (for s-digraphs) we handle the case 
where $\Gamma(\alpha)\cong \Gamma^*(\alpha)$; there is
$g\in G$ with $(\Gamma(\alpha))^g=\Gamma^*(\alpha)$, and $\alpha, \alpha^g$ and
 $\alpha^{g^{-1}}$ are equivalent under a natural $G$-conguence (Lemma~\ref{isom2}). Thus, we reduce to the case when 
$\Gamma(\alpha)$ and 
$\Gamma^*(\alpha)$ are non-isomorphic set-homogeneous a-digraphs (or s-digraphs, in (iii)). Arguing 
inductively, we may assume they belong to the list
in Theorem~\ref{maintheorem} (ii) (or (iii)). As noted above, $|\Gamma(\alpha)|=|\Gamma^*(\alpha)|$. Thus,
 there are very limited possibilities, and these are handled by ad hoc arguments, which ultimately prove that every finite set-homogeneous digraph $M$ satisfies $\Gamma(\alpha) \cong \Gamma^*(\alpha)$. This case is handled easily for a-digraphs at the end of Section 3, but 
 it requires a lot of work for s-digraphs -- see Section 5.
 
 Only elementary group 
theory is used -- in particular, there is no use of the classification of finite simple groups.

\section{A $27$ vertex set-homogeneous digraph}
\label{sec_27vertex}

Before embarking on the proof of the main theorem, in this section we shall first complete our description of the set-homogeneous digraphs listed in Theorem~\ref{maintheorem}, by giving a construction for the largest sporadic example $H_3$. We begin with a description of a distance transitive undirected graph, and then go on to use it to define the digraph $H_3$.  

\subsection*{\boldmath The 3-fold cover of $K_9$}

Consider the graph $X$ with 27 vertices
\begin{equation*}
V(X) = \{(u, \alpha) \mid u \in GF(3)^2, \alpha \in GF(3)\}
\end{equation*}
where $GF(3)^2$ is the space of $2$-dimensional row vectors.  The $\sim$-relation is given by
\begin{equation*}
(u, \alpha) \sim (v, \beta)  \Longleftrightarrow \left| u^T \ v^T \right| = \alpha-\beta
\end{equation*}
where  $\left| u^T \ v^T \right|$ is the determinant of the $2\times 2$ matrix with columns $u^T$ and $v^T$.

Then \verb+nauty+ \cite{nauty} tells us that $G = {\rm Aut}(X)$ has order $1296 = 48 \times 27$, and with a little thought we can identify its elements.

\begin{itemize}
\item For each $x \in GF(3)^2$, there is an automorphism
\begin{equation*}
(u, \alpha) \mapsto \left( u + x, \left| u^T \ x^T \right| + \alpha \right)
\end{equation*}
Also, for each $\gamma\in GF(3)$ there is an automorphism $(u,\alpha)\mapsto (u,\alpha+\gamma)$.
The automorphisms of these two types generate a group $K$ of order 27 isomorphic to $(C_3 \times C_3) \times C_3$, which acts regularly on $V(X)$.

\item each $M \in GL(2,3)$ acts as follows
\begin{equation*}
 (u, \alpha) \mapsto (uM, \left|M\right| \alpha)
\end{equation*}
where $|M|$ is the determinant of $M$. Notice that $GL(2,3)$ stabilizes $((0,0), 0)$ and as it has order 48, this is the full vertex stabilizer.

\end{itemize}
Every element of the group $G$ is then a product of elements from these two subgroups and so $G = GL(2,3) K$.
It is now possible to completely determine the structure of $G$, but for our purposes this will not be necessary.

The graph $X$ is a distance-transitive graph of diameter three with 
parameters $\{8,6,1;1,3,8\}$ and hence (see for example \cite[Section 4.5]{GodsilAndRoyle}) 
it is a 3-fold cover of $K_9$. Each of the fibres of the cover consists of a 
triple of vertices of the form 
\begin{equation*}
\left\{ (u,0), (u,1), (u,-1) \right\}
\end{equation*}
and these vertices are mutually at distance 3 from each other. The 9 triples form a system of imprimitivity.

The group $G$ has four self-paired orbitals, namely the diagonal and three others:
\begin{description}
\item[${\cal O}_1$] The orbit containing
$\left( ((0,0), 0) , ((0,1), 0) \right)$, which consists of the edges or ``distance-1-graph'' of $X$.
\item[${\cal O}_2$] The orbit containing
$\left( ((0,0), 0) , ((0,1), 1) \right)$, which forms the ``distance-2-graph'' of $X$.
\item[${\cal O}_3$] The orbit containing 
$\left( ((0,0), 0) , ((0,0), 1) \right)$
which forms the the ``distance-3-graph'' of $X$ and is  a collection of 9 triangles.
\end{description}

The 3-fold cover of $K_9$ was first constructed as a coset graph in Example 3.5 of \cite{glp}, and proved there to be distance transitive. Moreover, it was proved to be the unique distance transitive antipodal 3-fold cover of $K_9$ in the classification result ``Main Theorem" of that paper. It occurs in case (5)(c) with $q=2$. (Note that such a graph is not bipartite so does not occur in case (1), and note also that in Example 3.4 of \cite{glp} the value of $r$ is at least 5 in the unitary case so that the graph does not occur in case (4)(b).)

\subsection*{\boldmath The s-digraph $H_3$}

If two vertices $(u, \alpha)$ and $(v, \beta)$ have distance 2 in $X$ then $u \not= v$ and either 
\begin{equation}\label{plus}
\left| u^T \ v^T \right| = (\alpha-\beta) + 1
\end{equation}
or
\begin{equation}\label{minus}
\left| u^T \ v^T \right| = (\alpha-\beta) - 1
\end{equation}
and so we can define two relationships ${\cal O}_2^+$, ${\cal O}_2^-$ according to whether \eqref{plus} or \eqref{minus} holds respectively, where clearly ${\cal O}_2 = {\cal O}_2^+ \cup {\cal O}_2^-$.  Similarly we can partition ${\cal O}_3 = {\cal O}_3^+ \cup {\cal O}_3^-$. The relations ${\cal O}_1$, ${\cal O}_2^+$, ${\cal O}_2^-$, ${\cal O}_3^+$, ${\cal O}_3^-$ are then the orbitals of the group $H = SL(2,3) K$ which is a subgroup of index two in $G$.

The s-graph $H_3$ is obtained by taking  the relation $\sim$ to consist of the pairs in $O_3^+ \cup O_3^-$, and the relation $\to$ to consist of the pairs in $O_2^+$.
\begin{equation*}
E(H_3) = {\cal O}_2^+ \cup {\cal O}_3
\end{equation*}
In other words, we take 9 triangles on the fibres and ``one half'' of each edge of the distance-2 graph, while the unrelated pairs form a copy of the 3-fold cover of $K_9$.
In the terminology of Notation~\ref{notat}, 
$$\Delta=O_3^+, \Gamma=O_2^+, \mbox{~and~} \Lambda=O_1.$$

 For $\omega= ((0,0),0)$ the ``out-neighbourhood'' $\Gamma(\omega)$ of $\omega$ consists of 8 vertices
\begin{center}
\begin{tabular}{cccc}
$((0,1), 1)$ & $((0,-1), 1)$ & $((1,0), 1)$ & $((-1,0), 1)$\\
$((1,1), 1)$ & $((-1,-1), 1)$ & $((1,-1), 1)$ & $((-1,1), 1)$
\end{tabular}
\end{center}
and Figure~\ref{FIG-h0} above shows the digraph that they induce is in fact $H_0$ (where each vertex  $(u,1)$ is simply labelled with $u$). The ``in-neighbourhood'' $\Gamma^*(\omega)$ is simply the 8 vertices
\begin{center}
\begin{tabular}{cccc}
$((0,1), -1)$ & $((0,-1), -1)$ & $((1,0), -1)$ & $((-1,0), -1)$\\
$((1,1), -1)$ & $((-1,-1), -1)$ & $((1,-1), -1)$ & $((-1,1), -1)$
\end{tabular}
\end{center}
and obviously the mapping $(u,1) \mapsto (u,-1)$ is an isomorphism from $\Gamma(\omega)$ to $\Gamma^*(\omega)$.

Testing set-homogeneity appears to be a task best suited to a computer; using an orderly algorithm we compute one representative of each orbit of $H$ on subsets of $H_3$, and then confirm that the induced subdigraphs are pairwise non-isomorphic.

\section{Proof of Theorem~\ref{maintheorem} parts (i) and (ii).}
We first prove some lemmas which yield a rapid proof of Theorem~\ref{maintheorem}(ii). Very similar ideas, 
together with knowledge of (ii), then yield a (much longer) proof of (iii).  

Throughout, $M$ will always denote a finite set-homogeneous s-digraph (possibly an a-digraph), and $G = \Aut(M)$. We freely use without specific mention that, by set-homogeneity, $G$ is transitive on the  vertex set of $M$.
We remark that any finite vertex-transitive connected digraph is strongly connected; that is, if there is an unoriented path
 between any two vertices,
then between any two vertices, there is an oriented path in each direction. See for example Lemma 5.9 of 
\cite{bmmn}.

Recall that a {\em tournament} is an a-digraph $M$ such that for any distinct vertices $a,b$
of $M$, exactly one of $a\to b$ or $b\to a$ holds. 

\begin{lemma} \label{sethomtourn}
Let $M$ be a finite set-homogeneous tournament. Then $M$ is homogeneous, so of size $1$ or isomorphic to $D_3$.
\end{lemma}

\begin{proof} To show that $M$ is homogeneous, we use the idea of Enomoto \cite{enomoto}
 for finite set-homogeneous graphs. For the second assertion, by \cite{lachlan} any finite homogeneous tournament has size 1 or is isomorphic to the directed triangle $D_3$.

Let $M$ be a finite set-homogeneous tournament and let $\phi:A \rightarrow B$ be an isomorphism between induced 
subdigraphs of $M$. We show how to extend $\phi$ to an isomorphism $A\cup\{a\}\rightarrow B\cup\{b\}$ for some $a\in M\setminus A, b\in M\setminus B$. Repeating this process several times extends $\phi$ to an
automorphism of $M$. Among all vertices of $M \setminus A$
 choose $a \in M \setminus A$ so that the set $A' := \Gamma(a) \cap A$ is as large as possible. (We may assume $A'\neq \varnothing$; for if $A'=\varnothing$, then either apply the argument below to the tournament obtained from $M$ by reversing all arcs, or argue directly as follows. If $A'=\varnothing$ then $x\rightarrow a$ for all $x\in A$ and $a\in M\setminus A$; by set-homogeneity there is an automorphism $g$ of $M$ taking $A$ to $B$, and hence $y\rightarrow b$ for all $y\in B, b\in M\setminus B$; thus, for any $a\in M\setminus A$ and $b\in M\setminus B$ the map $\phi$ has a one-point extension $\hat{\phi}$ with $\hat{\phi}(a)=b$.) Define $B' := A'^\phi$.
 By set-homogeneity, since $A' \cong B'$ there is an automorphism of $M$ sending $A'$ to $B'$ (setwise). It follows that
 the number of vertices of $M$ dominating $A'$ equals the number dominating $B'$. Also the isomorphism $\phi: A \rightarrow B$ 
sends the set of vertices of $A$ that dominate $A'$ bijectively to the set of vertices of $B$ dominating $B'$. It follows that
 there exists $b \in M \setminus B$ dominating $B'$. Moreover, $b$ does not dominate any vertex of $B \setminus B'$, since if 
it did then applying the above argument in reverse would yield a vertex $a' \in M \setminus A$ dominating more than $|A'|$ 
vertices of $A$, contradicting the maximality of $|\Gamma(a) \cap A|$ for the original choice of $a$. It follows that the 
extension of $\phi$ to $\hat{\phi}:A \cup \{ a \} \rightarrow B \cup \{ b \}$ given by defining $a^{\hat{\phi}} = b$ is an 
isomorphism.  
\end{proof}

We say that $M \in \mathcal{S}$ is \emph{$\Gamma$-connected} if between every pair of vertices of $M$ there is a path where successive terms in the path are $(\Gamma\cup\Gamma^*)$-related (that is, the a-digraph obtained by removing every $\Delta$-edge from $M$ is connected). Recall Notation~\ref{notat}, which is used in the next lemma and indeed throughout the proof of Theorem~\ref{maintheorem}. 

\begin{lemma} \label{disconnect}

\

\begin{enumerate}[(i)]
\item  Suppose $M\in {\cal A}$, and that $M$ is not $\Gamma$-connected. Then $M$ is isomorphic to $\bar{K}_n$ or 
$\bar{K}_n[D_3]$ for some $n\geq 1$.
\item Suppose $M \in {\cal S}$, and that $M$ is not $\Gamma$-connected. Then $M$ or $\bar{M}$ is isomorphic to one of:
$K_n[A]$ for some $A\in {\cal A}$ and $n>1$; $L$ or $L[D_3]$, for some $L \in {\cal L}$; or $F_6$.
\end{enumerate}
\end{lemma}

\begin{proof} (i) The collection of $\Gamma$-connected components forms a system of imprimitivity. Thus, 
by 2-set-homogeneity,
if $\alpha,\beta$ lie in the same $\Gamma$-connected component then $\alpha$ and $\beta$ cannot be unrelated, so the connected
components are
all tournaments. These components  are all set-homogeneous, and are all isomorphic. Thus, by Lemma~\ref{sethomtourn}, they 
are all isomorphic to $D_3$ or all of size 1.

(ii) Write $x\equiv y$ if $x,y$ lie in the same $\Gamma$-connected component of $M$. 
So $\equiv$ is a $G$-congruence.
By 2-set-homogeneity, either there are no independent pairs in the same $\equiv$-class, or there are no $\sim$-related
 pairs in the same
 $\equiv$-class; for at least one of these binary relations must hold of some pair from distinct $\equiv$-classes. There are two cases to consider. 

\

\noindent \textbf{Case 1:} \textit{There are no $\sim$-related pairs in an $\equiv$-class.}

\

\noindent Then each $\equiv$-class is a 
set-homogeneous a-digraph, so lies 
in ${\cal A}$, and all such are isomorphic, to $A$, say. If $A$ is not a tournament, then by 2-set-homogeneity applied to 
unrelated pairs,
any two vertices from distinct $\equiv$-classes are $\sim$-related, so $M \cong K_n[A]$ for some $n>1$. On 
the other hand, if $A$ is a tournament, then it is isomorphic to $D_1$ or $D_3$. If $A \cong D_1$, then $M$ has no $\Gamma$-arcs,
so is a homogeneous graph so lies in ${\cal L}$. 

Thus, we may  suppose $A \cong D_3$. 

Suppose first that there do {\em not} exist distinct $\equiv$-classes $B_1$ and $B_2$
such that $B_1 \times B_2$ contains both $\sim$-related and unrelated pairs. Then, for each pair of $\equiv$-classes, either
 all possible edges exist between the $\equiv$-classes, or none exist. In this case, there is an induced graph on $M/\equiv$, and it
 is easy to see that this is set-homogeneous. Thus, $M \cong  L[D_3]$ for some $ L \in {\cal L}$.

In the other case there are distinct $\equiv$-classes $B_1$ and $B_2$
such that $B_1 \times B_2$ contains both $\sim$-related and unrelated pairs. 
Then by 2-set-homogeneity, all such pairs of distinct $\equiv$-classes have this form.
We claim that in this case, either $M$ or $\bar{M}$ is isomorphic to $F_6$.

Put $B_i=\{x_i,y_i,z_i\}$ with $x_i\rightarrow y_i\rightarrow z_i\rightarrow x_i$ for $i=1,2$.

\

\noindent {\em Claim.} Each vertex of $B_1$ is $\sim$-related to a vertex of $B_2$, and is unrelated to a vertex in $B_2$.

\

\begin{proof}[Proof of Claim.]
Suppose  that some vertex of $B_1$, say $x_1$, is $\sim$-related to all members of $B_2$; the other case, when some vertex of $B_1$ is  unrelated to every vertex of $B_2$, is handled similarly. Now if say
$y_1$ is $\sim$-related to an element of $B_2$, without loss of generality $x_2$, then there is $g\in G$ with  
$(x_1,y_1,x_2)^g=(x_2,y_2,x_1)$. In particular $(x_1,x_2)^g=(x_2,x_1)$, 
so $\Delta$ is self-paired. Hence there is $h\in G$ with $(x_1,x_2)^h=(y_1,x_2)$, so 
$\Delta(y_1)=\Delta(x_1)^h$ contains $B_2^h=B_2$. Also, $y_1^h=z_1$, so $\Delta(z_1)$ contains $B_2$.
Thus, all pairs from $B_1 \times B_2$ are $\sim$-related, 
contradicting our assumption. Thus, $y_1$, and likewise $z_1$, are unrelated to all vertices of $B_2$.

By 3-set-homogeneity, there is $g\in G$ with $(y_1,z_1,x_2)^g=(x_2,y_2,y_1)$. In particular $(y_1,x_2)^g=(x_2,y_1)$, so the $G$-orbit $\Lambda$ is self-paired. Thus $\Lambda(y_2)=\Lambda(z_1)^g$ contains $B_2^g=B_1$, contradicting our assumption that $x_1$ and $y_2$ are $\sim$-related.
\end{proof}

Thus, by the claim,  we may suppose $x_1$ is $\sim$-related to some but not all members of $B_2$, and we suppose that
 it is $\sim$-related just to $x_2$ (there is a similar argument, recovering $\bar{F_6}$,  when $x_1$ is $\sim$-related to
 two vertices of $B_2$).  By the claim, 
$x_2$ is unrelated to some vertex of $B_1$, say to $y_1$; the case where $x_2$ is unrelated to $z_1$ is similar -- note that the s-digraphs $F_6$ and $\overline{F_6}$ which we are recovering are isomorphic to their weak complements in the sense of Remark~\ref{weakcomp}.
 Now by set-homogeneity there is  $g\in G$ with $(x_1,x_2,y_1)^g=(x_2,x_1,y_2)$.  In particular,  $(x_1,x_2)^g=(x_2,x_1)$, so $\Delta$ is self-paired; that is, $G$ is transitive on ordered $\sim$-edges. Now let $\{i,j\}=\{1,2\}$. Since by the claim each vertex of $B_i$ is $\sim$-related to some vertex of $B_j$, and since $x_1$ is $\sim$-related to a unique member of $B_2$, it follows that each vertex of $B_i$ is $\sim$-related to a unique vertex of $B_j$. 

 If $y_1\sim z_2$, 
then there is $k\in G$ with $(x_1,z_2,x_2)^k=(z_2,x_1,y_1)$, so in this case the orbital $\Lambda$ on independent pairs
 is self-paired. 
Hence there is $h\in G$ with $(x_1,y_2)^h=(y_1,y_2)$, and we must have $y_1^h=z_1$
 which is impossible as $y_1\not\sim y_2$ and $z_1\sim y_2$.
Thus, $y_1\sim y_2$, and similarly, $z_1 \sim z_2$.
It follows that $B_1 \cup B_2$ carries a structure isomorphic to $F_6$, as indeed does the union of any two blocks.
 Also, $\Lambda$ is not self-paired, for 
if $g\in G$ with $(x_1,y_2)^g=(y_2,x_1)$, 
then also $y_1^g=z_2$ (as $y_1$ is the only vertex with $x_1\rightarrow y_1$), but $y_1\sim y_2$ and $z_2\not\sim x_1$.

Thus, to see that $M\cong F_6$, it suffices to show $M=B_1 \cup B_2$. So suppose that there is a third block
 $B_3=\{x_3,y_3,z_3\}$, with
$x_3\rightarrow y_3\rightarrow z_3 \rightarrow x_3$
and $x_2\sim x_3$, $y_2\sim y_3$, $z_2\sim z_3$. 

We first eliminate the case when 
$x_1\sim x_3$, so suppose this holds. Then also  $y_1\sim y_3$ and 
$z_1\sim z_3$. Thus, there is $g\in G$ with $\{x_1,y_2,y_3\}^g=\{z_1,y_2,y_3\}$. Then $g$ fixes $B_1$ setwise and $y_1^g=y_1$ as $y_1$ is the unique element of
 $B_1$ $\sim$-related to $y_2$ and $y_3$, 
but $x_1\rightarrow y_1$ and $y_1\rightarrow z_1$, contradicting $x_1^g=z_1$. 

Thus,  $x_1\sim z_3$ or $x_1 \sim y_3$.
If $x_1 \sim z_3$ then since $B_1 \cup B_3$ induces a copy of $F_6$ it follows that $z_1 \sim y_3$ and $y_1 \sim x_3$. 
The resulting structure may be viewed as a $9$-gon with the vertices ordered $x_1$, $x_2$, $x_3$, $y_1$, $y_2$, $y_3$, 
$z_1$, $z_2$, $z_3$ and with the three directed triangles inscribed.  Let $Y_9$ denote this $9$ vertex digraph. Viewed in
 this way, it is easy to describe the automorphisms of the graph $Y_9$. 
On the other hand, if $x_1 \sim y_3$ then  since $B_1 \cup B_3$ induces a copy of $F_6$ it follows that 
$z_1 \sim x_3$ and $y_1 \sim z_3$. Again, the three $\equiv$-classes induce a copy of $Y_9$. 

Thus, we may suppose that there 
are at least three $\equiv$-classes, and the induced digraph on the union of any three is isomorphic to $Y_9$.
 In particular
 this means that $M$ does not embed any $\Delta$-triangles (copies of $K_3$). Now fix $x_1 \in M$ and consider $\Delta(x_1)$. Since $M$ does not 
embed any $\Delta$-triangles, $\Delta(x_1)$ is an independent set. It follows that  $G_{x_1}$ is highly homogeneous 
on $\Delta(x_1)$ so
 since  $\Lambda$ is not self-paired the only possibility is $|\Delta(x_1)|=3$. So we just need to consider the case where there
 are exactly four $\equiv$-classes, $B_1$, $B_2$, $B_3$ and $B_4$. Let $x_1 \in B_1$ and put $\Delta(x_1) = \{ \alpha, \beta, \gamma \}$. By
 set-homogeneity there is an automorphism $g \in G_{x_1}$ with $\{ \alpha, \beta \}^g = \{ \beta, \gamma \}$, and since $\Lambda$ 
is not self paired this forces $(\alpha, \beta, \gamma)^g = (\beta, \gamma, \alpha)$.  But then $g$ restricted to 
$B_2 \cup B_3 \cup B_4$ is an automorphism of $Y_9$ that cyclically permutes a copy of $\bar{K_3}$. 
This is a contradiction, since by inspection of $Y_9$ it 
is clear that no such automorphism exists. We conclude that there are just two blocks and that $M\cong F_6$ (or $M\cong \overline{F_6}$ in the other case mentioned after the claim).

\

\noindent \textbf{Case 2:} \textit{There are no independent pairs in the same $\equiv$-class.}

\

\noindent In this case, the argument of Case~1 applies, with 
the roles of 
 independent pairs and $\sim$-connected pairs reversed. We find that $\bar{M} \cong K_n[A]$ for some $A \in {\cal A}$, or 
 again $M$ is isomorphic to $L$ or to $L[D_3]$ for some $L \in {\cal L}$, or to $F_6$ or $\overline{F_6}$. \end{proof}

\begin{lemma} \label{samenbours}

\

\begin{enumerate}[(i)]
\item Let $M\in {\cal A}$, and suppose that $M$ is $\Gamma$-connected and there are distinct $\alpha,\beta\in M$ with $\Gamma(\alpha)=\Gamma(\beta)$. Then
$M$ is isomorphic to  $D_3[\bar{K}_n]$ for some $n\geq 2$.
\item Let $M \in {\cal S}$, and suppose that $M$ is $\Gamma$-connected and there are distinct $\alpha,\beta\in M$ with $\Gamma(\alpha)=\Gamma(\beta)$. 
Then $M$ or $\bar{M}$ is isomorphic to one of: $L$ or $D_3[L]$ for some $L \in {\cal L}$, or $A[K_n]$ for some 
$A\in {\cal A}$ and $n>1$.
\end{enumerate}
\end{lemma}

\begin{proof} (i) By assumption, $|M|>1$. Since $M$ is $\Gamma$-connected, $M$ is not an independent set. Define $x\equiv y$ to hold if and only if $\Gamma(x)=\Gamma(y)$. Then $\equiv$ is a $G$-congruence. Since $x\to y$ implies $x\not\equiv y$, 
there is more than one class. The classes are isomorphic, of size greater than one. By set-homogeneity applied to arcs,
the $\equiv$-classes are all isomorphic to $\bar{K}_n$ for some fixed $n>1$. Also by
 2-set-homogeneity applied to independent pairs, if $x\not\equiv y$ then
$x\rightarrow y$ or $y\rightarrow x$ (recall that a $\equiv$-class contains no $\Gamma$-arcs). 

Suppose $B_1,B_2$ are distinct $\equiv$-classes, $\alpha_1\in B_1$ and $\alpha_2\in B_2$, and $\alpha_1 \rightarrow \alpha_2$. Then, as $B_1$ is a $\equiv$-class,
$x\rightarrow \alpha_2$ for all $x\in B_1$. It follows that $x\rightarrow y$ for all $x\in B_1$ and $y\in B_2$; for as $y\in B_2$ we have $y\equiv \alpha_2$, and as  $x\to \alpha_2$ holds,
$y \to x$ cannot hold. Thus,
 there is a tournament structure on $\{B_1^g:g\in G\}$, with $B_1^g\rightarrow B_1^h$ if and only if there is $x\in B_1^g$ and 
$ y\in B_1^h$ with $x\rightarrow y$. This tournament is clearly set-homogeneous, so by Lemma~\ref{sethomtourn} must be isomorphic to $D_3$ (since $M$ is $\Gamma$-connected). The conclusion follows.

(ii) We argue as in (i), with the same congruence $\equiv$.  The structure induced on each $\equiv$-class is a graph with respect to $\sim$, and is set-homogeneous so belongs to ${\cal L}$. If this graph is neither complete nor independent, then no pair from $B_1 \times B_2$ is $\sim$-related or unrelated, and as in (i) we find that
$M$ is isomorphic to $L$ or $D_3[L]$ for some $L \in {\cal L}$. Suppose next that each $\equiv$-class induces 
$K_n$. We claim that in this case, for any two distinct $\equiv$-classes $B_1$ and $B_2$,  we have one of: all possible arcs from $B_1$ to $B_2$, or all possible arcs from $B_2$ to $B_1$, or no arcs between $B_1$ and $B_2$. From this it will follow that there is an induced member of ${\cal A}$ on $M/\equiv$, so $M \cong A[K_n]$ for some $A\in {\cal A}$ and $n>1$.
 Likewise, 
if each $\equiv$-class is an independent set then $\overline{M} \cong A[{K_n}]$ for some $A\in {\cal A}$ and $n>1$.

To prove the claim, suppose  there are distinct $\equiv$-classes $B_1$, $B_2$, and $x_1\in B_1$, 
$x_2\in B_2$ with
$x_1\rightarrow x_2$. We cannot have $y_1\in B_1$, $y_2\in B_2$ with $y_2\rightarrow y_1$: indeed, otherwise, by definition of $\equiv$, $z\rightarrow x_2$ for all $z\in B_1$ and
$w\rightarrow y_1$ for all $w\in B_2$, so $y_1\rightarrow x_2$ and $x_2\rightarrow y_1$, which is impossible. Thus, seeking a contradiction we may suppose there are unrelated 
$y_1\in B_1$ and $y_2\in B_2$.
Since $B_1$ is a $\equiv$-class, $x\rightarrow x_2$ for all $x\in B_1$; in fact, 
there is a subset $C_2$ of $B_2$ with $x_2\in C_2$, $y_2\not\in C_2$,
such that for any $x\in B_1$, $x\rightarrow y$ for all $y\in C_2$ and $x$ is unrelated to all elements of $B_2\setminus C_2$. 
Since the pair of $\equiv$-classes $(B_1,B_2)$ cannot be interchanged, there are two distinct paired orbitals $\Lambda, \Lambda^*$ on unrelated pairs.

Now $|B_2\setminus C_2|=1$; for if there is $z_2\in B_2\setminus C_2$ with $z_2\neq y_2$ then by set-homogeneity, for any 
 $x\in B_1\setminus \{y_1\}$
there is $g\in G$ with
$\{x, y_2,z_2\}^g=\{x,y_1,y_2\}$, and this $g$ interchanges $B_1$ and $B_2$, which is a contradiction. 

Pick $h\in G$ with $y_2^h=x_2$ and put $B_3:=B_1^h$, so $B_3 \neq B_1,B_2$. 
Then $z\to y_2$ and $\{z,x_2\}$ are unrelated for all $z\in B_3$.
We can write the relationship between $B_1$ and $B_2$ as $B_1\Rightarrow
B_2$.  So $B_1\Rightarrow B_2$, and $B_3\Rightarrow B_2$, and either
$B_1\Rightarrow B_3$ or $B_3 \Rightarrow B_1$. In the first case, by 2-set-homogeneity there is
 $k\in G$ such that $B_1^k=B_1$ and $B_2^k=B_3$, and in the second case there is $k$ fixing $B_3$ setwise with $B_1^k=B_2$. We assume the former, the second case being similar.

Now put $C_3:=C_2^k$, and let 
$B_3\setminus C_3:=\{y_3\}$. We have $x_1\to x_2$, $y_3 \to y_2$, and the pairs $\{x_1,y_2\}, \{x_1,y_3\}$ and $\{x_2,y_3\}$ are unrelated.
Thus, the map $(x_1,y_3,x_2)\mapsto (y_3,x_1,y_2)$ is an isomorphism; this extends by 
set-homogeneity to an automorphism
(as the structure on $\{x_1,y_3,x_2\}$ is rigid), and the latter interchanges the pair $\{x_1,y_3\}$, contrary to 
the fact that these are unrelated. 
This is a contradiction and so establishes the claim in the first paragraph of the proof of part (ii). 
\end{proof}

Of course, the dual of the above lemma also holds, and deals with the case that there are distinct vertices $\alpha, \beta \in M$ with $\Gamma^*(\alpha) = \Gamma^*(\beta)$. 

\begin{lemma}\label{nhood}
Let $M \in {\cal S}$, and $\alpha \in M$. Then the induced s-digraphs on $\Gamma(\alpha)$ and
 $\Gamma^*(\alpha)$ are set-homogeneous, so lie in
${\cal S}$.
\end{lemma}

\begin{proof} Suppose that $U,V$ are induced subdigraphs of $\Gamma(\alpha)$, and $\phi:U\rightarrow V$ is an isomorphism. Then 
$\phi$ extends to an isomorphism
$U \cup \{\alpha\}\rightarrow V\cup \{\alpha\}$ fixing $\alpha$. Thus, by set-homogeneity of $M$, there is $g\in G$ 
with
$(U\cup\{\alpha\})^g=V\cup\{\alpha\}$. Since $\alpha$ is the unique element of $U \cup \{\alpha\}$ which dominates the rest 
of this set, $g$ must fix $\alpha$, so induces an automorphism of $\Gamma(\alpha)$.

The same argument applies to $\Gamma^*(\alpha)$.
\end{proof}

\begin{lemma} \label{isom}
Let $M\in {\cal A}$, $\alpha \in M$, and suppose that the induced a-digraphs on $\Gamma(\alpha)$ and $\Gamma^*(\alpha)$ are isomorphic. Then
$M$ is isomorphic to one of the a-digraphs listed in Theorem~\ref{maintheorem}(ii). 
\end{lemma}
\begin{proof} By Lemma~\ref{disconnect}, we may assume that $M$ is $\Gamma$-connected. 
If $|\Gamma(\alpha)|=1$, then $M$ has in and out degree 1, so by connectedness is isomorphic to $D_n$ for some $n$.
In this case, by set-homogeneity on unrelated pairs, $n\leq 5$. Thus, we may suppose $|\Gamma(\alpha)|>1$. 

By set-homogeneity, there is $g\in G$ with $(\Gamma(\alpha))^g=\Gamma^*(\alpha)$. Put
 $\gamma:=\alpha^g$
and $\beta:=\alpha^{g^{-1}}$. Then $\Gamma^*(\alpha)=\Gamma(\gamma)$ and $\Gamma(\alpha)=\Gamma^*(\beta)$. 

We may  assume  that for any vertices $x$ and $y$, if $\Gamma(x)=\Gamma(y)$ or $\Gamma^*(x)=\Gamma^*(y)$ then $x=y$; for otherwise our result follows by Lemma~\ref{samenbours} and the remark following it. As $G_\alpha$ fixes $\Gamma(\alpha)$ and $\Gamma^*(\alpha)$, it follows that $G_\alpha$ fixes $\beta$ and $\gamma$, and hence that $G_\alpha=G_\beta=G_\gamma$.

Suppose first that $\gamma\in \Gamma(\alpha)$.  Now $\beta = \alpha^{g^{-1}} \not\in
\Gamma(\alpha)^{g^{-1}}= \Gamma(\alpha)$, so $\gamma\neq \beta$. 
Since $G_\alpha$ acts transitively on $\Gamma(\alpha)$ and fixes $\Gamma^*(\alpha)$ setwise, 
$\Gamma(\gamma)=\Gamma( \gamma')$ for all $\gamma'\in \Gamma(\alpha)$. 
Since by assumption, $|\Gamma(\alpha)|>1$, this contradicts our assumption in the last paragraph. Thus $\gamma\not\in\Gamma(\alpha)\cup \Gamma^*(\alpha)$ so $\alpha$ and $\gamma$ are unrelated.

Next, suppose that $\beta=\gamma$. If $x\in \Gamma(\alpha)$ and $y\in \Gamma^*(\alpha)$ then by set-homogeneity
there is $g\in G$ with $(\alpha,x,\gamma)^g=(\gamma,y,\alpha)$. Thus, the ordered pairs of unrelated points form a single self-paired orbital.

We claim now that
$M=\{\alpha,\gamma\}\cup \Gamma(\alpha) \cup \Gamma^*(\alpha)$. For suppose $x\in \Gamma(\alpha)$, and $y\in \Gamma(x)$, with $y\not\in \Gamma(\alpha)\cup\Gamma^*(\alpha)$. Then $\alpha$ and $y$ are unrelated, so by 3-set-homogeneity there is $g\in G$ with $(\alpha,x,y)^g=(\alpha,x,\gamma)$. As $G_\alpha$ fixes $\gamma$, it follows that $y=\gamma$, as required.

It follows 
that there is a $G$-congruence
$\equiv$ on $M$ with classes of size 2, where $x\equiv y$ if and only if 
$\Gamma(x)\cup \Gamma^*(x)=\Gamma(y)\cup \Gamma^*(y)$.
 By 2-set-homogeneity, any two vertices in  distinct $\equiv$-classes
 are joined by an arc. Clearly if $x\equiv y$ and $x\neq y$ then $\Gamma(x)=\Gamma^*(y)$, so the digraph
 induced on the union of 
any two blocks is isomorphic to $D_4$. By set-homogeneity applied to arcs, $G$ acts 2-transitively on the 
set of
 $\equiv$-classes. Thus,
$\Gamma(\alpha)$ (and likewise $\Gamma^*(\alpha)$) meets each $\equiv$-class other than $\{\alpha,\gamma\}$ in 
a singleton, 
so carries the structure
of a set-homogeneous tournament. Since we have $|\Gamma(\alpha)|>1$, it follows from Lemma~\ref{sethomtourn} that  each of $\Gamma(\alpha)$
and $\Gamma^*(\alpha)$ is isomorphic to $D_3$, so $|M|=8$. 
Once the arcs within $\Gamma(\alpha)$ are chosen, the rest of the construction is forced, and it can be checked that whatever the arcs are on $\Gamma(\alpha)$ the resulting digraph  is isomorphic to $H_0$.

Finally, suppose that $\beta\neq \gamma$ and $\gamma\not\in \Gamma(\alpha)$. Then, as in the fourth paragraph
 of 
the proof, we may assume $\beta \not\in \Gamma^*(\alpha)$. Thus, $\{\alpha,\beta\}$ and $\{\alpha,\gamma\}$ are both 
unrelated pairs. There is no $x\in \Gamma^*(\alpha)$ with $\beta \rightarrow x$, for otherwise
there would be $g\in G$ with $(\alpha,x,\beta)^g=(\alpha,x,\gamma)$, which is impossible as $G_\alpha$ fixes $\beta$. Thus,
 $d(\beta,\alpha)>2$ and likewise $d(\alpha,\gamma)>2$, so the orbital which contains $(\alpha,\beta)$ and 
$(\gamma,\alpha)$ is not
self-paired.  Define $x\equiv y$ on $M$ if $G_x=G_y$; then $\beta\equiv \alpha$ and $\gamma \equiv \alpha$,
 as noted above. 
If $\alpha \equiv x$ then clearly $x\not\in \Gamma(\alpha) \cup \Gamma^*(\alpha)$, so if also $x\neq \alpha$
 then $\{\alpha,x\}$
 is an unrelated pair.  In this case,
$(\alpha,x)$ lies in the same orbit as one of $(\alpha,\beta)$, $(\alpha,\gamma)$, and since $G_\alpha$ fixes $x$ this implies that
$x\in\{\beta,\gamma\}$. 
Thus $\equiv$-classes have size three, and must be independent sets. 
Moreover, two elements in distinct $\equiv$-classes cannot be unrelated. Hence it follows
from set-homogeneity on arcs that $G$ is
 2-homogeneous on 
$M/\equiv$.
If $x\in \Gamma(\alpha)$, then there is an arc from $\alpha$ to $x$ and an arc from $x$ to $\beta$. Since $\alpha\equiv \beta$, it follows by 2-homogeneity on the set of $\equiv$-classes that there are arcs in both directions between any two classes. 
  Thus, if $B_1$ and $B_2$ are $\equiv$-classes then  there is $g\in G$ interchanging $B_1$ and $B_2$. 
However, the number of arcs between $B_1$ and $B_2$ is 9, which is odd, 
so the number of arcs from $B_1$ to $B_2$ does not equal the number from $B_2$ to $B_1$, a contradiction to 
the existence of $g$.
\end{proof}

\begin{lemma}\cite[Lemma 1.2(4)]{lachlan} \label{2verts}
Suppose that $M \in {\cal S}$,  $\alpha \in M$, $\beta\in \Gamma(\alpha)$ and $\gamma \in \Gamma^*(\alpha)$. Then
$\Gamma(\alpha) \cap \Gamma^*(\beta) \cong \Gamma(\gamma) \cap \Gamma^*(\alpha) $.
\end{lemma}

\begin{proof} This follows from 2-set-homogeneity, which ensures that there is $g\in G$ with $(\alpha,\beta)^g=(\gamma,\alpha)$.
\end{proof}

We are now in a position to prove the first two parts of our main result. 

\subsection*{Proof of Theorem~\ref{maintheorem}(i), (ii).} For part (i), by \cite{enomoto} and \cite{Ronse1978}, if $M \in {\cal L}$ then $M$ is homogeneous, and by \cite{gardiner} the listed graphs are 
precisely the finite homogeneous graphs.

For part (ii), first, by Lemma~\ref{disconnect} (i), we may suppose that $M$ is connected. We now argue by induction on $|M|$. By
Lemma~\ref{nhood} and the inductive hypothesis, $\Gamma(\alpha)$ and $\Gamma^*(\alpha)$ belong to the list in (ii), 
and as $\Gamma$ and $\Gamma^*$ are paired orbitals,
$|\Gamma(\alpha)|=|\Gamma^*(\alpha)|$. 

By Lemma~\ref{isom}, we may suppose that $\Gamma(\alpha)\not\cong \Gamma^*(\alpha)$. Also, if one of $\Gamma(\alpha)$ 
or $\Gamma^*(\alpha)$ has no arc,
then $M$ does not embed $P_3$, and hence both are isomorphic to $\bar{K}_n$ for some $n$, so are isomorphic, a contradiction. 
Thus, by inspection of the list in (ii),
we may suppose that for some $n>1$,  $\Gamma(\alpha)\cong \bar{K}_n[D_3]$ and $\Gamma^*(\alpha)\cong D_3[\bar{K}_n]$, or vice 
versa. Since the first case is obtained from the second by reversing all arcs (and the list of examples  in Theorem~\ref{maintheorem}(ii) is closed under this operation), we suppose
$\Gamma(\alpha)\cong \bar{K}_n[D_3]$ and $\Gamma^*(\alpha)\cong D_3[\bar{K}_n]$.
Now, however, we find that if $\beta\in \Gamma(\alpha)$ and $\gamma \in \Gamma^*(\alpha)$ then
$|\Gamma(\alpha)\cap \Gamma^*(\beta)|=1$ but $|\Gamma^*(\alpha)\cap \Gamma(\gamma)|=n>1$. This contradicts Lemma~\ref{2verts}. \qed

\section{Proof of Theorem~\ref{maintheorem}(iii): case $\Gamma(\alpha)\cong\Gamma^*(\alpha)$}
\label{sec_CaseInNEqualsOutN}

Let $M$ be a finite set-homogeneous s-digraph, with $G:=\Aut(M)$.
In proving Theorem~\ref{maintheorem}(iii), it follows from 
Lemma~\ref{disconnect} that we may assume that $M$ is $\Gamma$-connected.
If $|\Gamma(\alpha)|=1$, then the $\Gamma$-structure is just a cycle, and it can be checked using  set-homogeneity that $M$ is up to complementation one of $D_1$, $D_3$, $D_4$, $D_5$, $E_6$, $E_7$. In this section we deal with the case where $\Gamma(\alpha)\cong
\Gamma^*(\alpha)$. We argue by induction on $|M|$, and note that the result certainly holds when 
$|M|\leq2$. By Lemma~\ref{nhood}, and the inductive 
hypothesis, we therefore have that $\Gamma(\alpha)$ and $\Gamma^*(\alpha)$ lie in 
the list of $s$-digraphs in Theorem~\ref{maintheorem}(iii). 
Thus,  our assumption throughout the section is:

\begin{center}
\begin{tabular}{ll}
(A)&
\begin{tabular}{l}
$M$ is a finite $\Gamma$-connected 
set-homogeneous s-digraph with $\Gamma$-connected complement;\\ 
$\Gamma(\alpha)\cong\Gamma^*(\alpha)$ for some  $\alpha \in M$, and 
$|\Gamma(\alpha)|>1$;\\
$\Gamma(\alpha)$ and $\Gamma^*(\alpha)$ belong to the list in Theorem~\ref{maintheorem}(iii);\quad
and the following statement $(*)$,\\ justified by Lemma~\ref{samenbours}(ii)
and the remark following Lemma~\ref{samenbours}:\\
$(*)$~~for any distinct $\alpha,\beta$, $\Gamma(\alpha)\neq \Gamma(\beta)$ and $\Gamma^*(\alpha)\neq \Gamma^*(\beta)$.\\
\end{tabular}
\end{tabular}
 
\end{center}

Observe that the complement of any s-digraph satisfying (A) also satisfies (A).

Our first lemma uses Lemma~\ref{samenbours} to
identify two major sub-cases to be considered.

\begin{lemma}\label{isom2}
Suppose $M\in {\cal S}$ satisfies (A). Then
 there exist unique vertices $\beta,\gamma$ such that $\Gamma^*(\beta) = 
\Gamma(\alpha)$, and $\Gamma^*(\alpha)=\Gamma(\gamma)$. Moreover, after replacing $M$ by its complement if necessary, one of the following holds. 
\begin{enumerate}
\item[$(a)$] \label{c5}
$\beta\ne\gamma$ and there is an equivalence relation $\equiv$ on $M$ given by
\[
u \equiv v \Leftrightarrow G_u = G_v 
\] 
such that
		the induced structure on each $\equiv$-class is $C_j$ for some $j \in \{3,4,5 \}$, with the corresponding cyclic group induced by the setwise stabiliser;
		
\item[$(b)$]
$\beta=\gamma$ and there is a $G$-congruence $\equiv$ on $M$ given by
\[
u \equiv v \Leftrightarrow \left({\left(\Gamma(u) = \Gamma^*(v)\right) \wedge 
\left(\Gamma^*(u) = \Gamma(v)\right) \vee u=v}\right)
\]
with $\equiv$-classes of size two such that:  $u\equiv v \Leftrightarrow u\sim v$ and the restriction of
the relation $\Delta$ to $\Gamma(\alpha)\cup\Gamma^*(\alpha)$ is a 
matching between $\Gamma(\alpha)$ and $\Gamma^*(\alpha)$; moreover
$\Gamma(\alpha)$ is isomorphic to one of the $a$-digraphs of 
Theorem~{\rm\ref{maintheorem}(ii)}, and the 
$s$-digraph induced on the union of any two $\equiv$-classes is isomorphic to either 
$\overline{D_4}$ or $\overline{K_{2,2}}$.

\end{enumerate}
\end{lemma}

\begin{proof} 
By set-homogeneity, and the assumption of the lemma, there exists 
$g\in G$ such that $\Gamma(\alpha)^g=\Gamma^*(\alpha)$, so
for $\beta:=\alpha^{g^{-1}}$ and $\gamma:=\alpha^g$ we have $\Gamma^*(\beta) 
= \Gamma(\alpha)$ and $\Gamma^*(\alpha)=\Gamma(\gamma)$. 
By assumption (A),
 the vertices $\beta,\gamma$ are unique. 
There are then two possibilities, depending on whether or not $\beta = \gamma$. 

First suppose $\beta \neq \gamma$. Define an 
equivalence relation $\equiv$ on $M$ by $x\equiv y$ if and only 
if $G_x=G_y$. 
In this case the vertices $\alpha,\beta,\gamma$ 
are $\equiv$-equivalent. Since $|\Gamma(\alpha)|>1$, there are distinct $\delta,\delta'\in \Gamma^*(\alpha)$ and $g\in G_\alpha$ with $\delta^g=\delta'$, so $\alpha\not\equiv\delta$.
Thus, 
 the induced structure on each $\equiv$-class is an undirected graph 
and its setwise stabiliser in $G$ 
induces on it a set-homogeneous group of automorphisms that is 
regular on vertices, by definition of $\equiv$. It follows from Theorem~\ref{maintheorem}(i) 
that the only such graphs are, up to 
complementation, one of: 
$K_3\cong C_3$, $K_{2,2}\cong C_4$, or the pentagon $C_5$. Moreover, in each case, the
only set-homogeneous group that is regular on vertices is a cyclic group. Thus, (a) holds.

Now suppose that
$\beta=\gamma$ (so that $g$ interchanges $\alpha$ and $\beta$). 
Because of assumption (A) the relation given in part (b) is a 
$G$-congruence $\equiv$ on $M$  and the $\equiv$-class containing $\alpha$ is
$[\alpha]:=\{\alpha,\beta\}$.
All $\equiv$-classes have size $2$, and  we cannot have a 
$\Gamma$-arc between $\equiv$-related elements. 
By 2-set-homogeneity and since $\equiv$ is 
a $G$-congruence, either  `$u \equiv v \Leftrightarrow (u=v \vee u \sim v)$', or
`$u \equiv v \Leftrightarrow (u=v \vee u$ and $v$ are unrelated)'. In the latter case, 
we may replace $M$ by $\bar{M}$. Thus we may assume that
$u \equiv v \Leftrightarrow (u=v \vee u \sim v)$.

Let $a\in\Gamma(\alpha)$. By 2-set-homogeneity there is an element 
$h\in G$ such that $(\alpha,a)^h=(a,\beta)$, and then $b:=\beta^h$ must 
lie in $\Gamma(\beta)=\Gamma^*(\alpha)$ and the $\equiv$-class 
$[a]=[\alpha]^h=\{a,b\}$.
This yields the claimed $\Delta$-matching between $\Gamma(\alpha)$ and 
$\Gamma^*(\alpha)$, and that $[\alpha]\cup [a] \cong \overline{D_4}$. It also implies that $\Gamma(\alpha)$ does not
contain any $\equiv$-classes and hence the induced digraph on
$\Gamma(\alpha)$ does not contain undirected edges. Thus by induction
$\Gamma(\alpha)$  is isomorphic to one of the $a$-digraphs of 
Theorem~{\rm\ref{maintheorem}(ii)}. 
Finally, to determine the possibilities for
the induced $s$-digraph on
the union of two $\equiv$-classes we may assume that one of the classes is 
$\{\alpha,\beta\}$. If the other class is $[a]$ for some $a\in\Gamma(\alpha)$
then the induced $s$-digraph is $\overline{D_4}$, and otherwise it is
$\overline{K_{2,2}}$.
\end{proof}

It follows from Lemma~\ref{isom2} that, for the case where $\Gamma(\alpha)
\cong\Gamma^*(\alpha)$,  it is sufficient to prove 
Theorem~\ref{maintheorem}(iii) in the case where 
part (a) holds, and the case where part (b) holds.  
From now on we shall use $[\alpha]$ to denote the $\equiv$-class of the vertex $\alpha$, with $\equiv$ as in Lemma~\ref{isom2}(a) or (b).

\subsection{Dealing with Lemma~\ref{isom2}(b)}

Throughout this subsection $M$ will denote an s-digraph satisfying assumption (A) and the conditions of 
Lemma~\ref{isom2}(b). 
In what follows we prove that the pair $(\Gamma(\alpha),M)$ is isomorphic
to  $(D_3, \overline{H_0})$,  $(D_4, H_2)$, or 
$(\overline{K_2}, H_1)$, completing the proof of Theorem~\ref{maintheorem}(iii)
in this case. We use the notation of Lemma~\ref{isom2}.

\begin{lemma}\label{d5}
$\Gamma(\alpha)$ is not isomorphic to $D_5$.
\end{lemma}

\begin{proof}
Suppose that $\Gamma(\alpha) \cong D_5$. Let $\Gamma(\alpha) = 
\{a_1, a_2, a_3, a_4, a_5 \}$ with 
$a_i \rightarrow a_{i+1 \pmod{5}}$ for all $i$. Let $b_i$ be the 
corresponding elements of $\Gamma(\beta)$ under the matching
given by $\equiv$.  Since $a_5\rightarrow a_1$, it follows from
Lemma~\ref{isom2}(b) that the subdigraph induced on $[a_1]\cup[a_5]$ is
isomorphic to $\overline{D_4}$, and in particular $a_1\rightarrow b_5 \rightarrow b_1 \rightarrow a_5$.
Likewise $b_1\to b_2 \to b_3\to b_4 \to b_5$.
Similarly $b_5\rightarrow a_4\rightarrow b_3$, and the 
subdigraphs induced on $\{a_1,a_2,a_3,a_4\}$ and $B:=\{a_1,b_5,a_4,b_3\}$
are both isomorphic to a directed path of length 3 (which is rigid).
Thus there is a unique isomorphism mapping
$
( a_1, a_2, a_3, a_4 ) \mapsto 
( a_1, b_5, a_4, b_3 )
$
and by 4-set-homogeneity this extends to an automorphism $\phi$ of $M$.
Let $\alpha' = \alpha^\phi$ and $\beta' = \beta^\phi$ so that
$[\alpha']=\{\alpha',\beta'\}$. Then $\Gamma(\alpha')$ 
contains $B$, and this implies that $\alpha'\not\in\{\alpha,\beta\}\cup
\Gamma(\alpha)\cup\Gamma^*(\alpha)$. A similar consideration of 
$\Gamma^*(\beta')$ shows that $\beta'$ does not lie in this set either.

Now $[\alpha']\cup[a_1]=([\alpha]\cup[a_1])^\phi$ induces $\overline{D_4}$,
by Lemma~\ref{isom2}(b), and it follows that the subdigraph induced
on $\{\alpha\}\cup[\alpha']\cup[a_1]$ is rigid. A similar argument 
shows that this digraph is isomorphic to the digraph 
induced on $\{\alpha\}\cup[\alpha']\cup[a_5]$.
By 5-set-homogeneity there is 
an automorphism $\pi$ of $M$ extending the unique isomorphism
$
( \alpha, a_1, b_1, \alpha', \beta' ) \mapsto
( \alpha, a_5, b_5, \beta', \alpha' ). 
$
Since $\pi$ fixes $\alpha$ and maps $a_1$ to $a_5$ it follows that $\pi:a_i
\mapsto a_{i-1\pmod{5}}, \ b_i\mapsto b_{i-1\pmod{5}}$ 
for each $i$. Consequently 
$
\Gamma(\beta')=
\Gamma(\alpha')^\pi \supseteq \{ a_1, b_5, a_4, b_3 \}^\pi = \{ a_5, b_4, a_3, b_2 \}.
$
Now (considering the induced subdigraphs on $[\alpha']\cup[a_i]$ for 
$i=1,3,4,5$), 
$ 
\Gamma(\beta') \supseteq \{ b_1, a_5, b_4, a_3 \},
$
and we conclude that
$
\Gamma(\beta') = \Gamma(\alpha')^\pi = \{ a_5, b_4, a_3, b_2, b_1 \}.
$
This is a contradiction since the digraph induced by these vertices is not isomorphic to $D_5$. 
\end{proof}

\begin{lemma}\label{tournament}
If $\Gamma(\alpha)$ is a tournament,  or is not $\Gamma$-connected, then 
$(\Gamma(\alpha),M)\cong (\overline{K_2},H_1)$ or
$(\Gamma(\alpha),M)$ $\cong (D_3,\overline{H_0})$.
\end{lemma}

\begin{proof}
First suppose that $\Gamma(\alpha)$ is a tournament.
Then, as $|\Gamma(\alpha)|>1$, by Lemma~\ref{sethomtourn} we have 
$\Gamma(\alpha) \cong D_3$. 
By Lemma~\ref{isom2}(b) (last clause), it follows that each vertex of $\Gamma(\alpha)$ 
dominates one in $\Gamma^*(\alpha)$, and hence, since $|\Gamma(\alpha)|=3$, 
no arc has just one vertex in
$
 \{ \alpha, \beta \} \cup \Gamma(\alpha) \cup \Gamma(\beta),
$
so by $\Gamma$-connectedness of $M$ this set equals $M$.
This implies that $M$ has no unrelated pairs, and so $M$ must be the complement 
of a set-homogeneous a-digraph. By Theorem~\ref{maintheorem}(ii), the 
only possibility is that $M=\overline{H_0}$
(which satisfies the conditions).

Thus, we may suppose that $\Gamma(\alpha)$ is not $\Gamma$-connected. Then, by 
Theorem~\ref{maintheorem}(ii), $\Gamma(\alpha)\cong \overline{K_n}$ or 
$\overline{K_n}[D_3]$ for some $n>1$. 
In the latter case, if $A:=(a\rightarrow a'\rightarrow a''\rightarrow a)$ 
is a connected component of
$\Gamma(\alpha)$, and $b,b',b''\in\Gamma^*(\alpha)$ are such that $a\equiv b$,
$a'\equiv b'$, and $a''\equiv b''$, then it follows from 
Lemma~\ref{isom2}(b) that $B:=(b\rightarrow b'\rightarrow b''\rightarrow b)$
is a connected component of $\Gamma^*(\alpha)$. Thus, by Lemma~\ref{isom2}(b), $\Gamma(\alpha)\cong
A\cup(\Gamma^*(\alpha)\setminus B)$. If $\Gamma(\alpha)\cong \overline{K_n}$,
let $A=\{a\}\subseteq\Gamma(\alpha)$ and $B=\{b\}\subseteq\Gamma^*(\alpha)$ 
with $a\equiv b$. Then in either case, by set-homogeneity,  there is an 
automorphism $\theta$ of $M$ such that $\Gamma(\alpha)^\theta=
A\cup(\Gamma^*(\alpha)\setminus B)$.
 Let $\alpha' = \alpha^\theta$ and $\beta' = \beta^\theta$. Then $[\alpha']=
\{\alpha', \beta' \}\ne [\alpha]$. Also, $[\alpha'] \cap (\Gamma(\alpha) \cup \Gamma^*(\alpha))=\varnothing$, so
$[\alpha]\cup [\alpha']$ carries a copy of $\overline{K_{2,2}}$. 
Note that since $\theta$ preserves $\equiv$, $\Gamma(\alpha) \cap \Gamma(\beta')=\Gamma(\alpha) \setminus A$.

Let $a_2\in\Gamma(\alpha)\setminus A$ and $b_2\in\Gamma^*(\alpha)\setminus B$ such that $a_2\equiv b_2$. 
Then it can be checked that
$
( \alpha\rightarrow a\sim b\rightarrow \alpha'\sim \beta' ) \cong 
( \alpha\rightarrow a_2\sim b_2\rightarrow \beta'\sim \alpha' ),
$
and this configuration is rigid.
Thus  there is an automorphism of $M$ fixing $\alpha$ and 
swapping $\alpha'$ and $\beta'$, and in particular, $\Gamma(\alpha) \cap 
\Gamma(\alpha')\cong\Gamma(\alpha)\cap\Gamma(\beta')$. This implies that
$n=2$ so $\Gamma(\alpha)$ is isomorphic either to $\overline{K_2}$ or 
to $\overline{K_2}[D_3]$. 
If  $\Gamma(\alpha) \cong \overline{K_2}[D_3]$, then $\Gamma(a)$ contains 
$\{\beta, \beta', a', b''\}$. However,
the subdigraph induced by these four vertices is $\Gamma$-connected, whereas 
$\Gamma(a) \cong \overline{K_2}[D_3]$, and this is a contradiction. 

The only 
remaining possibility is that $\Gamma(\alpha) \cong \overline{K_2}$. 
In this case, let $\Gamma(\alpha) = \{a, a' \}$ and $\Gamma(\beta) = \{ b, b' \}$ with  $a\equiv b$ and $a'\equiv b'$. With $\theta$ defined as in the second paragraph of the proof, we have $\Gamma(\alpha') = \Gamma(\alpha)^\theta = \{ a,a' \}^\theta = \{ a,b' \}$ and since automorphisms preserve $\equiv$ we deduce that $\Gamma(\beta') = \Gamma(\beta)^\theta = \{ b,b'\}^\theta = \{ a', b \} $, that $\Gamma^*(\alpha') = \Gamma^*(\alpha)^\theta = \Gamma(\beta)^\theta = \{ a', b \} $ and similarly that $\Gamma^*(\beta') = \{ a,b' \}$. Now since $\Gamma(v) \cong \bar{K_2}$ and in particular $|\Gamma(v)| =2$ for every vertex $v$ of $M$, by $\Gamma$-connectedness of $M$ we conclude that $M = \{ \alpha, \beta \} \cup \Gamma(\alpha) \cup \Gamma^*(\alpha) \cup \{ \alpha', \beta' \}$, all relations have been determined, and $M \cong H_1$. 
\end{proof}

\begin{lemma}\label{remaining}
Let $m$ be the largest size of an independent set that embeds into 
$\Gamma(\alpha)$. Suppose that for every vertex $v \in \Gamma(\alpha)$ 
there is a unique independent subset of $\Gamma(\alpha)$, of size $m$, 
containing $v$. Then $M$ is isomorphic to one of the s-digraphs listed in 
Theorem~{\rm\ref{maintheorem} (iii)}.
\end{lemma}

\begin{proof}
By Lemma~\ref{tournament}, we may  assume that $\Gamma(\alpha)$ 
is $\Gamma$-connected and  $m\geq2$. 

\

\noindent\textit{Claim.} We have $m=2$.

\

\noindent\emph{Proof of Claim.}\quad 
Suppose to the contrary that $m\geq3$. Let $I_m = \{a_1, \ldots, a_m\}$ 
be an independent subset of $\Gamma(\alpha)$ of size $m$. 
Arguing along similar lines to the proof of Lemma~\ref{tournament}, 
let $I_m'  = \{ b_1, \ldots, b_m \}$ be the corresponding set of elements 
in $\Gamma(\beta) = \Gamma^*(\alpha)$ under the matching given
 by $\equiv$, so $a_i \equiv b_i$ for each $i$. By Lemma~\ref{isom2}(b), $\{a_1,b_2,\dots,b_m\}$ is independent so,
by $m$-set-homogeneity, there is an automorphism $\theta$ of $M$ such that
$I_m^\theta = \{ a_1, b_2, \ldots, b_m \}.$
Since $\theta$ preserves the $\equiv$ relation, $(I_m')^\theta
=\{b_1,a_2,\dots,a_m\}$.
Let $\alpha' = \alpha^\theta$ and $\beta' = \beta^\theta$, so $b_1, 
a_2\in (I_m')^\theta\subset\Gamma(\beta')$. Thus
$\{ \alpha, \beta \} \neq \{ \alpha', \beta' \}$, and also, since $\alpha,
\beta\not\in \{a_1,\dots,a_m,b_1,\dots,b_m\}$,
we have $\{\alpha', \beta'\} \cap \{ a_1, \ldots, a_m, b_1, \ldots, b_m \} 
= \varnothing$. 

We claim that $\{ \alpha', \beta' \} \cap (\Gamma(\alpha) \cup \Gamma^*(\alpha)) = \varnothing$. Indeed, by the uniqueness assumption of the lemma, the fact that $m \geq 3$, and inspection of the list in Theorem 1.1(ii), it is easy to see that the only possibilities are that either $\Gamma(\alpha) \cong \bar{K_n}$ or $\Gamma(\alpha) \cong D_3[\bar{K_n}]$ for some $n \geq 3$. Now suppose, for instance, that $\alpha' \in \Gamma(\alpha)$. Then we would have $\{ \alpha', a_1, a_2 \} \subseteq \Gamma(\alpha)$ with $a_1 \leftarrow \alpha' \leftarrow a_2$ and $a_1 \parallel a_2$. But this is impossible, since this structure clearly does not embed into either $\bar{K_n}$ or $D_3[\bar{K_n}]$. Similarly we can show $\alpha' \not\in \Gamma^*(\alpha)$, $\beta' \not\in \Gamma(\alpha)$ and $\beta' \not\in \Gamma^*(\alpha)$, proving that $\{ \alpha', \beta' \} \cap (\Gamma(\alpha) \cup \Gamma^*(\alpha)) = \varnothing$.

Now 
$
( \alpha, a_1, b_1, \alpha', \beta' ) \cong 
( \alpha, a_2, b_2, \beta', \alpha' )
$
and this configuration is rigid. By 5-set-homogeneity there is an 
automorphism $\pi$ of $M$ 
that extends this isomorphism. In particular, $\pi$ fixes $\alpha$ and 
sends $a_1$ to $a_2$. Since $I_m$ is both the unique independent set in $\Gamma(\alpha)$ of size $m$ containing $a_1$, and the
 unique independent set in $\Gamma(\alpha)$ of size $m$ containing $a_2$, it follows that $I_m^\pi = I_m$. Thus $a_2=a_1^\pi\in I_m^{\theta\pi}\subseteq \Gamma(\alpha^{\theta\pi}) =\Gamma(\beta')$, so 
$
a_2 \in \Gamma(\alpha) \cap \Gamma(\beta')\cap I_m 
$
and hence
$
a_2^\pi \in \Gamma(\alpha^\pi) \cap \Gamma({\beta'}^\pi) \cap I_m^\pi=
\Gamma(\alpha) \cap \Gamma(\alpha') \cap I_m.
$ 
Since $(I_m')^\theta=\{b_1,a_2,\dots,a_m\}\subseteq
\Gamma^*(\alpha')$, which is disjoint from $\Gamma(\alpha')$, it follows that $\Gamma(\alpha)\cap\Gamma(\alpha')\cap I_m=\{a_1\}$. Hence $a_2^\pi=a_1$.
Also, $b_3\in I_m^\theta\subseteq\Gamma(\alpha')$, and so $a_3\in
\Gamma(\beta')$. Thus $a_3\in\Gamma(\alpha)\cap\Gamma(\beta')\cap I_m$, whence 
$a_3^\pi \in \Gamma(\alpha)\cap\Gamma(\alpha')\cap I_m=\{a_1\}$, which 
implies $a_3^\pi = a_1 = a_2^\pi$, and this contradiction yields the claim. 

\

Thus $m=2$, and the digraph $\Gamma(\alpha)$ has the property that for every vertex $v \in \Gamma(\alpha)$ there is a unique 
vertex $v'$ in $\Gamma(\alpha)$ unrelated to $v$. 
In particular $|\Gamma(\alpha)|$ is even. There are no edges or arcs between the  $\equiv$-classes $[v]$ and $[v']$, so the subset $[v]\cup[v']$ induces $\overline{K_{2,2}}$, by Lemma~\ref{isom2}(b). 
It follows that  there is at
 least one $\equiv$-class $\{\alpha', \beta'\}$ unrelated to $\{\alpha, \beta\}$. 

Also, since $v, v'\in\Gamma(\alpha)=\Gamma^*(\beta)$, it follows from   2-set-homogeneity that, given any independent pair $u,u'$, the set
 $\Gamma(u) \cap \Gamma(u')$ is non-empty.

We aim to show that $M$ is equal to the disjoint union
 of $\{ \alpha, \beta\}$,  $\{\alpha', \beta'\}$, and $\Gamma(\alpha) \cup \Gamma(\beta)$. Define the following sets
\[
\begin{array}{c}
A_1 = \Gamma(\alpha) \cap \Gamma(\alpha'), \quad A_2 = \Gamma(\alpha) \cap \Gamma(\beta') \\
B_1 = \Gamma(\beta) \cap \Gamma(\beta'), \quad B_2 = \Gamma(\beta) \cap \Gamma(\alpha'). 
\end{array}
\]   
It follows from the observation in the previous paragraph that $A_1, A_2, B_1, B_2$
are all non-empty, and by 
$2$-set-homogeneity we have $A_1 \cong A_2 \cong B_1 \cong B_2$. 
Let $a_1 \in A_1$, and $[ a_1]=\{a_1,b_1\}$. Then $[a_1]\cup[\alpha]$ and
$[a_1]\cup[\alpha']$ both induce $\overline{D_4}$, and it follows that 
$b_1 \in B_1$. Similarly there exist $a_2\in A_2$ and $b_2\in B_2$ with
$a_2\equiv b_2$.
Then 
$
(\alpha, a_1, b_1, \alpha', \beta'  )\cong 
( \alpha, a_2, b_2, \beta', \alpha' )
$
and this configuration is rigid, so by $5$-set-homogeneity there is an 
automorphism of $M$ that fixes $\alpha, \beta$ and swaps $\alpha', \beta'$. 
Also, note that
 since every vertex in $\Gamma(\alpha)$ is independent from a unique vertex  in $\Gamma(\alpha)$, it follows that the orbital $\Lambda$ is self-paired. 

Now we show that $\Gamma(\alpha) = A_1 \cup A_2$.  Let $N = |\Gamma(\alpha)| = |\Gamma^*(\alpha)|$, so $N$ is even. Let $v,v' \in \Gamma(\alpha)$
 with $v$ unrelated to $v'$, and let $w,w' $ be the corresponding 
elements in $\Gamma(\beta)$ such that $v \equiv w$ and $v' \equiv w'$.  
Define the following sets:
\[
\begin{array}{c}
X_1 = \Gamma(v) \cap \Gamma(v') \cap \Gamma(\beta), \quad X_2 = \Gamma^*(v) \cap \Gamma^*(v') \cap \Gamma(\beta) \\
Y_1 = \Gamma(v) \cap \Gamma^*(v') \cap \Gamma(\beta), \quad Y_2 = \Gamma^*(v) \cap \Gamma(v') \cap \Gamma(\beta). \\
\end{array}
\]
For each vertex $x$ in $\Gamma^*(\alpha)=\Gamma(\beta)$ there is, in 
$\Gamma^*(\alpha)$, a unique vertex $x'$ unrelated to $x$, and it follows (on considering 
$[v]\cup[x]$ and $[v']\cup[x]$, for $x\in\Gamma(\beta)\setminus\{w,w'\}$ and by Lemma~\ref{isom2}(b)) that 
\[
\Gamma(\beta) \setminus \{ w,w' \} = X_1 \cup X_2 \cup Y_1 \cup Y_2.
\]
Let $X_1', X_2', Y_1', Y_2'$ be the corresponding subsets of $\Gamma(\alpha)$ given by the $\equiv$-matching. Note that 
$
|\Gamma(\beta) \setminus \{ w,w' \}| = N-2.
$
So we have
$
N-2 = |X_1 \cup X_2| + |Y_1 \cup Y_2|. 
$

We claim that $|\Gamma(v) \cap \Gamma(v')|\geq N/2$. 
If $|X_1 \cup X_2| \geq |Y_1 \cup Y_2|$ then since
$
\Gamma(v) \cap \Gamma(v') \supseteq X_1 \cup X_2' \cup \{ \beta \},
$
it follows that
\[
|\Gamma(v) \cap \Gamma(v')| \geq |X_1 \cup X_2'| + 1 \geq (N-2)/2 + 1 = N/2.
\]
On the other hand, if $|X_1 \cup X_2| < |Y_1 \cup Y_2|$, then since 
$
\Gamma(v) \cap \Gamma^*(v') \supseteq Y_1 \cup Y_2'
$
we have 
\[
|\Gamma(v) \cap \Gamma^*(v')| \geq |Y_1 \cup Y_2'| > (N-2)/2
\]
and hence $|\Gamma(v) \cap \Gamma^*(v')| \geq N/2$. 
Since $\Gamma(v')=\Gamma^*(w')$ and, as we showed above, some automorphism of  
$M$ fixes $v$ and $w$ and swaps $v'$ and $w'$,  we see also in this case that
\[
|\Gamma(v) \cap \Gamma(v')| = |\Gamma(v) \cap \Gamma^*(w')|
= |\Gamma(v) \cap \Gamma^*(v')| \geq N/2 
\]
proving the claim.

Thus, since $G$ is transitive on independent pairs, for any independent pair 
$x,y$ in the  graph $M$ we have $|\Gamma(x) \cap \Gamma(y)| \geq N/2$. 
In particular this is true for the pairs $(\alpha, \alpha'), (\alpha,\beta'), (\beta,\alpha'), (\beta,\beta')$ and we conclude that
\[
|A_1| = |A_2| = |B_1| = |B_2| = N/2 = |\Gamma(\alpha)|/2.
\]
Therefore we have
$
\Gamma(\alpha) = A_1 \cup A_2$ and $\Gamma(\beta) = B_1 \cup B_2. 
$
Since we are assuming that $M$ is $\Gamma$-connected, and every vertex in the set below has the correct number $2N$ of $(\Gamma \cup \Gamma^*)$-neighbours, 
we have accounted for all the vertices and
\[
M = \{\alpha, \beta\} \cup A_1 \cup A_2 \cup B_1 \cup B_2 \cup \{\alpha', \beta'\}. 
\]
In particular, $[\alpha']$ is the unique $\equiv$-class unrelated to $[\alpha]$, and hence is $G_\alpha$-invariant. 
Since $\Gamma(\alpha)$ is connected,  there is an arc between $A_1$ and $A_2$.
Moreover, the partition $A_1 \cup A_2$ of $\Gamma(\alpha)$ is preserved 
by $G_{\alpha}$, so it follows from 3-set-homogeneity that neither  $A_1$ nor
$A_2$ contains an arc. Therefore each $A_i$ is an independent set, and
since $m=2$ it follows that $|A_i|=2$ and between every $a_1 \in A_1$ and 
$a_2 \in A_2$ there is an arc.  Along with the fact that $\Gamma(\alpha)$ is 
set-homogeneous, this implies that $\Gamma(\alpha)$ is isomorphic to $D_4$. The structure on $M$ is now 
fully determined, and we can verify that
$M \cong H_2$. 
\end{proof}

Now we return to the proof of Theorem~\ref{maintheorem}(iii) in the case where
Lemma~\ref{isom2}(b) holds. The case $|\Gamma(\alpha)|=1$ is dealt 
with in the first paragraph of Section~\ref{sec_CaseInNEqualsOutN} so we may assume that 
$|\Gamma(\alpha)|>1$. Also, as noted there, we may assume inductively that assumption (A) holds and in particular that $\Gamma(\alpha)$ is isomorphic to a 
set-homogeneous a-digraph listed in 
Theorem~\ref{maintheorem}(ii).  
The case $\Gamma(\alpha) \cong D_5$ is dealt 
with by Lemma~\ref{d5}, and the cases with $\Gamma(\alpha)$ isomorphic to  $D_3$, $\bar{K_n}$,
or $\bar{K_n}[D_3]$ $(n \geq 1)$ are  handled by Lemma~\ref{tournament}. 
Finally if $\Gamma(\alpha)$ is isomorphic to $D_4$, $H_0$, 
or $D_3[\bar{K_n}]$ $(n \geq 2)$, then it satisfies the conditions of 
Lemma~\ref{remaining}, and so these cases are dealt with by that lemma. 
This covers all possibilities for $\Gamma(\alpha)$ and so completes the 
proof of Theorem~\ref{maintheorem}(iii) in the case where 
Lemma~\ref{isom2}(b) holds. 

\bigskip
It remains to consider the case 
where Lemma~\ref{isom2}(a) holds, that is, where $\beta \ne \gamma$. 

\subsection{Dealing with Lemma~\ref{isom2}(a).}

Throughout this subsection, replacing $M$ by $\overline M$ if necessary, we may assume that $M$ is a symmetric digraph satisfying the conditions of Lemma~\ref{isom2}(a), so that each $\equiv$-class is isomorphic to one of $C_3=K_3$, $C_4$ or $C_5$. Also, as noted in the first paragraph of Section~\ref{sec_CaseInNEqualsOutN}, we may assume that Assumption (A) holds. 
The rest of this section will be spent dealing with this case, and the examples $J_n$ (for $n\geq 3$) and $H_3$ are identified. 

\begin{lemma}
Under the assumptions for this subsection listed above, each $\equiv$-class is isomorphic to $K_3$. 

\end{lemma}
\begin{proof}
Suppose to the contrary that each $\equiv$-class $B$
induces $C_j$, where $j\in\{4,5\}$. Then $B$ contains both a 
$\sim$-related pair and an independent pair, and so by 2-set-homogeneity, 
any two vertices in distinct $\equiv$-classes are related by an arc. 
Let $B_1, B_2$ be distinct $\equiv$-classes. Suppose that all arcs go 
from $B_1$ to $B_2$.
Then the set of $\equiv$-classes  carries the structure of a 
set-homogeneous tournament, which, by Lemma~\ref{sethomtourn}, 
must be $D_1$ or $D_3$, and hence $M$ is
$C_j$ or $D_3[C_j]$. However, neither of these examples satisfies Assumption (A)
(as either $|\Gamma(\alpha)|=1$ or condition (*) fails). 

Thus we may assume that there are arcs in both directions between 
$B_1$ and $B_2$, so there is an automorphism of $M$ 
interchanging $B_1$ 
and $B_2$. In particular the number of arcs from $B_1$ to $B_2$ equals the 
number from $B_2$ to $B_1$. Since the total number 
of such arcs is $j^2$, and must be even, it follows that $j=4$.
So the $\equiv$-classes $B_i$ are cycles $C_4$ and as noted in Lemma~\ref{isom2}(a), 
the group induced on $B_i$ is cyclic of order 4. 
Furthermore, for each $\mu\in B_2$, $|\Gamma(\mu)\cap B_1|=2$ and $|\Gamma^*(\mu)\cap B_1|=2$,
(for by transitivity on arcs, if $\mu,\nu \in B_2$ and $\Gamma(\mu)\cap B_1$
 and $\Gamma(\nu)\cap B_1$ are non-empty, then both sets have the same size, and as the
 number of arcs from $B_1$ to $B_2$ equals the number from $B_2$ to $B_1$, the only 
remaining possibility is that there are  two elements $x\in B_2$ such that 
$\Gamma(x)\supseteq B_1$, and that for the remaining elements 
$y\in B_2$, $\Gamma^*(y)\supseteq B_1$ -- an impossibility as $B_1$ and $B_2$ can be swapped).

Let $B_1=\{\beta_i:0\leq i\leq 3\}$, with
 $\beta_i \sim \beta_{i+1 ({\rm mod}~ 4)}$ for each $i$. 
Also let $\mu \in B_2$ such
that $\Gamma(\mu)\cap B_1$ contains $\beta_0$. 
As $\Gamma^*(\beta)=\Gamma(\alpha)$, there is no $\delta$ with $\alpha,\beta\in \Gamma(\delta)$. Hence, if $\alpha\sim \beta$, then
by 2-set-homogeneity on edges,
$\Gamma(\alpha)$ has no edges, so
$\Gamma(\mu) \cap B_1$ cannot contain an edge of $B_1$; hence, as $|\Gamma(\mu)\cap B_1|=2$,  $\Gamma(\mu)\cap B_1
=\{\beta_0,\beta_2\}$. On the other hand if $\alpha\not\sim\beta$ then, as 
$[\alpha] \cong C_4$, we have $\alpha \sim \gamma \sim \beta$, and as $\Gamma^*(\gamma)=\Gamma(\beta)$ we again find that $\Gamma(\beta)$ contains no edges, and hence that  $\Gamma(\mu)\cap B_1
=\{\beta_0,\beta_2\}$.

Since $|\Gamma^*(x)\cap B_2|=2$ for all $x\in B_1$, there is $\delta \in B_2$ with $\Gamma(\delta)\cap B_1
=\{\beta_1,\beta_3\}$. Then 
$(\beta_0,\beta_1,\mu)\cong (\beta_1,\beta_0,\delta)$ and this 
configuration is rigid, but any automorphism of $M$ extending this isomorphism
fixes $B_1$ setwise and interchanges $\beta_0,\beta_1$, which is impossible 
since the group induced on $B_1$ is the cyclic group $\mathbb{Z}_4$. 
\end{proof}

The rest of this subsection will be spent dealing with the remaining case where  each $\equiv$-class is isomorphic to $K_3$.  

\begin{lemma}\label{structure}
Suppose that each $\equiv$-class of $M$ induces $K_3$, and let $B_1$ denote the $\equiv$-class $\{ \alpha, \beta, \gamma \}$. Then we have the following.
\begin{enumerate}[(i)]
\item $\Delta$ is not self-paired.

\item The following equalities hold:  
	\begin{itemize}
	\item $\Gamma(\alpha) = \Gamma^*(\beta) = \Lambda(\gamma)$ 
	\item $\Gamma(\beta)=\Gamma^*(\gamma) = \Lambda(\alpha)$
	\item $\Gamma(\gamma) = \Gamma^*(\alpha) = \Lambda(\beta)$
      	\end{itemize}	
\item For any two $\equiv$-classes $B$ and $B'$ the subdigraph induced by $B \cup B'$ is isomorphic to $\overline{E_6}$, so $\Lambda$ is  self-paired.
\item The set $\Gamma(\alpha)$ intersects every $\equiv$-class other than $B_1$ in exactly one vertex. The same is true for $\Gamma^*(\alpha)$ and $\Lambda(\alpha)$. 
\item The vertex set of $M$ is 
$B_1\cup\Gamma(\alpha)\cup\Gamma(\beta)\cup\Gamma(\gamma)$. 
\end{enumerate}
\end{lemma}
\begin{proof}

(i) As noted in Lemma~\ref{isom2}(a), the group induced on $B_1$ is cyclic of order 3, so admits the cycle $(\alpha,\beta,\gamma)$. An automorphism inducing $(\alpha,\beta)$ would have to fix $B_1$ so would fix $\gamma$, which is impossible. 

(ii) By Lemma~\ref{isom2}, $\Gamma(\alpha)=\Gamma^*(\beta)$, $\Gamma^*(\alpha)
=\Gamma(\gamma)$, and it follows using an automorphism inducing the 3-cycle $(\alpha,\beta,\gamma)$ that also 
$\Gamma(\beta)=\Gamma^*(\gamma)$. In particular, the sets $\Gamma(\alpha),\Gamma(\beta),\Gamma(\gamma)$ are disjoint. Thus 
(possibly replacing $\Lambda$ by $\Lambda^*$ if necessary -- but see the second assertion of (iii))
we have  $\Gamma(\alpha)=
\Lambda(\gamma)$, and hence also
$\Lambda(\alpha)=\Gamma(\beta)$ and $\Lambda(\beta)=\Gamma(\gamma)$.

(iii) Let $B_2$ be another $\equiv$-class such that there are arcs 
between $B_1$ and $B_2$. Without loss of generality we may suppose that $\alpha'\in B_2$ with $\alpha\rightarrow \alpha'$, and hence also
$\alpha'\to \beta$. Thus, by 2-set-homogeneity there is $g\in G$ with $(\alpha,\alpha')^g=(\alpha',\beta)$, and $g$ interchanges $B_1$ and $B_2$.  As $\Gamma(\alpha')\cap B_1\neq \varnothing$, we also have $\Gamma(\beta)\cap B_2\neq \varnothing$. It is now easily seen that $B_1 \cup B_2$ carries the structure of $\overline{E_6}$.
 If $\{a_1,a_2,a_3,a_4,a_5,a_6\}$ is a copy of $\overline{E_6}$
with $a_i\to a_{i+1}$ ($\mod 6$) then by 4-set-homogeneity there is an automorphism taking $(a_1,a_2,a_3,a_4)$ to $(a_4,a_5,a_6,a_1)$, and hence interchanging $a_1$ and $a_4$. So $\Lambda$ is self-paired.

Since $B_1 \times B_2$ has pairs in $\Gamma$, $\Gamma^*$ and $\Lambda$, and since, for each  pair $C_1,C_2$ of distinct $\equiv$-classes, $C_1 \times C_2$ contains a pair in at least one of these orbitals, it follows from 2-set-homogeneity that $G$ is 2-homogeneous on the set of $\equiv$-classes, and in fact (as $g$ above swaps $B_1$ and $B_2$) this action is 2-transitive. Thus $B\cup B'$ induces $\overline{E_6}$ for each pair $B, B'$ of distinct $\equiv$-classes.

(iv), (v) These follow immediately from (iii) and 2-transitivity on the set of $\equiv$-classes.
\end{proof}

As noted in Assumption (A), the subdigraph induced by $\Gamma(\alpha)$ is a set-homogeneous s-digraph isomorphic to one of the digraphs listed in Theorem~\ref{maintheorem}(iii). We restrict the possibilities further.

\begin{lemma}\label{Nisadigraph}
If each $\equiv$-class of $M$ induces $K_3$, then
the subdigraph induced by $\Gamma(\alpha)$ is isomorphic to one of: $D_3$, $D_4$, $D_5$, $H_0$, $\bar{K}_n$, $\bar{K}_n[D_3]$ or
$D_3[\bar{K}_n]$, for some $n\geq2$. Also, $M$ has more than two $\equiv$-classes.
\end{lemma}
\begin{proof}
It follows from Lemma~\ref{structure}(iv) that $\Gamma(\alpha)$ contains no pairs of vertices that are $\sim$-related. Along with Lemma~\ref{nhood} this implies that $\Gamma(\alpha)$ is a set-homogeneous a-digraph, and hence is in the list of Theorem~\ref{maintheorem}(ii). Note that, by Assumption (A), $|\Gamma(\alpha)|>1$, and each $\equiv$-class meets $\Gamma(\alpha)$ in at most one point, so there are at least 3 $\equiv$-classes.
\end{proof}

\begin{lemma}
\label{lem_list}
If each $\equiv$-class of $M$ induces $K_3$ then $\Gamma(\alpha)\not\cong D_3$, and in particular $\Gamma(\alpha)$ contains an unrelated pair.
\end{lemma}
\begin{proof} \begin{sloppypar} 
By Lemma~\ref{Nisadigraph} the only possibility for $\Gamma(\alpha)$ that does not contain an unrelated pair is $D_3$, so it is sufficient to prove that $D_3$ does not arise. Assume to the contrary that $\Gamma(\alpha) \cong D_3$. \end{sloppypar}

Now by Lemma~\ref{structure}, $M$ has exactly four $\equiv$-classes, say $B_1 = \{ \alpha, \beta, \gamma \}$ and $B_i = \{ x_i, y_i, z_i  \}$, for $i = 2,3,4$, where the vertices have been labelled so that $\Gamma(\alpha) = \Gamma^*(\beta) = \Lambda(\gamma) = \{ x_2,x_3,x_4 \}$, $\Gamma(\beta)=\Gamma^*(\gamma) = \Lambda(\alpha) = \{ y_2, y_3, y_4 \}$, and 
$\Gamma(\gamma) = \Gamma^*(\alpha) = \Lambda(\beta) = \{ z_2, z_3, z_4 \}$. 

By assumption $\Gamma(\alpha) = \{ x_2, x_3, x_4 \} \cong D_3$, and we may suppose that the vertices are labelled so that  $x_2 \rightarrow x_3 \rightarrow x_4 \rightarrow x_2$. Since $M$ embeds $D_3$, by set-homogeneity applied to arcs it follows that every arc of $M$ belongs to at least one copy of $D_3$. In particular this applies to the arc $\alpha \rightarrow x_2$. Now since $\Gamma^*(\alpha) = \Gamma(\gamma) = \{ z_2, z_3, z_4 \}$, and since it is not the case that $x_2 \rightarrow z_2$ (because $x_2\sim z_2$), it follows that either we have $\alpha \rightarrow x_2 \rightarrow z_3 \rightarrow \alpha$ or  $\alpha \rightarrow x_2 \rightarrow z_4 \rightarrow \alpha$. Suppose the latter -- the former case is dealt with in  the same way. Now we have $z_4 \in \Gamma(\gamma)$ so $\gamma \in \Gamma^*(z_4)$, and also $x_2 \in \Gamma^*(z_4)$. But $x_2 \in \Gamma(\alpha) = \Lambda(\gamma)$ (by Lemma~\ref{structure}) and it follows that $\{ \gamma, x_2 \} \subseteq \Gamma^*(z_4)$ is an unrelated pair,  contradicting the fact that $\Gamma^*(z_4) \cong D_3$. 
\end{proof}

Now we deal with the case that $M$ does not embed the $3$-chain $P_3$. Recall from Section~1
 the definition of the s-digraph $J_n$ listed in Theorem~\ref{maintheorem}(iii).

\begin{lemma}\label{noP3}
If each $\equiv$-class 
induces $K_{3}$, and $\Gamma(\alpha)\cong \overline{K_n}$, then
$M$ is isomorphic to $J_n$, where $n=1+|\Gamma(\alpha)|$. 
\end{lemma}
\begin{proof}
Let $m=|\Gamma(\alpha)|$. By assumption $m \geq 2$ and $\Gamma(\alpha)$ contains no arcs. It follows from Lemma~\ref{structure} (especially part (v)) that `independence or equality' is an 
equivalence relation $E$ on $M$, and there are precisely three 
$E$-classes, namely $\{\alpha\}\cup\Lambda(\alpha)$, $\{\beta\}
\cup\Lambda(\beta)$, and $\{\gamma\}\cup\Lambda(\gamma)$.

It is now straightforward to show that $M \cong J_{m+1}$. For example, we must show that if $\alpha'\in \Lambda(\alpha)$ and $\beta'\in \Lambda(\beta)$ with $\alpha'\not\equiv \beta'$, then
$\beta' \to \alpha'$. To see this, by 2-set-homogeneity there is $g\in G$ with $(\alpha,\beta')^g=(\alpha',\beta)$ (as both pairs are in $\Gamma^*$). Then $g$ fixes the three $E$-classes, and as $x \to \alpha $ for all $x\in (\{\beta\}\cup \Lambda(\beta))$ with $x\not\equiv \alpha$, also
$x \to\alpha'$ for all $x\in (\{\beta\}\cup \Lambda(\beta))$ with $x\not\equiv\alpha'$.
\end{proof}

By Lemmas~\ref{Nisadigraph}, \ref{lem_list} and \ref{noP3}, we may assume that $\Gamma(\alpha)$ is isomorphic to one of $D_4$, $D_5$, $H_0$, $\overline{K_n}[D_3]$ or $D_3[\overline{K_n}]$ (for some $n \geq 2$). In each case $\Gamma(\alpha)$ contains an arc and hence $M$ embeds a 3-chain $P_3$. We shall see that under the latter assumption $M$ is uniquely determined. 
 Over the next few lemmas we consider each of the possibilities for $\Gamma(\alpha)$ in turn, ruling them all out, with the exception of $H_0$. Then in Lemma~\ref{27VertexExample} we show that  $\Gamma(\alpha) \cong H_0$ implies that $M \cong H_3$.   

\begin{lemma}\label{NOTD5}
If each $\equiv$-class of $M$ induces $K_3$, then $\Gamma(\alpha)\not\cong D_5$. 
\end{lemma}
\begin{proof}
Suppose that $\Gamma(\alpha) \cong D_5$. Let $v \in \Lambda(\alpha)$. Then $G_{\alpha v}$ fixes the $\equiv$-class $[v]$, and hence fixes the unique point $u$ in $[v]\cap\Gamma(\alpha)$. It follows, since $\Aut(D_5)$ is regular, that $G_{\alpha v}$ fixes $M$ pointwise. 
Let $w\equiv v$ with $w \in \Gamma^*(\alpha)$, so $[v]=\{u,v,w\}$.
By Lemma~\ref{structure}(ii), $\Gamma(\alpha) = \Gamma^*(\beta) = \Lambda(\gamma)$ and it follows that $\alpha\rightarrow u \rightarrow \beta\rightarrow v\rightarrow \gamma\rightarrow w\rightarrow\alpha$. 

Since $\Lambda$ is self-paired (by Lemma~\ref{structure}(iii)), there is an automorphism  $\pi$ extending the transposition $(\alpha \ v)$, and $\pi$ must swap the 
$\equiv$-classes of $\alpha$ and $v$, and fix $\Gamma(\alpha)\cap \Gamma(v)$ setwise. Thus, as $|G_{\alpha v}|=1$, $\pi$ induces one of the transpositions $(\beta \ u)$ or $(\gamma \ u)$, and it follows on considering the directed 6-cycle in the previous paragraph that $\pi$ must induce $(\gamma\ u)(\beta\ w)$. 

Now, as $|\Gamma(v)|=5$ it follows that $\Gamma(v)$ contains at least two points $u_1,u_2$ in one of 
$\Lambda(\alpha)$, $\Gamma^*(\alpha)$ or $\Gamma(\alpha)$. In the first two cases, $\{\alpha, v,u_i\}$ is rigid for each $i$, so by set-homogeneity there is an automorphism fixing $\alpha$ and $v$ and mapping 
$u_1$ to $u_2$, contradicting $G_{\alpha v}=1$. Thus $\Gamma(\alpha)\cap\Gamma(v)\supseteq\{u_1,u_2\}$.

Suppose that equality holds. By Lemma~\ref{structure}(ii), $\Gamma(\alpha) = \Gamma^*(\beta) = \Lambda(\gamma)$ and since $\{ u_1, u_2 \} \subseteq \Gamma(\alpha)$ it follows that $u_i \rightarrow \beta$ and $(\gamma, u_i ) \in \Lambda$ for $i=1,2$.
Since $\pi$  induces $(\gamma\ u)(\beta\ w)$ and fixes $\{u_1,u_2\}$, it follows that
$u_i \rightarrow w$ and $(u,u_i)\in\Lambda$ for $i=1,2$. 
But now the isomorphism $(\alpha, v,w,u_1) \mapsto
(\alpha, v, w, u_2)$ extends to an automorphism, since this configuration 
is rigid; this is impossible since $G_{\alpha v}=1$.

Thus $|\Gamma(v) \cap \Gamma(\alpha)| \geq 3$, say $\Gamma(v) \cap \Gamma(\alpha) \supseteq \{ u_1, u_2, u_3 \}$. The induced digraphs on $\{\alpha, v, u_i\}$ are isomorphic but not rigid, for $i=1,2,3$, so by 3-set-homogeneity there are automorphisms
$\pi_{ij}$ of $M$ that map  $\{\alpha, v, u_i\}$ to $\{\alpha, v, u_j\}$ for distinct $i,j$. Now $u_i^{\pi_{ij}}=u_j$, and since $G_{\alpha v}=1$ it follows that $\pi_{ij}$ interchanges $\alpha$ and $v$. However this implies that the composition $\pi_{12}\circ\pi_{23}$ lies in $G_{\alpha v}$ and maps $u_1$ to $u_3$, which is a contradiction.
\end{proof}

\begin{lemma}\label{morerulingout}
If each $\equiv$-class of $M$ induces $K_3$, then $\Gamma(\alpha)\not\cong D_3[\bar{K_n}]$ for any $n \geq 2$. 
\end{lemma}
\begin{proof}
Suppose to the contrary that $\Gamma(\alpha) \cong D_3[\bar{K_n}]$ $(n \geq 2)$. Fix $a \in \Gamma(\alpha)$ and consider $\Gamma(a)$. Since $M$ embeds $D_3[\bar{K_n}]$ as an induced subgraph it follows that every arc in $M$ extends to a copy of $D_3$ in at least $n$ different ways, in particular this is true of the arc $(\alpha,a)$. In other words $|\Gamma^*(\alpha) \cap \Gamma(a)| \geq n$. Let $\Gamma^*(\alpha) \cap \Gamma(a) \supseteq  \{x_1, \ldots, x_n  \}$. By Lemma~\ref{structure} we have $\Gamma^*(\beta) = \Gamma(\alpha)$ so $a \rightarrow \beta$, and $\Lambda(\beta) = \Gamma^*(\alpha)$. Now $\{ \beta, x_1, \ldots, x_n  \} \subseteq \Gamma(a)$ where $\beta$ is $\Lambda$-related to $x_i$ for all $i$. 
This is a contradiction since in $\Gamma(a)$ each vertex is $\Lambda$-related to exactly $n-1$ other vertices. 
\end{proof}

\begin{lemma}\label{2HomAction}
If $\Gamma(\alpha)$ is isomorphic to one of: $D_4$, $H_0$, or $\bar{K}_n[D_3]$ for some  $n>1$, then for all $a_1, a_2 \in \Gamma(\alpha)$, if $(a_1,a_2) \in \Lambda$ then there exists $g \in G_{\alpha}$ such that $a_1^g = a_2$ and $a_2^g=a_1$. 
\end{lemma}
\begin{proof}
Suppose first that $\Gamma(\alpha) \cong D_4$. Say $\Gamma(\alpha) = \{ d_1, d_2, d_3, d_4 \}$ with 
$d_i \rightarrow d_{i+1} ~~({\rm mod}~4)$. Then by set-homogeneity the isomorphism $(\alpha, d_1, d_2, d_3) \mapsto (\alpha,d_3, d_4, d_1 )$ extends to an automorphism (since this configuration is rigid) interchanging the non-adjacent pair $\{ d_1, d_3  \}$ and also the pair $\{ d_2, d_4\} $. 

Similar arguments allow us to deal with the cases $\Gamma(\alpha)$ isomorphic to $H_0$, and $\bar{K}_n[D_3]$ for $n>1$.
\end{proof}

\begin{lemma}\label{doubleedgeconfigurations}
Suppose that each $\equiv$-class induces $K_3$. Let $f: \Gamma(\alpha) \rightarrow \Gamma^*(\alpha)$ be the bijection 
defined by $(a,f(a)) \in \Delta\cup\Delta^*$ for all $a \in \Gamma(\alpha)$. Then for all $a_1, a_2 \in \Gamma(\alpha)$, the subdigraph induced by $\{ a_1, a_2, f(a_1), f(a_2) \}$ is one of the $6$ digraphs illustrated in Figure~{\rm\ref{fig_6digraphs}}. In particular $f$ is a digraph isomorphism. 
\end{lemma}
\begin{figure}
\def\x{0.55}
\begin{center}
\begin{tikzpicture}[decoration={ 
markings,
mark=
at position \x
with 
{ 
\arrow[scale=1.5]{stealth} 
} 
} 
]
\tikzstyle{vertex}=[circle,draw=black, fill=white, inner sep = 0.75mm]

\matrix[row sep=5mm,column sep = 2cm]{
 

\node (A1)  [vertex,label={90:{\small $a_1$}}] at (0,1.5) {};
\node (A2)  [vertex,label={90:{\small $a_2$}}] at (1.5,1.5) {};
\node (FA1) [vertex,label={-90:{\small $f(a_1)$}}] at (0,0) {};
\node (FA2) [vertex,label={-90:{\small $f(a_2)$}}] at (1.5,0) {};
\arrowdraw{A1}{A2}
\arrowdraw{A2}{FA1}
\arrowdraw{FA1}{FA2}
\edgedraw{A1}{FA1}
\edgedraw{A2}{FA2}

&

\node (A1)  [vertex,label={90:{\small $a_1$}}] at (0,1.5) {};
\node (A2)  [vertex,label={90:{\small $a_2$}}] at (1.5,1.5) {};
\node (FA1) [vertex,label={-90:{\small $f(a_1)$}}] at (0,0) {};
\node (FA2) [vertex,label={-90:{\small $f(a_2)$}}] at (1.5,0) {};
\arrowdraw{A1}{A2}
\arrowdraw{FA2}{A1}
\arrowdraw{FA1}{FA2}
\edgedraw{A1}{FA1}
\edgedraw{A2}{FA2}

&

\node (A1)  [vertex,label={90:{\small $a_1$}}] at (0,1.5) {};
\node (A2)  [vertex,label={90:{\small $a_2$}}] at (1.5,1.5) {};
\node (FA1) [vertex,label={-90:{\small $f(a_1)$}}] at (0,0) {};
\node (FA2) [vertex,label={-90:{\small $f(a_2)$}}] at (1.5,0) {};
\arrowdraw{A2}{A1}
\arrowdraw{A1}{FA2}
\arrowdraw{FA2}{FA1}
\edgedraw{A1}{FA1}
\edgedraw{A2}{FA2}

\\


\node (A1)  [vertex,label={90:{\small $a_1$}}] at (0,1.5) {};
\node (A2)  [vertex,label={90:{\small $a_2$}}] at (1.5,1.5) {};
\node (FA1) [vertex,label={-90:{\small $f(a_1)$}}] at (0,0) {};
\node (FA2) [vertex,label={-90:{\small $f(a_2)$}}] at (1.5,0) {};
\arrowdraw{A2}{A1}
\arrowdraw{FA1}{A2}
\arrowdraw{FA2}{FA1}
\edgedraw{A1}{FA1}
\edgedraw{A2}{FA2}

&

\def\x{0.75}; 

\node (A1)  [vertex,label={90:{\small $a_1$}}] at (0,1.5) {};
\node (A2)  [vertex,label={90:{\small $a_2$}}] at (1.5,1.5) {};
\node (FA1) [vertex,label={-90:{\small $f(a_1)$}}] at (0,0) {};
\node (FA2) [vertex,label={-90:{\small $f(a_2)$}}] at (1.5,0) {};

\arrowdraw{A2}{FA1}
\arrowdraw{A1}{FA2}
\edgedraw{A1}{FA1}
\edgedraw{A2}{FA2}

&

\def\x{0.75}; 

\node (A1)  [vertex,label={90:{\small $a_1$}}] at (0,1.5) {};
\node (A2)  [vertex,label={90:{\small $a_2$}}] at (1.5,1.5) {};
\node (FA1) [vertex,label={-90:{\small $f(a_1)$}}] at (0,0) {};
\node (FA2) [vertex,label={-90:{\small $f(a_2)$}}] at (1.5,0) {};
\arrowdraw{FA2}{A1}
\arrowdraw{FA1}{A2}
\edgedraw{A1}{FA1}
\edgedraw{A2}{FA2}
\\
};
\end{tikzpicture}
\end{center}
\caption{The $6$ possibilities for the digraph induced by $\{ a_1, a_2, f(a_1), f(a_2) \}$ in Lemma~\ref{doubleedgeconfigurations}.} 
\label{fig_6digraphs}
\end{figure}
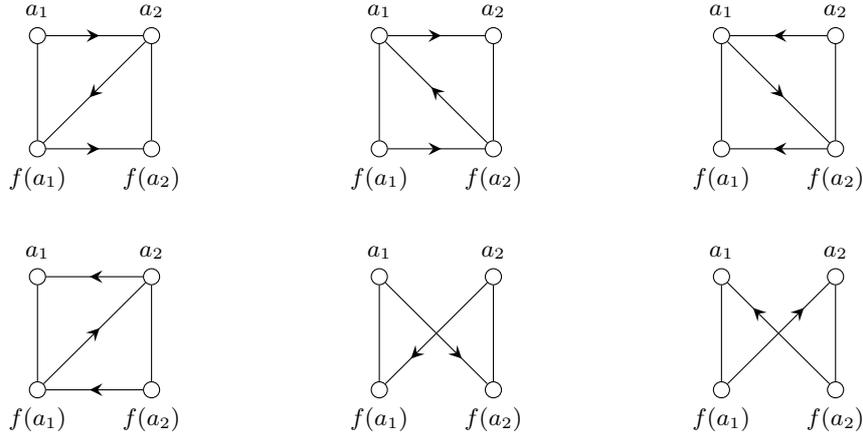
\begin{proof}
By Lemmas~\ref{Nisadigraph} -- \ref{morerulingout} the digraph $\Gamma(\alpha)$ is isomorphic to one of: $D_4$, $H_0$, or $\bar{K}_n[D_3]$ for some $n>1$. Let  $a_1, a_2 \in \Gamma(\alpha)$. By Lemma~\ref{structure}(iii), the substructure induced by $\{ a_1, a_2, f(a_1), f(a_2) \}$ embeds in the s-digraph $\bar{E_6}$. Also if  $(a_1,a_2) \in \Lambda$ then  by Lemma~\ref{2HomAction} it follows that there exists $g \in G_{\alpha}$ such that $a_1^g = a_2$ and $a_2^g=a_1$.  

First suppose that $(a_1, a_2) \in \Lambda$. Then from the comment in the previous paragraph there exists $g \in G_{\alpha}$ such that $a_1^g = a_2$ and $a_2^g=a_1$, and so $f(a_1)^g = f(a_2)$ and $f(a_2)^g=f(a_1)$, which since $\Gamma^*(\alpha)$ is an a-digraph implies $(f(a_1), f(a_2)) \in \Lambda$. With these restrictions, since  $\{ a_1, a_2, f(a_1), f(a_2) \}$ embeds in the digraph $\bar{E_6}$ we see by inspection that the only possibilities for the digraph induced by $\{ a_1, a_2, f(a_1), f(a_2) \}$ are the last two of the six illustrated in Figure~\ref{fig_6digraphs}. 

Now suppose that $a_1 \rightarrow a_2$ (the case $a_2 \rightarrow a_1$ is dual to this one). If $f(a_1)$ and $f(a_2)$ were unrelated then a similar argument to that of the prevous paragraph would imply that $a_1$ and $a_2$ are unrelated, which is not the case. 
It follows that either $f(a_1) \rightarrow f(a_2)$ or $f(a_2) \rightarrow f(a_1)$. We claim that the second of these possibilities cannot happen. Indeed, if $f(a_2) \rightarrow f(a_1)$ then by inspection of the substructures of $\bar{E_6}$ it would follow that the subdigraph induced by $\{ a_1, a_2, f(a_1), f(a_2) \}$ has $a_1 \rightarrow a_2$, $f(a_2) \rightarrow f(a_1)$, $(a_1,f(a_1)) \in \Delta\cup \Delta^*$, $(a_2,f(a_2)) \in \Delta\cup\Delta^*$, and every other pair $\Lambda$-related.  
But then by set-homogeneity the isomorphism $(f(a_1),\alpha,a_2) \mapsto (f(a_2),\alpha,a_1)$ extends to an automorphism $ \pi$ (since this configuration is rigid). But this forces $a_1^\pi = a_2$ (since $f(a_1)^\pi = f(a_2)$) and $a_2^\pi = a_1$, contradicting $a_1 \rightarrow a_2$. 

We conclude that if $a_1 \rightarrow a_2$ then $f(a_1) \rightarrow f(a_2)$.  Consideration of the substructures of $\bar{E_6}$ shows that the only possibilities for the subdigraph induced by $\{ a_1, a_2, f(a_1), f(a_2) \}$ are those illustrated in Figure~\ref{fig_6digraphs}. \end{proof}

\begin{lemma}
Suppose that each $\equiv$-class induces $K_3$.
Let $A$ and $B$ be any distinct pair from $\{  \Gamma(\alpha)$, $\Gamma^*(\alpha), \Lambda(\alpha) \}$. Then the bijection $f:A \rightarrow B$ defined by $(a,f(a)) \in \Delta\cup\Delta^*$ for all $a \in A$, is a digraph isomorphism.  
\end{lemma}
\begin{proof}
This follows from Lemma~\ref{structure}(ii) and Lemma~\ref{doubleedgeconfigurations}.

\end{proof}

\begin{lemma}\label{speciald3s}
Suppose that each $\equiv$-class induces $K_3$. Let $a_1, a_2 \in \Gamma(\alpha)$ with $a_1 \rightarrow a_2$. Then there does not exist $c \in \Gamma^*(\alpha) \cup \Lambda(\alpha)$ such that both $a_2 \rightarrow c$ and $c \rightarrow a_1$.
\end{lemma}
\begin{proof}
Suppose, seeking a contradiction, that there exists $c \in \Gamma^*(\alpha) \cup \Lambda(\alpha)$ with $a_2 \rightarrow c$ and $c \rightarrow a_1$. If $c \in \Gamma^*(\alpha)$ then let $\{a_1', a_2' \} \subseteq \Gamma^*(\alpha)$ where $a_i$ is $\sim$-related to $a_i'$ for $i=1,2$. If $c \in \Lambda(\alpha)$ then let $\{a_1', a_2' \} \subseteq \Lambda(\alpha)=\Gamma(\beta)$ where $a_i$ is $\sim$-related to $a_i'$ for $i=1,2$. Also let $c' \in \Gamma(\alpha)$ be the unique vertex of $\Gamma(\alpha)$ to which $c$ is $\sim$-related. 
By Lemmas~\ref{Nisadigraph}--\ref{morerulingout}, $\Gamma(\alpha)$ is isomorphic to one of: $D_4$, $H_0$, or $\bar{K}_n[D_3]$ for some $n>1$.

For $i=1,2$, if $c\in\Gamma^*(\alpha)$ then Lemma~\ref{doubleedgeconfigurations} applies with $c', a_i\in\Gamma(\alpha)$ and $c,a_i'\in\Gamma^*(\alpha)$, while if $c\in\Lambda(\alpha)$ then the lemma applies with $c, a_i'\in\Gamma(\beta)$ and $c',a_i\in\Gamma^*(\beta)$. In both cases, by comparing the configuration on $\{c,c',a_i,a_i'\}$ with those in Lemma~\ref{doubleedgeconfigurations}, since $c \rightarrow a_1$ and $a_2\rightarrow c$, we see that both of the following hold:
\begin{itemize}
\item[(1)] either (a)\quad $a_1' \rightarrow c\ \&\ a_1 \rightarrow c'\ \&\ (c',a_1')\in\Lambda$, or (b)\quad $(c',a_1), (c,a_1')\in\Lambda$;  
\item[(2)] either (a)\quad $ c \rightarrow a_2'\ \&\ c' \rightarrow a_2\ \&\ (c',a_2')\in\Lambda$, or (b)\quad $(c',a_2), (c,a_2')\in\Lambda$.  
\end{itemize}

Suppose first that (1a) and (2a) hold. 
Then the isomorphism 
 $(\alpha, c', a_2) \mapsto (\alpha, a_1, c')$ is between rigid configurations 
and so  extends to an automorphism $g$, and we must have $c^g=a_1'$.
However, this is impossible since $a_2 \rightarrow c$ implies $c' = a_2^g \rightarrow c^g = a_1'$ while in (1a), $(c',a_1')\in\Lambda$.
 
Suppose next that (1b) and (2a) hold. In this case $a_1 \rightarrow a_2 \leftarrow c'$ (with $(c',a_1) \in \Lambda$) embeds into $\Gamma(\alpha)$ which is a contradiction since none of $D_4$, $H_0$, or $\bar{K}_n[D_3]$ embeds this digraph. 
Similarly if (1a) and (2b) hold then  $c' \leftarrow a_1 \rightarrow a_2$ (with $(c',a_2) \in \Lambda$) embeds into $\Gamma(\alpha)$, again a contradiction.

This leaves the case where (1b) and (2b) hold. Here, by 3-set-homogeneity and Lemma~\ref{2HomAction}, there exists $g\in G_\alpha$ with $(a_1,c')^g=(a_2,c')$, and we must have $c^g=c$. This implies that $a_2\rightarrow c$ and $a_2=a_1^g\leftarrow c^g=c$ both hold, whch is a contradiction. 
\end{proof}

\begin{lemma}\label{NOTD4}
If each $\equiv$-class of $M$ induces $K_3$, then $\Gamma(\alpha)\not\cong D_4$. 
\end{lemma}
\begin{proof}
Suppose to the contrary that $\Gamma(\alpha) \cong D_4$. There are two cases to consider depending on whether or not $M$ embeds $D_3$.

Suppose that $M$ embeds $D_3$. Let $a_1, a_2 \in \Gamma(\alpha)$ with $a_1 \rightarrow a_2$. Since $M$ embeds $D_3$, by set-homogeneity applied to arcs it follows that there exists $c \in M$ with $a_1 \rightarrow a_2 \rightarrow c \rightarrow a_1$. Since $\Gamma(\alpha) \cong D_4$ we know $c \not\in \Gamma(\alpha)$, and so by Lemma~\ref{structure}(ii) and (iv) it follows that $c \in \Gamma^*(\alpha) \cup \Lambda(\alpha)$. But this contradicts Lemma~\ref{speciald3s}. 

Now suppose that $M$ does not embed $D_3$. Therefore for any $a \in \Gamma(\alpha)$ and 
$b \in \Gamma^*(\alpha)$ either $(a,b) \in \Lambda$, $(a,b)\in \Delta\cup\Delta^*$, or $b \rightarrow a$.
Let $\Gamma(\alpha) = \{ a_1, a_2, a_3, a_4 \}$ with $a_i \rightarrow a_{i+1~({\rm mod}~4)}$, and $\Gamma^*(\alpha) = \{ b_1, b_2, b_3, b_4  \}$ with $b_i \rightarrow b_{i+1~({\rm mod}~4)}$, and with $(a_i,b_i) \in \Delta\cup\Delta^*$ for $i=1,2,3,4$ (that such $b_i$ exist follows from Lemmas~\ref{structure} and~\ref{doubleedgeconfigurations}). By Lemma~\ref{doubleedgeconfigurations}, and since $M$ does not embed $D_3$ by assumption, we have $b_1 \rightarrow a_4$, $b_1 \rightarrow a_3$ and $b_2 \rightarrow a_4$. But then $\{b_2, a_3, a_4  \} \subseteq \Gamma(b_1)$ with $a_3 \rightarrow a_4 \leftarrow b_2$, contradicting $\Gamma(b_1) \cong D_4$.
\end{proof}

\begin{lemma}\label{NOTKND3}
If each $\equiv$-class of $M$ induces $K_3$, then $\Gamma(\alpha)\not\cong \bar{K}_n[D_3]$ for any $n\geq2$. 
\end{lemma}
\begin{proof}
Suppose that $\Gamma(\alpha)$ is isomorphic to $\bar{K}_n[D_3]$ for some $n > 1$. Since $M$ embeds $D_3$ it follows that for all $a \in \Gamma(\alpha)$ we have $|\Gamma(a) \cap \Gamma^*(\alpha)| \geq 1$. 

Suppose that $|\Gamma(a) \cap \Gamma^*(\alpha)| > 1$. Then there exist distinct $d,d' \in \Gamma^*(\alpha)$ with $\alpha \rightarrow a \rightarrow d \rightarrow \alpha$ and $\alpha \rightarrow a \rightarrow d' \rightarrow \alpha$. Choosing $a_1, a_2 \in \Gamma(\alpha)$ with $a_1 \rightarrow a_2$, by set-homogeneity applied to arcs, mapping the arc $\alpha \rightarrow a$ to the arc $a_1 \rightarrow a_2$, shows that there  exist distinct vertices $e,e'$ with $a_1 \rightarrow a_2 \rightarrow e \rightarrow a_1$ and $a_1 \rightarrow a_2 \rightarrow e' \rightarrow a_1$. At most one of $e$ or $e'$ can belong to $\Gamma(\alpha)$ (since $\Gamma(\alpha) \cong \bar{K}_n[D_3]$), and hence one of $e$ or $e'$ belongs to $\Gamma^*(\alpha) \cup \Lambda(\alpha)$, contradicting Lemma~\ref{speciald3s}. 

We conclude that 

\begin{align}
\label{eqn_WeConcludeThat}
\mbox{$|\Gamma(a) \cap \Gamma^*(\alpha)| = 1$ for all $ a \in \Gamma(\alpha)$.} 
\end{align}

Now let $a_1, a_2, a_3 \in \Gamma(\alpha)$ with 
$a_1 \rightarrow a_2 \rightarrow a_3 \rightarrow a_1$, and let 
$b_1, b_2, b_3 \in \Gamma^*(\alpha)$ with $(a_i,b_i) \in \Delta\cup\Delta^*$ for
 $i=1,2,3$. Let $\{ d \} = \Gamma(a_1) \cap \Gamma^*(\alpha)$. We claim that 
$d \in \{ b_1, b_2, b_3 \}$. Indeed, suppose $d \not\in \{ b_1, b_2, b_3 \}$, say 
$d = b_1' \in \{ b_1', b_2', b_3' \} \subseteq \Gamma^*(\alpha)$ with 
$b_1' \rightarrow b_2' \rightarrow b_3' \rightarrow b_1'$ and 
$\{ a_1', a_2', a_3'  \} \subseteq \Gamma(\alpha)$ and 
$(a_i', b_i') \in \Delta\cup\Delta^*$ for $i=1,2,3$. Then by 3-set-homogeneity and Lemma~\ref{2HomAction} applied to the subdigraphs induced on $\{\alpha,a_1,a_1\}$ and $\{\alpha, a_1, a_2'\}$, there 
is an automorphism $g \in G_{\alpha,a_1}$ with ${a_1'}^g = a_2'$. But then 
${b_1'}^g = b_2'$ and so as $a_1 \rightarrow d=b_1'$, we have  
$a_1 = a_1^g \rightarrow {b_1'}^g = b_2'$, contradicting \eqref{eqn_WeConcludeThat}.

Therefore with $\{ d \} = \Gamma(a_1) \cap \Gamma^*(\alpha)$ we have $d \in \{ b_1, b_2, b_3 \}$. 
By Lemma~\ref{doubleedgeconfigurations}, $d=b_3$. Now by Lemma~\ref{doubleedgeconfigurations}, we find  $\Gamma(b_1) \cap \{ a_1, a_2, a_3 \} = \varnothing$. 

Now since $\Gamma(b_1) \cong \bar{K}_n[D_3]$ we have $|\Gamma(b_1)| = 3n$. Since
 $\Gamma(b_1) \cap \{ a_1, a_2, a_3 \} = \varnothing$ it follows that
 $|\Gamma(b_1) \cap \Gamma(\alpha)| \leq 3n-3$. Now if $|\Gamma(b_1) \cap \Gamma(\alpha)| = 3n-3$ then
 since $\Gamma(\alpha) \cong \bar{K}_n[D_3]$ and $\Gamma(b_1) \cap \{ a_1, a_2, a_3 \} = \varnothing$ there 
would exist $\{ a_1', a_2', a_3'  \} \subseteq \Gamma(\alpha) \cap \Gamma(b_1)$ with $a_i' \rightarrow a_{i+1~({\rm mod}~3)}'$. But then the digraph induced by $\{ \alpha, a_1', a_2', a_3' \} \subseteq \Gamma(b_1)$ would contradict $\Gamma(b_1) \cong  \bar{K}_n[D_3]$. 

We conclude that $|\Gamma(b_1) \cap \Gamma(\alpha)| < 3n-3$ which, since $\Gamma(b_1) \cap \{\alpha, \beta, \gamma\} = \{ \alpha \}$,  implies that $|\Gamma(b_1) \cap \Lambda(\alpha)| \geq 2$. But then by Lemma~\ref{structure}(ii) we have $b_1 \in \Gamma^*(\alpha) = \Gamma(\gamma)$ and $|\Gamma(b_1) \cap \Gamma^*(\gamma)| = |\Gamma(b_1) \cap \Lambda(\alpha)| \geq 2$, which contradicts \eqref{eqn_WeConcludeThat} above. 
\end{proof}

\begin{lemma}\label{27VertexExample}
If each $\equiv$-class of $M$ induces $K_3$, then 
$(\Gamma(\alpha),M) \cong (H_0,H_3)$. 
\end{lemma}

\begin{proof}
Since $M$ embeds $P_3$, it follows by Lemmas~\ref{Nisadigraph}--\ref{NOTKND3} that that $\Gamma(\alpha) \cong \Gamma^*(\alpha) \cong \Lambda(\alpha) \cong H_0$. 

Let $a \in \Gamma(\alpha)$, let $\{u,v\} \subseteq \Gamma(\alpha) \cap \Gamma(a)$, and $\{w,x\} \subseteq \Gamma(\alpha) \cap \Gamma^*(a)$. Also let $a' \in \Gamma^*(\alpha)$ be the unique vertex $\sim$-related to $a$. Similarly we define $w'$, $x'$, $u'$, and $v'$ in $\Gamma^*(\alpha)$ which are $\sim$-related to $w, x, u$ and $v$ respectively. By Lemma~\ref{doubleedgeconfigurations} either (i) $u' \rightarrow a$ and $(a',u) \in \Lambda$, or (ii) $u \rightarrow a'$ and $(a,u') \in \Lambda$. 

We now rule out possibility (ii). Indeed, if (ii) holds then we  have $a \to w'$ 
(since there is an automorphism
taking $(\alpha, a, u)$ to $(\alpha, w, a)$, and this takes $a'$ to $w'$). Similarly, $a \to x'$. Thus, 
  $|\Gamma(a) \cap \Gamma^*(\alpha)| \geq 2$ and hence the arc $\alpha \rightarrow a$ extends 
to a copy of $D_3$ in at least $2$ distinct ways. Thus by set-homogeneity applied to arcs, 
the arc $w \rightarrow a$ extends to $D_3$ in at least $2$ ways, and since in $H_0$ every arc extends to a unique copy of $D_3$ it follows that there exists $y \in \Gamma^*(\alpha) \cup \Lambda(\alpha)$ such that $\{ w,a,y \}$ induces a copy of $D_3$, and this contradicts Lemma~\ref{speciald3s}. Hence configuration (ii) cannot arise.

Since $H_0$, and therefore also $M$, embeds $D_3$ it follows that there exists $b' \in \Gamma^*(\alpha)$ with $a \rightarrow b'$. Let $a' \in \Gamma^*(\alpha)$ with $(a,a') \in \Delta\cup\Delta^*$ and $b \in \Gamma(\alpha)$ with $(b,b') \in \Delta\cup\Delta^*$. If there were an arc between $a$ and $b$ then by Lemma~\ref{doubleedgeconfigurations}, we would have $b\rightarrow a$, but then we would have configuration (ii) (with $(b,a)$ in place of $(a,u)$). It follows that $(a,b)\in\Lambda$ and then from Lemma~\ref{doubleedgeconfigurations} that  $(a',b') \in \Lambda$ and $b \rightarrow a'$. Since $G_{\alpha}$ is transitive on independent 2-sets from $\Gamma(\alpha)$ it follows that for all $c,d \in \Gamma(\alpha)$ with $c$ and $d$ unrelated the digraph induced by $\{ c,d,c',d'\}$ is isomorphic to that induced by $\{a,b,a',b'\}$. 

These observations determine all the relations between $\Gamma(\alpha)$ and $\Gamma^*(\alpha)$. Applying automorphisms mapping $\alpha$ to $\beta$, and separately $\alpha$ to $\gamma$, this determines all the other relations in the structure $M$. 
(That is to say, such a structure is unique if it exists.)
It can be checked that  the $27$ vertex set-homogeneous digraph $H_3$ described in Section~\ref{sec_27vertex} satisfies the conditions of this lemma. Thus, $M\cong H_3$. \end{proof}

This completes the proof of Theorem~\ref{maintheorem}(iii) in the case $\Gamma(\alpha)\cong\Gamma^*(\alpha)$.

\section{Proof of Theorem~\ref{maintheorem}(iii): case $\Gamma(\alpha)\not\cong\Gamma^*(\alpha)$}
\label{sec_5}

In this section we complete the proof of Theorem~\ref{maintheorem}(iii) and therefore also the proof of Theorem~\ref{maintheorem}. We argue inductively as in the proof of Theorem~\ref{maintheorem}(ii). 
Let $M$ be a finite set-homogeneous s-digraph. As $\Gamma(\alpha)$ and $\Gamma^*(\alpha)$ are set-homogeneous, by induction we may assume that 
they belong to the list in Theorem~\ref{maintheorem}(iii).
From the main result of Section~\ref{sec_CaseInNEqualsOutN}, along with Lemmas~\ref{disconnect}, \ref{isom} and \ref{samenbours}(ii),
 it follows that we may suppose that the following hold. 
\begin{enumerate}[(I)]
\item The $\Gamma$-digraph of $M$ is connected. \label{11}
\item The digraphs $\Gamma(\alpha)$ and $\Gamma^*(\alpha)$ both belong to the list in Theorem~\ref{maintheorem}(iii), and 
$\Gamma(\alpha)$ is not isomorphic to $\Gamma^*(\alpha)$. In particular $|\Gamma(\alpha)|>1$. \label{22}
\item If $\alpha\neq \beta$ then  $\Gamma(\alpha)\neq \Gamma(\beta)$ and $\Gamma^*(\alpha)\neq \Gamma^*(\beta)$. \label{33}
\end{enumerate}

Ultimately we shall prove that there are in fact no finite set-homogeneous s-digraphs satisfying  $\Gamma(\alpha)\not\cong\Gamma^*(\alpha)$. 

\smallskip

\begin{quote}
\begin{center}
\textit{Throughout this section we assume that (I), (II) and (III) all hold.} 
\end{center}
\end{quote}

\smallskip

For any vertex-transitive digraph $\Sigma$ and $\beta\in \Sigma$, let $\Sigma^+$ denote the isomorphism type of
$\{x \in \Sigma: \beta\rightarrow x\}$ and $\Sigma^-$ denote that of $\{x\in \Sigma: x\rightarrow \beta\}$.
 
\begin{lemma}\label{DHom}
Let $M$ be a finite set-homogeneous digraph. Then we have the following. 
\begin{enumerate}[(i)]
\item $|\Gamma(\alpha)| = |\Gamma^*(\alpha)|$, and if one of $G_\alpha^{\Gamma(\alpha)}$, $G_\alpha^{\Gamma^*(\alpha)}$ is $2$-homogeneous, then they both are.
\item $\Gamma(\alpha)^-\cong \Gamma^*(\alpha)^+$.
\end{enumerate}
\end{lemma}
\begin{proof}
Part (i) follows from the fact that $\Gamma(\alpha)$ and $\Gamma^*(\alpha)$  are paired suborbits, along with 
 \cite[p. 230]{cameron0} (see also \cite[Exercise~2.18]{cameron3} for the case when one of $G_\alpha^{\Gamma(\alpha)}$ or $G_\alpha^{\Gamma^*(\alpha)}$ is  2-transitive). Part (ii) was proved above in Lemma~\ref{2verts}. 
\end{proof}

A common theme below will be searching for subdigraphs $X$ of $\Gamma(\alpha)$ which are maximal subject
 to having isomorphic copies in $\Gamma^*(\alpha)$. By set-homogeneity, such $X$
will lie in $\Gamma^*(\mu)$ for some $\mu$, and if there are many copies of $X$ counting arguments are available.

\begin{lemma}\label{NullNeighbours}
If $\Gamma(\alpha) \not\cong \Gamma^*(\alpha)$ then neither $\Gamma(\alpha)$ nor $\Gamma^*(\alpha)$ is isomorphic to $K_n$ or to $\bar{K_n}$. 
\end{lemma}
\begin{proof}
Suppose that one of $\Gamma(\alpha)$ or $\Gamma^*(\alpha)$ is isomorphic to $K_n$ or $\bar{K}_n$.

In this case, one of $G_\alpha^{\Gamma(\alpha)}$, $G_\alpha^{\Gamma^*(\alpha)}$ is 2-homogeneous, so both are, by Lemma~\ref{DHom}. 
 Also, $P_3$ does not embed in $M$, so $\Gamma^*(\alpha)$ has no arcs and 
 $\Gamma(\alpha)$ and $\Gamma^*(\alpha)$ {\em both} lie in ${\cal L}$. By Theorem~\ref{maintheorem}(i) since they are not isomorphic, one of them is isomorphic to $K_n$ and 
the other to $\bar{K}_n$. 

Without loss of generality suppose that
$\Gamma(\alpha)\cong K_n$ and $\Gamma^*(\alpha) \cong \bar{K_n}$. Now, as $\Gamma^*(\alpha)$ has no $\Delta$-edges, er have:

\begin{itemize}
 \item[(a)]  if $\beta_1,\beta_2\in \Gamma(\alpha)$ are distinct then 
$\Gamma(\beta_1) \cap \Gamma(\beta_2)=\varnothing$; 
\item[(b)] likewise, if $\gamma_1,\gamma_2\in \Gamma^*(\alpha)$ are distinct, then
$\Gamma^*(\gamma_1) \cap \Gamma^*(\gamma_2)=\varnothing$;
\item[(c)] furthermore, if $\beta \in \Gamma(\alpha)$ then $|\Gamma(\beta) \cap
\Gamma^*(\alpha)|\leq 1$, and $\Gamma^*(\beta)\cap\Gamma^*(\alpha)=\varnothing$.
\end{itemize}

\

\noindent {\em Claim~1.} Both $\sim$-related pairs and independent pairs come from self-paired orbitals.

\begin{proof}[Proof of Claim~1.] We prove this for $\sim$-related pairs; the proof for unrelated pairs is similar. By set-homogeneity $G_\alpha$ is highly homogeneous on $\Gamma(\alpha)$, 
so by Fact~\ref{livwag} we may suppose that $|\Gamma(\alpha)|=3$. 
Pick distinct $\beta,\gamma,\delta \in \Gamma(\alpha)$, and let $\Gamma(\beta)=\{\beta_1,\beta_2,\beta_3\}$ and
$\Gamma(\gamma)=\{\gamma_1,\gamma_2,\gamma_3\}$. 

First suppose that $\beta\sim\gamma_1$ and $\beta\sim\gamma_2$ with $\gamma_1 \neq \gamma_2$.
Then $\{\beta,\gamma_1,\gamma_2\}$ induces $K_3$, so equals $\Gamma(\epsilon)$ for some $\epsilon \neq \gamma$ (by 3-set-homogeneity) and we have $|\Gamma(\alpha)\cap\Gamma(\epsilon)|=1$. Thus, $\gamma_1,\gamma_2 \in \Gamma(\gamma)\cap \Gamma(\epsilon)$,
so by 2-set-homogeneity, for any 2-subset $X$ of $\Gamma(\alpha)$, there is $\eta\ne\alpha$ with $\Gamma(\eta) \cap \Gamma(\alpha)\supseteq X$. In fact, by property (III), $\Gamma(\eta) \cap \Gamma(\alpha) = X$, and distinct $2$-subsets of $\Gamma(\alpha)$ give rise to distinct  $\eta$. Thus there are distinct such $\eta_1,\eta_2$ with $\eta_1\rightarrow\beta, \eta_2\rightarrow\beta$, and hence the 4-set $\{\alpha,\epsilon,\eta_1,\eta_2\}\subseteq\Gamma^*(\beta)$, contradicting $|\Gamma^*(\beta)|=3$. Thus  $\beta$ is $\sim$-related to at most one $\gamma_i$. 

Now $|\Gamma^*(\beta) \cap \Gamma(\gamma)|\leq 1$ (since $\Gamma^*(\beta)$ has no $\Delta$-edges), and $\Gamma(\beta)\cap\Gamma(\gamma)=\varnothing$, so there is at least one  element of $\Gamma(\gamma)$, say 
$\gamma_1$, which is $\Lambda$-related to $\beta$. Likewise, there is an element of $\Gamma(\beta)$, say 
$\beta_1$ which is $\Lambda$-related to $\gamma$. The isomorphism $(\beta, \gamma, \gamma_1)\rightarrow (\gamma, \beta, \beta_1)$ of these rigid subdigraphs extends to 
an automorphism of $M$ which interchanges $\beta$ and $\gamma$ and hence  $\Delta$ is self-paired.
\end{proof}

\noindent \textit{Claim~2.} There exist $a \in \Gamma(\alpha)$ and $b \in \Gamma^*(\alpha)$ with $(a,b) \in \Delta$, if and only if there exist $a' \in \Gamma(\alpha)$ and $b' \in \Gamma^*(\alpha)$ with $(a',b') \in \Lambda$.  

\begin{proof}[Proof of Claim~2.]
Let $a \in \Gamma(\alpha)$ and $b \in \Gamma^*(\alpha)$ and suppose that $a \sim b$. By transitivity on $\sim$-related pairs, there exists $\gamma$ with $\{ a, b \} \subseteq \Gamma(\gamma)$. Since $\{ \gamma, \alpha \} \subseteq \Gamma^*(a)$ it follows that $(\gamma, \alpha) \in \Lambda$. But $\gamma \in \Gamma^*(b)$ and $\alpha \in \Gamma(b)$, so for any automorphism $g$ of $M$ taking $b$ to $\alpha$ we have $\gamma^g\in\Gamma^*(\alpha), \alpha^g\in\Gamma(\alpha)$ and $(\gamma^g,\alpha^g)\in\Lambda$.

The converse is proved using a dual argument.  
\end{proof}

\noindent {\em Claim~3.} If $\beta\in\Gamma(\alpha)$, then $\Gamma(\beta)\cap\Gamma^*(\alpha)=\varnothing$ 
(that is, $M$ does not embed $D_3$). 

\begin{proof}[Proof of Claim~3.]  Suppose to the contrary that $\beta\in\Gamma(\alpha)$ and $\Gamma(\beta)\cap\Gamma^*(\alpha)\ne\varnothing$. Since, as we observed in (c) before Claim 1, $|\Gamma(\beta)\cap\Gamma^*(\alpha)|\leq1$, it follows that the arcs  
determine a matching between $\Gamma(\alpha)$ and $\Gamma^*(\alpha)$. Also by observation (c), $|\Gamma^*(\alpha) \cap (\Gamma(\beta)\cup \Gamma^*(\beta))|\leq 1$, and hence as $|\Gamma^*(\alpha)|>1$, it follows by Claim 1 that $(\Lambda(\beta) \cup\Delta(\beta))\cap \Gamma^*(\alpha)\neq \varnothing$. Thus,
 by Claims 1 and 2, whenever $(u,v)\in \Delta \cup \Lambda$, 
there is $w$ with $u\rightarrow w\rightarrow v$.

Now if $|\Gamma(\alpha)| \neq 3$ then $G_{\alpha}$ is
doubly transitive on $\Gamma(\alpha)$ and $\Gamma^*(\alpha)$ by Fact~\ref{livwag}, so $G_{\alpha\beta}$ fixes $\beta'$, where $\{\beta'\}=\Gamma^*(\alpha)\cap\Gamma(\beta)$, and is transitive on both $\Gamma(\alpha)\setminus\{\beta\}$ and $\Gamma^*(\alpha)\setminus\{\beta'\}$. Since $\Delta(\beta)\cap \Gamma^*(\alpha)\neq \varnothing$,  
 it follows that there is a $\Delta$-edge between $\beta$ and every
vertex in $\Gamma^*(\alpha)\setminus\{\beta'\}$. But
this is true of every vertex in $\Gamma(\alpha)$ and
that contradicts Claim 2. 

Now suppose $|\Gamma(\alpha)| = 3$, say
$\Gamma(\alpha) = \{a,b,c\}$.  Let $A = \Gamma(a)
\setminus \Gamma^*(\alpha)$ and note that (by Claim 2 with $a$ replacing $\alpha$) $|A|=2$, $\alpha$ is $\sim$-related to one vertex of $A$, and
$\alpha$ is unrelated to the other vertex
of $A$. The same is true for each of the sets $\Gamma(b)
\setminus \Gamma^*(\alpha)$ and $\Gamma(c) \setminus
\Gamma^*(\alpha)$, and the three sets $\Gamma(a)\setminus\Gamma^*(\alpha)$, $\Gamma(b)\setminus \Gamma^*(\alpha)$ and $\Gamma(c)\setminus\Gamma^*(\alpha)$ are pairwise disjoint by observation (a) above.  It follows that $|\Delta(\alpha)| =
3$, since $\Delta(\alpha) \subseteq \Gamma \circ \Gamma(\alpha)$ and we have  described all elements of $\Gamma \circ \Gamma(\alpha)$.  As $\Delta$ is self-paired it follows that $G_a$
is transitive on $\Delta(a)$.
However, $\Delta(a)$ consists of a disjoint union of the $\Delta$-edge $\{b,c\}$ together with a point $d$ of $\Gamma^*(\alpha)$.
 This is a
contradiction since $\Delta(a)$ is  not a
transitive digraph, since $d$ is $\sim$-related to at most one element of $\Gamma(\alpha)$.
\end{proof}

To finish the proof of the lemma, by Claim 3, each pair $\{a,b\}$ with
 $a\in\Gamma(\alpha)$ and $b\in\Gamma^*(\alpha)$ is $\sim$-related or unrelated,
and by Claims 2 and 3, $\Gamma\circ\Gamma(\alpha)$ contains both $\Delta(\alpha)$ and $\Lambda(\alpha)$. It follows using set-homogeneity fixing $\alpha$ 
that the bipartite graph ``between'' $\Gamma(\alpha)$ and $\Gamma^*(\alpha)$ is a set-homogeneous bipartite graph. (This bipartite graph is obtained by dropping the edges within $\Gamma(\alpha)$, and the two parts are `coloured' as in the next paragraph.)
Using the argument of Enomoto (see \cite{enomoto} and also the proof of Lemma~\ref{sethomtourn}) this implies that this 
bipartite graph is actually a homogeneous bipartite graph. 

The finite homogeneous bipartite graphs were classified in \cite{goldstern}. They are the complete bipartite graph, the null bipartite graph (with no edges), the perfect matching, and the `complement of perfect matching' bipartite graph (with vertex set $X \cup Y$ and edge set $\{ \{ x, y \}: x \in X, y \in Y, y \neq f(x)  \}$ where $f:X \rightarrow Y$ is a fixed bijection). 
(Here, as in \cite{goldstern}, we view the two parts of the bipartition as `coloured'
 with distinct unary predicates.)
From Claim 3 
it follows that the bipartite graph ``between'' $\Gamma(\alpha)$ and $\Gamma^*(\alpha)$ is not isomorphic either to the complete bipartite graph 
$K_{n,n}$ or to the null bipartite graph. Thus the only remaining possibility is the case of 
 a perfect matching (or its complement). We deal with the former case. The latter case is then eliminated automatically, since it is complementary to the former. 

Put $\Gamma(\alpha) = \{ \beta_1, \ldots, \beta_n \}$ and $\Gamma^*(\alpha) = \{ \gamma_1, \ldots, \gamma_n \}$ 
with $\beta_i \sim \gamma_i$ for each $i$. Define $B_i = \Gamma(\beta_i)$ noting that $i \neq j$ implies 
$B_i \cap B_j = \varnothing$ by observation (a). The arc $\gamma_1 \rightarrow \alpha$ has the property that there is a unique vertex $\beta_1$ such that $\alpha \rightarrow \beta_1$ and $(\gamma_1, \beta_1) \in \Delta$. By set-homogeneity applied to arcs it follows that the same is true for each arc $\alpha \rightarrow \beta_i$ $(i=1,\ldots,n)$. It follows that for each set $B_i$ there is a unique $b_i \in B_i$ with $(\alpha, b_i) \in \Delta$ and hence, as $\Delta(\alpha) \subseteq \Gamma\circ\Gamma(\alpha)$, we have  $\Delta(\alpha) = \{ b_1, b_2, \ldots, b_n \}$. Note that $\alpha$ is unrelated to every $b \in B_i \setminus \{b_i \}$ for $i=1,\ldots,n$.

Now consider $\Delta(b_1)$.
The vertex $b_1$ is $\sim$-related to $\alpha$, and to all of the other $n-1$ vertices of $B_1$. So the digraph induced 
by $\Delta(b_1)$ is isomorphic to a disjoint union of $K_1$ and $K_{n-1}$. But 
$\Delta(b_1)$ is a vertex transitive graph (since $\Delta$ is self-paired by Claim~1) so the only possibility is $n=2$.

Now consider the case $n=2$. Then by the last paragraph,
 $|\Delta(\alpha)|=2$, and similarly $|\Lambda(\alpha)|=2$, so $|M|=9$. Let $N$ be the undirected $\Delta$-subgraph of 
$M$ (that is, the graph obtained by replacing every arc with an independent pair). Now $N$ is a vertex transitive graph all of whose vertices 
have degree $2$. Thus $N$ is a disjoint union of isomorphic cycles. By considering $\Gamma(\alpha) \cup \Gamma^*(\alpha)$ we
see that $N$ embeds a path with $4$ vertices (namely $\gamma_1,\beta_1,\beta_2,\gamma_2$), hence each of the cycles of $N$ must have at least $5$ vertices. Since $N$ has 
exactly $9$ vertices, there can only be one cycle, hence $N$ is isomorphic to a $9$-cycle. Label the vertices $v_1, \ldots, v_9$
so that $v_i \sim v_{i+1~({\rm mod}~9)}$. 
Now as $\Delta$ is self-paired,  $G_{v_1}$ is transitive on $\Delta(v_1)=\{v_2,v_9\}$, so
 $(v_2, v_9) \in \Lambda$. Let $g\in G_{v_1}$ interchange $v_2$ and $v_9$.
Since $\Gamma(v_1) \cong K_2$, we have 
$\Gamma(v_1) = \{ v_i, v_{i+1} \}$ for some $i \in \{3,4,\ldots, 7\}$. Since $g$ fixes $v_1$ we have $\Gamma(v_1)^g = \Gamma(v_1)$.
However, as $g$ lies in the dihedral group of order 18, and reflects the cycle in $v_1$, the only edge fixed by $g$ is $\Gamma(v_1)=\{v_5,v_6\}$. Under the rotation of $N$ mapping $v_1$ to $v_5$ we see that $\Gamma(v_5)=\{v_1,v_9\}$. However we also have $v_1\in \Gamma^*(v_5)$, a contradiction. 

This completes the proof of the lemma. 
\end{proof}

\begin{lemma}
If  $\Gamma(\alpha) \not\cong \Gamma^*(\alpha)$ and $\Gamma(\alpha)$ is isomorphic to either $K_m[\bar{K_n}]$ or its complement (for some $m,n > 1$), then $\Gamma^*(\alpha) \cong \bar{K_r}[K_s]$ (or its complement) for some $r,s >1$. 
\end{lemma}
\begin{proof}
Suppose that $\Gamma(\alpha)$  has the form $K_m[\bar{K}_n]$, for $m,n>1$.
In this case, $M$ again embeds no copies of $P_3$, so both $\Gamma(\alpha)$ and $\Gamma^*(\alpha)$ are set-homogeneous graphs, and hence are homogeneous by Enomoto's result. By 
Lemma~\ref{NullNeighbours}, neither is complete nor independent. As $|\Gamma(\alpha)|=mn$, it is composite, so $\Gamma^*(\alpha)\not\cong C_5$. 
So by Theorem~\ref{maintheorem}(i) either $\Gamma^*(\alpha) \cong K_3\times K_3$,
 or $\Gamma^*(\alpha)$ also has form $K_r[\bar{K}_s]$ or its complement. Note here that $K_3\times K_3$ is isomorphic to its (graph) complement.

We now rule out the first possibility: $\Gamma(\alpha)\cong K_m[\bar{K}_n]$, $\Gamma^*(\alpha)\cong K_3 \times K_3$. Here $m=n=3$. The graph $K_3$ 
is maximal subject to embedding in both $\Gamma(\alpha)$ and $\Gamma^*(\alpha)$. 
Fix $\delta\in \Gamma(\alpha)$ and consider copies of $K_3$ in $\Gamma(\alpha)$ containing $\delta$ -- namely, those
 transversals of the three $\Gamma(\alpha)$-blocks
which contain
$\delta$. There are 9 of these, each one dominates a vertex $\beta$ (by 3-set-homogeneity), and different copies dominate different $\beta$, since $K_3$ is maximal with respect to embedding in both $\Gamma(\alpha)$ and $\Gamma^*(\alpha)$. In
 addition, since  $\bar{K}_3$ embeds in $\Gamma^*(\alpha)$, there is $\beta'$ dominated by the copy of $\bar{K}_3$ in $\Gamma(\alpha)$ which contains $\delta$
 and it is different from the other $9$ vertices $\beta$ by consideration of substructures of $\Gamma^*(\alpha) \cong K_3 \times K_3$. 
Thus, $|\Gamma(\delta)|\geq 10$, which is impossible.

The case where $\Gamma(\alpha)\cong \bar{K}_m[K_n]$ is handled similarly: it leads to consideration of $\Gamma(\alpha)\cong \bar{K}_3[K_3]$ with $\Gamma^*(\alpha)\cong K_3 \times K_3$ -- and one 
considers copies of $\bar{K}_3$ in $\Gamma(\alpha)$ which contain $\delta$.
\end{proof}

The next lemma will be needed for some of the arguments that follow. It is straightforward, so the proof has been omitted. 

\begin{lemma}
\label{lem_basic2}
Let $M$ be a finite set-homogeneous s-digraph and let $T \subseteq M$ induce a copy of $D_3$ in $M$. Moreover, suppose that $G_{\alpha \beta} = 1$ where $G = \Aut(M)$ and $\alpha \rightarrow \beta$. Then the following are equivalent:  
\begin{enumerate}[(i)]
\item
the group induced by $G$ on $T$ has size $1$ (respectively has size $3$);
\item
for any pair $T'$ and $T''$ of copies of $D_3$ in $M$ exactly one (respectively every) isomorphism $\phi: T' \rightarrow T''$ extends to an automorphism of $M$;
\item
for every arc $\alpha \rightarrow \beta$ of $M$, $\Gamma(\beta) \cap \Gamma^*(\alpha)$ has size at least $2$ (respectively has size $1$). 
\item
for every arc $\alpha \rightarrow \beta$ of $M$, $\Gamma(\beta) \cap \Gamma^*(\alpha)$ has size $3$ (respectively has size $1$). 
\end{enumerate}
\end{lemma}
\begin{proof} The implications $(i)\Leftrightarrow (ii)$ and $(iv)\Rightarrow (iii)$ are immediate, and do not require the assumption $G_{\alpha\beta}=1$. 

$(i)\Rightarrow (iv).$ This does require the assumption that $G_{\alpha\beta}=1$, under either hypothesis. In practice, we work with the conditions in (ii) corresponding to those of (i). First, suppose the group induced by $G$ on $T$ has order $3$, and let $\gamma_1,\gamma_2\in \Gamma(\beta)\cap \Gamma^*(\alpha)$. Then, by the second clause of (ii), 
the isomorphism $(\alpha,\beta,\gamma_1)\mapsto (\alpha,\beta,\gamma_2)$ extends to an automorphism of $M$, which must be the identity as $G_{\alpha\beta}=1$. Hence $\gamma_1=\gamma_2$.

On the other hand, suppose the group induced on $T$ is trivial, and let $T=\{\alpha,\beta,\gamma\}$ with $\alpha\to \beta\to \gamma\to \alpha$. We may say that the arcs $\alpha\to \beta$, $\beta \to \gamma$ and $\gamma \to \alpha$ are arcs of types $1,2,3$ respectively {\em of $T$}. For any copy $T'=\{\alpha',\beta',\gamma'\}$ of $T$ with  $\alpha'\to\beta'\to \gamma'\to \alpha'$, there is a unique automorphism of $M$ taking $T$ to $T'$, so $T'$ has unambiguously a unique arc of each of the types $1,2,3$. Now since $\Aut(M)$  is transitive on arcs, there are $\gamma_1, \gamma_2,\gamma_3$ such that 
$\alpha\to\beta\to \gamma_i \to \alpha$ and the arc $\alpha\to \beta$ has type $i$ in the copy $T_i:=\{\alpha,\beta,\gamma_i\}$ of $D_3$ (for $i=1,2,3$). Thus, $|\Gamma(\beta)\cap\Gamma^*(\alpha)|\geq 3$. Also, for any $\delta$ such that $\alpha\to\beta\to\delta\to\alpha$, there is $i=\{1,2,3\}$ such that $\alpha\to \beta$ has type $i$ in the copy $T':=\{\alpha,\beta,\delta\}$ of $D_3$. Hence, by 3-set-homogeneity, some automorphism of $M$ induces $(\alpha,\beta,\delta)\mapsto (\alpha,\beta,\gamma_i)$, so $\delta=\gamma_i$ as $G_{\alpha\beta}=1$. Thus, $|\Gamma(\beta)\cap \Gamma^*(\alpha)|=3$.

$(iii)\Rightarrow (i).$ Again, under either hypothesis, this requires the assumption
 that $G_{\alpha\beta}=1$. If $|\Gamma(\beta)\cap \Gamma^*(\alpha)|=1$ then the group induced by $G$ on $T$ has size $3$, since otherwise it would have size $1$ and the direction $(i)\Rightarrow (iv)$ would imply $|\Gamma(\beta)\cap \Gamma^*(\alpha)| \geq2$. Likewise, if $|\Gamma(\beta)\cap \Gamma^*(\alpha)|\geq 2$ then the group induced by $G$ on $T$ is trivial, since otherwise it would have size three and  the direction $(i) \Rightarrow (iv)$ would imply $|\Gamma(\beta)\cap \Gamma^*(\alpha)| =1$. 
\end{proof}
\begin{lemma}
If $\Gamma(\alpha) \not\cong \Gamma^*(\alpha)$, then $\{\Gamma(\alpha), \Gamma^*(\alpha)\}\not \cong \{K_m[\bar{K_n}], \bar{K_r}[K_s]\}$ (with 
$m, n, r, s$ at least $2$). 
\end{lemma}
\begin{proof}
The proof divides into several cases. It is sufficient to deal with the case where
$(\Gamma(\alpha), \Gamma^*(\alpha))\not \cong (K_m[\bar{K_n}], \bar{K_r}[K_s])$

\

\noindent \textbf{Case 1:}
\noindent \textit{ $\Gamma(\alpha)\cong K_m[\bar{K}_n]$, $\Gamma^*(\alpha)\cong \bar{K}_r[K_s]$, with $m,s>2$, or with
$n,r>2$.}

\

\noindent We suppose $m,s>2$ -- the other case is similar.
Put $p:=\min\{m,s\}\geq 3$. Let $\delta\in \Gamma(\alpha)$, and consider sets $T$
with $\delta\in T\subset\Gamma(\alpha)$ such that $|T|=p$ and $T$ meets each $\Gamma(\alpha)$-block in at most one point.
Such a set $T$ induces $K_p$, so as $K_p$ embeds in $\Gamma^*(\alpha)$, there is $\beta$ with
 $T\subset \Gamma^*(\beta)$. Since $T$ is maximal subject to embedding in both $\Gamma(\alpha)$ and $\Gamma^*(\alpha)$, distinct sets $T$ give distinct  $\beta$. In addition, since $\overline{K_2}$ embeds in $\Gamma^*(\alpha)$, there is at least one  further
 $\beta'$ (not one of the above $\beta$) such that $\Gamma^*(\beta')$ contains $\delta$ and at least one other point in the same $\Gamma(\alpha)$-block 
as $\delta$. As $|\Gamma(\delta)|=mn=rs$ and there are 
$\binom{m-1}{p-1}n^{p-1}$ such sets $T$, we have
\[
\binom{m-1}{p-1} n^{p-1}+1\leq mn=rs.  
\]
In both the cases $p=m$ and $p=s$, this has no solutions, with the exception of the case $n=r=2$, $m=s=3$.
So  $\Gamma(\alpha)\cong K_3[\bar{K}_2]$ and $\Gamma^*(\alpha)\cong \bar{K}_2[K_3]$.
Here the four $K_3$ in $\Gamma(\alpha)$ containing $\delta$ give rise to four distinct elements of $\Gamma(\delta)$. Denote this set of four vertices by $X$.
We also fix $\epsilon\in \Gamma^*(\alpha)$.

By set-homogeneity, $G_\alpha$ acts transitively on the copies of $K_3$ in $\Gamma(\alpha)$ and so $X$ is contained in a $G_\alpha$-orbit. It follows that every vertex in $X$ relates to $\alpha$ in the same way. Let $x \in X$. Clearly we do not have $\alpha \rightarrow x$, since $\Gamma(\alpha)$ does not embed an arc. If $x \rightarrow \alpha $ then
$|\Gamma(\delta)\cap\Gamma^*(\alpha)|\geq |X|=4$; this is impossible, as $\Gamma(\delta)\cong \Gamma(\alpha)$, and the largest structure which embeds in both $\Gamma(\alpha)$ and $\Gamma^*(\alpha)$ has size three.
  Therefore we must have either:
\begin{itemize}
\item
$X\subseteq \Lambda(\alpha)$ or $X\subseteq \Lambda^*(\alpha)$; or 
\item $X\subseteq \Delta(\alpha)$ or $X\subseteq \Delta^*(\alpha)$. 
\end{itemize}
In either case, by rigidity of the configurations $(\alpha, \delta, x)$ (where $x \in X$) we conclude that $G_{\alpha \delta}$ acts transitively on $X$, and that $X$ is equal to the intersection of $\Gamma(\delta)$ and one of the $G_\alpha$-orbits 
$\Lambda(\alpha), \Lambda^*(\alpha), \Delta(\alpha)$ or $\Delta^*(\alpha)$. We next prove that both $\Lambda$ and $\Delta$ are self-paired. 

Let $Y = \Gamma(\delta) \setminus X$, noting that $|Y|=2$. 
Conjugating $G_{\alpha \delta}$ by an element of $G$ mapping the arc $(\alpha, \delta)$ to $(\epsilon,\alpha)$, we see that $G_{\epsilon\alpha}$ has the same orbit structure on $\Gamma(\alpha)$ as $G_{\alpha \delta}$ on $\Gamma(\delta)$. Let $X'$ be the $G_{\epsilon\alpha}$-orbit of length 4 in $\Gamma(\alpha)\cong K_3[\bar{K_2}]$. Then $X'$ induces a $4$-element vertex transitive subgraph of $\Gamma(\alpha)$, and this forces $X'$ to be a union of two $\Gamma(\alpha)$-blocks $\{ \mu, \mu' \} \cup \{ \nu, \nu' \}$ (where $(\mu, \mu') \in \Lambda$, $(\nu, \nu') \in \Lambda$), and 
$X'\cong K_{2,2}$. There is an element of $G_{\epsilon\alpha}$ mapping $\mu$ to $\mu'$ and such an element must interchange $\mu$ and $\mu'$. It follows that $\Lambda$ is self-paired. 

Relabelling $\epsilon$ if necessary, we may suppose that $\delta\not\in X'$, so
$Y':=\Gamma(\alpha)\setminus X'=\{\delta, \delta'\}$. We claim that $\Delta$ is also self-paired. Suppose to the contrary that $\Delta\ne\Delta^*$. Then no element of $G_{\epsilon\alpha}$ interchanges the $\sim$-related pair $\{\mu, \nu\}$, so $G_{\epsilon\alpha}$ induces the cyclic group $\langle (\mu,\nu,\mu',\nu')\rangle$ on $X'$, and one of $(\mu,\nu), (\mu,\nu')$ lies in $\Delta$ and the other in $\Delta^*$. The same holds for $(\delta,\mu), (\delta,\mu')$, so no element of $G_{\epsilon \alpha \delta}$ maps $\mu$ to $\mu'$. However, $G_{\epsilon\alpha \delta}$ has index at most 2 in $G_{\epsilon\alpha}$ and hence the group it induces on $X'$ contains the involution $(\mu,\mu')(\nu,\nu')$. This contradiction proves that $\Delta$ is self-paired. 

By our comments above about $X$ we have also that $X'=\Gamma(\alpha)\cap\Lambda(\epsilon)$ or 
$X'=\Gamma(\alpha)\cap\Delta(\epsilon)$.
We claim further that $G_{\epsilon\alpha}$ is transitive on $Y'$. Suppose to the contrary that this is not so, so that the $G_{\epsilon\alpha}$-orbits in $\Gamma(\alpha)$ have lengths 1, 1, 4. In particular $G_{\epsilon\alpha}$ fixes $\delta$ and, since $|\Gamma(\alpha)|=|\Gamma^*(\alpha)|=6$, it follows that $G_{\epsilon\alpha}=G_{\epsilon\alpha\delta}=G_{\alpha\delta}$ and the groups induced by $G_\alpha$ on  $\Gamma(\alpha)$ and $\Gamma^*(\alpha)$ are permutationally isomorphic.
Thus the $G_{\epsilon\alpha}$-orbits in $\Gamma^*(\alpha)\cong \overline{K_2}[K_3]$ have lengths 1, 1, 4, and this is impossible. 
Therefore $G_{\epsilon\alpha}$ is transitive on $Y'$. 

For each $\sigma \in \Gamma^*(\alpha)$ let $B_\sigma$ denote the $G_{\sigma \alpha}$-orbit of $\Gamma(\alpha)$ with size $2$ (so $B_\epsilon = Y'$). 
Then for every $\sigma \in \Gamma^*(\alpha)$, $B_\sigma$ is an unrelated pair of $\Gamma(\alpha)$, and for every $h \in G_\alpha$ and $\sigma \in \Gamma^*(\alpha)$ we have $(B_\sigma)^h = B_{\sigma^h}$.
We claim that there exist distinct $\sigma, \sigma' \in \Gamma^*(\alpha)$ with $\sigma \sim \sigma'$ such that $B_\sigma = B_{\sigma'}$. Indeed, suppose otherwise. Then since $|\Gamma^*(\alpha)|=6$ and $\Gamma(\alpha)$ has only $3$ unrelated pairs, there must be $\sigma_1, \sigma_2 \in \Gamma^*(\alpha)$ with $B_{\sigma_1} = B_{\sigma_2}$ and $\sigma_1$ and $\sigma_2$ in distinct copies of $K_3$ in $\Gamma^*(\alpha)$. Let the two copies of $K_3$ in $\Gamma^*(\alpha)$ be $\{ \sigma_1, \sigma_1', \sigma_1'' \}$ and $\{\sigma_2, \sigma_2', \sigma_2'' \}$. By set-homogeneity there is an automorphism $g$ such that $\{ \alpha, \sigma_1, \sigma_2, \sigma_2' \}^g = \{ \alpha, \sigma_1, \sigma_2'', \sigma_2' \}$. By inspection of the substructure induced by these sets of vertices we see that $\alpha^g=\alpha$, $\sigma_1^g = \sigma_1$ and $\sigma_2^g \in \{ \sigma_2'', \sigma_2' \}$. Thus we have $\sigma_2 \sim \sigma_2^g$ and $B_{\sigma_2} = B_{\sigma_1} =  B_{\sigma_1^g} = (B_{\sigma_1})^g = (B_{\sigma_2})^g = B_{\sigma_2^g}$, and this contradiction proves the claim. 

Now we have shown that there exist distinct $\sigma, \sigma' \in \Gamma^*(\alpha)$ with $\sigma \sim \sigma'$ and $B_\sigma = B_{\sigma'}$. By 3-set-homogeneity (on substructures $\{\sigma,\sigma',\alpha\}$) it follows that 
\begin{equation*} \label{eqn_fact} \forall \tau_1, \tau_2 \in \Gamma^*(\alpha), \quad \tau_1 \sim \tau_2 \Rightarrow B_{\tau_1} = B_{\tau_2}. 
\end{equation*}
Since $\Gamma^*(\alpha)\cong \overline{K_2}[K_3]$ it follows that there are at most two distinct $B_\sigma$ for $\sigma\in\Gamma^*(\alpha)$. This is a contradiction since 
$G_\alpha$ acts transitively on the $3$ unrelated pairs of $\Gamma(\alpha)$, and hence each unrelated pair $B$ of $\Gamma(\alpha)$ occurs as $B_\sigma$ for some $\sigma \in \Gamma^*(\alpha)$.

\

Thus we have proved that $\min\{m,s\}=\min\{n,r\}=2$ and $mn=rs$. If $m=2$ then we must have $r=2$, and similarly if $n=2$ then we must have $s=2$. We complete the proof by considering three cases: all of $m,n,r,s$ are 2; $m=r=2$, $n=s>2$; and $n=s=2$, $m=r>2$.

\

\

\noindent \textbf{Case 2:} \textit{$m=n=r=s=2$.}

\

\noindent The following argument is a variant of  4.1 of \cite{lachlan}.

Suppose that $m=n=r=s=2$, so $\Gamma(\alpha) \cong K_2[\bar{K_2}]$ and $\Gamma^*(\alpha) \cong \bar{K_2}[K_2]$. In this case,
 by moving each vertex to another in the same
$G_\alpha$-block (in $\Gamma(\alpha)$ and $\Gamma^*(\alpha)$) we see that $\Lambda$ and $\Delta$ are both self-paired.
 For all $a \in \Gamma(\alpha)$ we have $|\Gamma(\alpha) \cap \Delta(a)| = 2$ and $|\Gamma(\alpha) \cap \Lambda(a)| = 1$.
 By set-homogeneity applied to arcs it follows that $|\Gamma(x) \cap \Delta(\alpha)| = 2$ and 
$|\Gamma(x) \cap \Lambda(\alpha)| = 1$ for all $x \in \Gamma^*(\alpha)$. Similarly, $|\Gamma^*(a) \cap \Delta(\alpha)|=1$
 and $|\Gamma^*(a) \cap \Lambda(\alpha)| = 2$ for all $a \in \Gamma(\alpha)$. Since $\Lambda$ is self-paired, $G_{\alpha}$ is transitive on 
$\Lambda(\alpha)$, so every vertex of $\Lambda(\alpha)$ is dominated by at least one vertex of $\Gamma^*(\alpha)$.
 So $\rightarrow$ gives a surjective map from $\Gamma^*(\alpha)$ to $\Lambda(\alpha)$ and it follows that 
$|\Lambda(\alpha)| \leq |\Gamma^*(\alpha)|=4$. Similarly $|\Delta(\alpha)| \leq 4$. We claim 
 that $|\Lambda(\alpha)| \in \{2,4\}$. 
 Indeed, $|\Lambda(\alpha)| \neq 1$, for as $\Gamma^*(\alpha)\cong \bar{K}_2[K_2]$, every element of $\Gamma^*(\alpha)$ has at least two non-neighbours. Also, the partition of $\Gamma^*(\alpha)$ given by the fibres of the $\rightarrow$ surjection from $\Gamma^*(\alpha)$ to $\Lambda(\alpha)$ must be preserved by all automorphisms $g \in G_\alpha$. 
 Since $G_\alpha$ is transitive on $\Gamma^*(\alpha)$ it follows that $|\Lambda(\alpha)| \in \{2,4\}$. 

If $|\Lambda(\alpha)| = 2$ then it follows that for all $a \in \Gamma(\alpha)$ we have
 $\Gamma^*(a) \cap \Lambda(\alpha) = \Lambda(\alpha)$, since $|\Gamma^*(a)\cap \Lambda(\alpha)|=2$. 
 Since $G_\alpha$ acts transitively on $\Gamma(\alpha)$ it follows that for any $\lambda \in \Lambda(\alpha)$ we
 have $\Gamma(\alpha) = \Gamma(\lambda)$ with $\alpha \neq \lambda$ contrary to assumption (III). 
 
Thus  $|\Lambda(\alpha)| = 4$, and similarly $|\Delta(\alpha)|=4$.
Therefore $|M|=17$ and $G$ acts transitively but not 2-transitively on vertices. By Burnside's Theorem (see \cite[p. 36]{cameron3}), $G$ is soluble and $G_\alpha$ is semiregular. As $G_\alpha$ has orbits of size 4, $G_\alpha=\mathbb{Z}_4$ and $|G|=17\times 4$.

Let $\epsilon\in \Gamma^*(\alpha)$. Then $\Gamma^*(\epsilon)\cap \Lambda(\alpha)$ has size at most one, 
for if it contained distinct points $\lambda_1,\lambda_2$ there would be an 
automorphism inducing $(\lambda_1, \epsilon, \alpha) \mapsto (\lambda_2, \epsilon, \alpha)$,
 which is impossible as $G_{\alpha\epsilon}=1$. 
Also $|\Gamma(\epsilon)\cap \Delta(\alpha)|=2$, so $|\Gamma^*(\epsilon)\cap \Delta(\alpha)|\leq 2$. Thus, as $|\Gamma^*(\epsilon)|=4$, we have $\Gamma^*(\epsilon) \cap \Gamma(\alpha)\neq \varnothing$. So $M$ embeds $D_3$, and
as $G_{\epsilon\alpha}=1$, Lemma~\ref{lem_basic2} applies. We cannot have $|\Gamma^*(\epsilon)\cap \Gamma(\alpha)|=3$ as $\Gamma(\alpha)$ and $\Gamma^*(\alpha)$ do not have three-vertex isomorphic substructures. Thus,  by Lemma~\ref{lem_basic2},
$G$ induces  a cyclic group of order three on each copy of $D_3$. In particular, 3 divides
$|G|$. This contradicts the last paragraph.

\

\noindent \textbf{Case 3:} \textit{$m=r=2$ and $n=s>2$.}

\

\noindent We have $\Gamma(\alpha)\cong K_2[\bar{K}_n]$ and $\Gamma^*(\alpha)\cong \bar{K}_2[K_n]$ with $n>2$. 
Pick $\delta\in \Gamma(\alpha)$. For each $\delta'\in \Gamma(\alpha)\setminus\{\delta\}$ there is $\beta$ with 
$\delta,\delta'\rightarrow \beta$. 
Furthermore, by counting arguments and consideration of the induced subdigraphs of $\Gamma^*(\alpha)$, distinct $\delta'$ yield distinct $\beta$, and $\beta$ is uniquely determined by 
$\delta$ and $\delta'$ (as otherwise $|\Gamma(\delta)|>2n$). Thus a set $X$ of size $2n-1$ of elements $\beta$ of $\Gamma(\delta)$ arises in this way, so as
$|\Gamma(\delta)|=2n$, $G_{\alpha\delta}$ has an orbit of size 1 in $\Gamma(\delta)$. Hence, if $\epsilon \in \Gamma^*(\alpha)$,
$G_{\epsilon\alpha}$ has an orbit of size 1, say $\{\epsilon'\}$, in $\Gamma(\alpha)$.  In particular, $G_{\epsilon\alpha}$ fixes each block (that is, copy of $\bar{K}_n$) of $\Gamma(\alpha)$. Also, by $(n+1)$-set-homogeneity, $G_{\alpha\delta}$ is $(n-1)$-homogeneous, and hence is transitive, on the block of $\Gamma(\alpha)$ not containing $\delta$. Similarly $G_{\alpha\delta}$ has an orbit of size $n$ on the subset $X$ of $\Gamma(\delta)$, and
 hence $G_{\epsilon\alpha}$ has an orbit of size $n$ on $\Gamma(\alpha)$. This must be 
the block of $\Gamma(\alpha)$ not containing $\epsilon'$.

As $P_3$ does not embed in $M$, $\Gamma(\epsilon) \cap \Gamma(\alpha)=\varnothing$. 
Now $\epsilon$ cannot be dominated by a block of $\Gamma(\alpha)$ (for $\Gamma^*(\epsilon)$ does not embed $\bar{K}_3$, as $\Gamma^*(\alpha)$ does not), 
so we may suppose that $\epsilon$ is unrelated or $\sim$-related
to all elements of some block $T$ of $\Gamma(\alpha)$, say the block containing $\delta$. Let $S$ be the block of $\Gamma^*(\alpha)$ containing $\epsilon$. Then by 3-set-homogeneity, $G_\alpha^S$ is 2-homogeneous, and hence primitive. 
Hence, as $G_{\alpha\{S\}\{T\}}$ is a normal subgroup of $G_{\alpha\{S\}}$ of index at most 2,
$G_{\alpha\{S\}\{T\}}$ is transitive on $S$.
It follows  that all elements of 
$S$ are related in the same way to all elements of $T$. Thus there are two cases to consider: either every vertex in $S$ is $\sim$-related to every vertex in $T$, or every vertex in S is unrelated to every vertex in $T$.

First suppose the former. Now for every $\epsilon'$ in $S$ there exists $\beta$ with $\epsilon', \delta \rightarrow \beta$. At least two distinct $\beta$ must arise in this way, since $S \cup \{ \delta \}$ induces a copy of $K_{n+1}$ which does embed into $\Gamma^*(\beta)$. Now  
since $|\Gamma(\delta)|=2n$ and $|X|=2n-1$ it follows that at least one such $\beta$ belongs to $X$. In other words (recall the definition of $X$), there exists $\delta' \in \Gamma(\alpha) \setminus \{ \delta \} $ and $\epsilon'\in S$ with $\{ \delta', \delta, \epsilon'  \} \subseteq \Gamma^*(\beta)$. If $\delta'\in T$ then the subdigraph induced by  $\{ \delta', \delta, \epsilon'  \}$ is a path of length 2 which does not embed into $\Gamma^*(\alpha) \cong \Gamma^*(\beta)$. Therefore $\delta'\not\in T$. We cannot have $\epsilon'$ $\sim$-related to $\delta'$ for otherwise by 3-set-homogeneity there would be an automorphism taking $(\epsilon', \alpha, \delta)$ to $(\epsilon', \alpha, \delta')$ so swapping the two $\Gamma(\alpha)$-blocks, a contradiction. Thus again
the substructure induced by  $\{ \delta', \delta, \epsilon'  \}$ does not embed into $\Gamma^*(\alpha) \cong \Gamma^*(\beta)$. This is  a contradiction. 

Now suppose that every element of S is $\Lambda$-related to every element of $T$. In this case the argument is given by interchanging the roles of $\epsilon$ and $\delta$ in the argument of the previous paragraph, and working with in-neighbourhood sets of vertices in $\Gamma^*(\alpha)$. So for each $\epsilon' \in \Gamma^*(\alpha) \setminus \{\epsilon\}$, there is a vertex $\beta$ with $\{ \epsilon, \epsilon'  \} \subseteq \Gamma(\beta)$, distinct $\epsilon'$ give rise to distinct $\beta$, and since $|\Gamma^*(\epsilon)|=2n$ we may argue that $\beta$ is uniquely determined by $\epsilon$ and $\epsilon'$. Thus the set $X'$ of all $\beta$ arising in this way has size $|X'| = 2n-1$. Arguing similarly to the previous paragraph, since $\Gamma(\epsilon)$ does not embed $\overline{K_{n+1}}$ it follows that there is a vertex $\beta$ with $\beta \rightarrow \epsilon, \delta$ and $\beta \in X'$. Thus there is a vertex $\epsilon' \in \Gamma^*(\alpha)$ with $\{  \epsilon', \epsilon, \delta \} \subseteq \Gamma(\beta)$, and regardless of whether or not $\epsilon'$ belongs to the same $\Gamma^*(\alpha)$-component as $\epsilon$ this gives a $3$-element substructure that cannot embed into $\Gamma(\alpha) \cong \Gamma(\beta)$, a contradiction.   

\

\noindent \textbf{Case 4:} $n=s=2$ and $m=r>2$. 

\

So $\Gamma(\alpha) \cong K_m[\bar{K_2}]$ and $\Gamma^*(\alpha) \cong \bar{K_m}[K_2]$. Fix $\delta \in \Gamma(\alpha)$: for $\delta' \in \Gamma(\alpha) \setminus \{ \delta \}$ there exists $\beta$ with $\delta, \delta' \rightarrow \beta$ and distinct $\delta'$ give distinct $\beta$. Now $\Lambda$ is self-paired, for if $\delta,\delta'$ are unrelated in $\Gamma(\alpha)$ then an automorphism mapping $(\alpha,\delta)$ to $(\alpha,\delta')$ must interchange $\delta,\delta'$. Likewise (arguing in $\Gamma^*(\alpha)$) $\Delta$ is self-paired.
If $\delta,\delta'$ are in distinct $\Gamma(\alpha)$-blocks (that is, in distinct copies of $\overline{K_2}$ in $\Gamma(\alpha)$) such $\beta$ is unique (otherwise $\Gamma(\delta)$ is too large),
 but if they are in the same block there could be two such $\beta$. Again, let $X$ (depending on 
$\alpha,\delta$) be the set of all points $\beta$ arising in this way.

\

\noindent \textbf{Case 4.1:} \textit{For every independent pair $\delta, \delta'\in \Gamma(\alpha)$, there is a unique $\beta$ with
 $\delta,\delta'\rightarrow \beta$, and for every $\sim$-related pair $\epsilon, \epsilon' \in \Gamma^*(\alpha)$ there is a unique $\beta'$ with $\beta' \rightarrow \epsilon, \epsilon'$.}

\

\noindent Let $\epsilon\in \Gamma^*(\alpha)$. As in Case~3, $G_{\alpha\delta}$ has an orbit of size one on $\Gamma(\delta)$, so $G_{\epsilon\alpha}$ has an orbit of size one on $\Gamma(\alpha)$. Thus $G_{\epsilon\alpha}$ has at least three orbits in $\Gamma(\alpha)$, with at least two of them of size 1.
By 3-set-homogeneity, $G_{\epsilon\alpha}$ is transitive on $\Gamma(\alpha) \cap \Lambda(\epsilon)$ and $\Gamma(\alpha) \cap \Delta(\epsilon)$, and since $\Gamma(\alpha)$ contains no arcs, $\Gamma(\alpha) \cap \Gamma(\epsilon)=\varnothing$. Thus $\Gamma(\alpha) \cap \Gamma^*(\epsilon)\ne\varnothing$ and as it cannot consist of the union of $m-1$ blocks of $\Gamma(\alpha)$ (since such a structure does not embed in $\Gamma^*(\alpha)$), we find that $\Gamma(\alpha) \cap \Gamma^*(\epsilon)$ is a singleton orbit of $G_{\alpha\epsilon}$, and the latter has orbits of sizes $1,1,2m-2$ on $\Gamma(\alpha)$.
This gives a $G_\alpha$-invariant matching between
 $\Gamma(\alpha)$-blocks and $\Gamma^*(\alpha)$-blocks.

If $\epsilon$ is $\sim$-related to $2m-2$ elements of $\Gamma(\alpha)$, choose $\delta$ in the latter set.
There is $\beta$ with $\epsilon,\delta\in \Gamma^*(\beta)$. Then $\beta\not\in \Gamma(\alpha)$ as 
the latter set has no arcs, and likewise $\beta\not\in \Gamma^*(\alpha)$. 
 We now claim that we may suppose that there is $\delta'\in (\Gamma^*(\beta)\cap 
\Gamma(\alpha))\setminus \{\delta\}$. 
Indeed, since each vertex of $\Gamma^*(\alpha)$ is $\sim$-related to $2m-2$ vertices of $\Gamma(\alpha)$, we see that   
the total number of pairs $(\epsilon^*, \delta^*) \in \Gamma^*(\alpha) \times \Gamma(\alpha)$ with $\epsilon^* \sim \delta^*$ is $2m(2m-2)$. On the other hand, since $G_\alpha$ acts transitively on $\Gamma(\alpha)$ it follows that every vertex of $\Gamma(\alpha)$ is $\sim$-related to the same number, $k$ say, of vertices of $\Gamma^*(\alpha)$. Therefore $2mk = 2m(2m-2)$ and 
we conclude that $\delta$ is $\sim$-related to $2m-2$ vertices of $\Gamma^*(\alpha)$. Now we can replace $\epsilon$ by any of the $2m-2$ points $\epsilon'$ of $\Gamma^*(\alpha)$ that are $\sim$-related to $\delta$, in each case obtaining a $\beta'$ with $\Gamma^*(\beta') \supseteq \{ \delta, \epsilon' \}$. By consideration of common substructures of $\Gamma(\alpha)$ and $\Gamma^*(\alpha)$ we conclude that at least two elements of $\Gamma(\delta)$ arise this way, so for at least one such element, $\delta'$ say, the corresponding $\beta'$ must lie in $X$, since $|X| = |\Gamma(\delta)|-1$.
This $\delta'$ is as claimed, and
we find that for all such $\delta'$, $\{\epsilon,\delta,\delta'\}$ carries a structure which does not embed in 
$\Gamma^*(\alpha)$, which is a contradiction.

If $\epsilon$ is unrelated to $2m-2$ elements of $\Gamma(\alpha)$ then, as in Case~3, we repeat the argument of the previous paragraph, with the roles of $\delta$ and $\epsilon$ interchanged, working with in-neighbour sets of vertices of $\Gamma^*(\alpha)$, and a set $X'$ of vertices each dominating a pair $\epsilon, \epsilon'$ from $\Gamma^*(\alpha)$. Here we make use of the assumption that pairs in the same $\Gamma^*(\alpha)$-block are dominated by a unique vertex. 

\

\noindent \textbf{Case 4.2:} \textit{For every independent pair $\delta, \delta'\in \Gamma(\alpha)$, there are at least two vertices $\beta, \beta'$ dominated by $\{ \delta, \delta' \}$, or for every $\sim$-related pair $\epsilon, \epsilon' \in \Gamma^*(\alpha)$ there are at least two vertices $\beta, \beta'$ dominating $\{ \epsilon, \epsilon' \}$.}

\

\noindent Suppose the former. The latter possibility may be dealt with using a similar argument. 

Let $\epsilon \in \Gamma^*(\alpha)$. Let $\{ \delta, \delta' \}$ be a $\bar{K_2}$-block of $\Gamma(\alpha)$ and let $\{ \beta, \beta' \}$ satisfy $\delta, \delta' \rightarrow \beta, \beta'$. The remaining $2m-2$ elements of $\Gamma(\delta)$ arise as the set of $\beta''$ dominated by pairs $\{ \delta, \delta'' \}$ with the $\delta''$ coming from $\Gamma(\alpha)$-blocks different from $\{ \delta, \delta' \}$. Let $Y = \Gamma(\delta) \setminus \{ \beta, \beta' \}$. Recall that $\Lambda$ and $\Delta$ are both self-paired (proved in the first paragraph of Case 4). 

We claim that $Y$ is a single $G_{\alpha \delta}$-orbit of size $2m - 2$, and moreover  either $Y\subseteq\Delta(\alpha)$ or $Y\subseteq\Lambda(\alpha)$. We show first that $Y$ is contained in a $G_\alpha$-orbit: for by $3$-set-homogeneity, $G_\alpha$ acts (unordered) edge-transitively on $\Gamma(\alpha)$, and for $\delta_1,\delta_2\in \Gamma(\alpha)\setminus \{ \delta, \delta' \}$, an element $g\in G_\alpha$ mapping the $\sim$-related pairs $\{\delta,\delta_1\}$ to $\{\delta,\delta_2\}$ takes $\beta_1$ to $\beta_2$, where $\{\beta_i\}= Y\cap\Gamma(\delta)\cap\Gamma(\delta_i)$. It follows that every $y \in Y$ relates to $\alpha$ in the same way. Let $y \in Y$ be arbitrary. Clearly we do not have $\alpha \rightarrow y$. If $y \rightarrow \alpha$ then $Y \subset\Gamma^*(\alpha)$ and hence $|\Gamma^*(\alpha)\cap\Gamma(\delta)|\geq 2m-2$. However, since $m>2$, the digraphs induced on $\Gamma^*(\alpha)$ and $\Gamma(\delta)\cong\Gamma(\alpha)$ do not contain isomorphic subdigraphs of size $2m-2$.  We conclude that either $y\sim\alpha$ for all $y \in Y$, or $y,\alpha$ are unrelated for all $y \in Y$. But then in either of these two  cases, by rigidity of the configurations $(\alpha, \delta, y)$ $(y \in Y)$ we conclude that $Y$ is a single $G_{\alpha \delta}$-orbit of size $2m-2$, and hence  either $Y\subseteq\Delta(\alpha)$ or $Y\subseteq\Lambda(\alpha)$. This completes the proof of the claim.

Now we study the pair $B=\{\beta,\beta'\}$ defined in the first paragraph. Since $B\subset\Gamma(\delta)$, $B$ is either $\sim$-related or unrelated. Suppose we have $B\subset\Gamma^*(\alpha)$. If $B$ is unrelated then  
every unrelated pair from $\Gamma^*(\alpha)$ is dominated by an unrelated pair from $\Gamma(\alpha)$, and distinct pairs are dominated by distinct pairs: this is a contradiction as $\Gamma(\alpha)$ and $\Gamma^*(\alpha)$ have different numbers of such pairs. 
We conclude that $B$ is $\sim$-related, but then $\Gamma^*(\beta) \cap \Gamma(\alpha) \not\cong \Gamma^*(\alpha) \cap \Gamma(\delta)$, contradicting 2-set-homogeneity. Thus $B\not\subset\Gamma^*(\alpha)$.

Next suppose that $\beta \in \Lambda(\alpha)$ and $\beta' \rightarrow \alpha$. Then the arc $\alpha \rightarrow \delta$ has the property that there is a unique vertex $\beta'$ such that $\delta \rightarrow \beta' \rightarrow \alpha$. On the other hand, for the arc $\beta' \rightarrow  \alpha$ there are two distinct vertices $\delta, \delta'$ satisfying $\alpha \rightarrow \delta \rightarrow \beta'$, $\alpha \rightarrow  \delta' \rightarrow \beta'$. This contradicts $2$-set-homogeneity.  Similarly we can rule out the case $\beta \in \Delta(\alpha)$ and $\beta' \rightarrow \alpha$. Thus $B\cap\Gamma^*(\alpha)=\varnothing$.

Now suppose that $\beta$ is contained in the same $G_\alpha$-orbit as $Y$ (either $\Delta(\alpha)$ or $\Lambda(\alpha)$). Then 3-set-homogeneity gives an automorphism mapping $(\alpha,\delta,y)$ to $(\alpha,\delta,\beta)$, which is a contradiction since $Y$ is fixed setwise by $G_{\alpha\delta}$. Thus either $B\subset\Delta(\alpha)$, $Y\subseteq\Lambda(\alpha)$, or $Y\subset\Delta(\alpha)$, $B\subseteq\Lambda(\alpha)$.

Suppose first that $B\subset \Lambda(\alpha)$. Then $\alpha\in \Lambda(\beta)\cap \Gamma^*(\delta)$ and it follows that $\Lambda(\delta)\cap \Gamma^*(\alpha)\neq \varnothing$. Thus we may suppose $\epsilon \in \Gamma^*(\alpha)$ satisfies $(\epsilon, \delta) \in \Lambda$. By set-homogeneity, $\{ \beta, \beta' \}$ is a $G_{\alpha \delta}$-orbit, and recall that in this case we have $Y\subset\Delta(\alpha)$. By arc-transitivity mapping $(\alpha, \delta)$ to $(\epsilon, \alpha)$ we conclude that $\Delta(\epsilon) \cap \Gamma(\alpha)$ is a $G_{\epsilon \alpha}$-orbit of size $2m-2$ (and $\epsilon$ is unrelated to the other two vertices of $\Gamma(\alpha)$). Since $m \geq 3$ and $|\Delta(\epsilon) \cap \Gamma(\alpha)| = 2m-2$ there must be at least one $\Gamma(\alpha)$-block $\{ \mu, \mu' \}$ satisfying $(\epsilon, \mu) \in \Delta$ and $(\epsilon, \mu') \in \Delta$. By $2$-set-homogeneity there exists $\nu$ such that $\mu, \epsilon \rightarrow \nu$.  

Note that $\delta \in \Gamma(\alpha)$ has the following property: for all $ \kappa \in \Gamma(\delta), |\Gamma^*(\kappa) \cap \Gamma(\alpha)| \geq 2$.

Since $G_\alpha$ acts transitively on $\Gamma(\alpha)$ it follows that $\mu$ also has this property, and hence $|\Gamma^*(\nu) \cap \Gamma(\alpha)| \geq 2$. In other words, there exists some $\mu'' \in \Gamma(\alpha)$ such that $\{ \mu'', \mu, \epsilon \} \subseteq \Gamma^*(\nu)$. There are two possibilities:
either $\mu'' = \mu'$ or $\mu'' \neq \mu'$, but in either case, regardless of the relationship between $\mu''$ and $\epsilon$, the substructure induced by $\{ \mu'', \mu, \epsilon \}$ contains at least two $\Delta$-edges and so does not embed in $\Gamma^*(\alpha) \cong \Gamma^*(\nu)$, which is a contradiction. We conclude that this subcase cannot happen.   

Finally, suppose that $B\subset\Delta(\alpha)$, $Y\subseteq\Lambda(\alpha)$. 
This is dealt with in much that same way as the case considered in the previous paragraphs. We find $\epsilon\in\Gamma^*(\alpha)$ such that $(\epsilon,\delta)\in\Delta$.
Easy arguments using 3-set-homogeneity show that 
$G_{\epsilon\alpha}$ has an orbit of size two on $\Gamma(\alpha)$, namely $\Gamma(\alpha)\cap\Delta(\epsilon)$. 
We may thus suppose that $\epsilon$ is $\sim$-related to $\delta \in \Gamma(\alpha)$. 
Arguing as above we conclude $\Delta(\epsilon) \cap \Gamma(\alpha) = \{ \delta, \delta' \}$ where $\{ \delta, \delta' \}$ is a $\Gamma(\alpha)$-block of $\Gamma(\alpha)$.  
There is $\mu$ with 
$\epsilon,\delta\in \Gamma^*(\mu)$, and again arguing as in the previous paragraph, there is $\delta'' \in \Gamma(\alpha)\setminus\{\delta\}$ with 
$\delta'' \rightarrow \mu$. 
It can then be checked that, wherever $\delta''$ lies in $\Gamma(\alpha)$, $\{\epsilon,\delta,\delta''\}$ is not isomorphic 
to a subdigraph of $\Gamma^*(\alpha)$, which is a contradiction.

\end{proof}

\begin{lemma}
If $M \in \mathcal{S}$ and $\Gamma(\alpha) \not\cong \Gamma^*(\alpha)$ then we cannot have $\Gamma(\alpha) \cong K_m[\bar{K_n}]$ ($m,n > 1$) and $\Gamma^*(\alpha) \cong K_r[\bar{K_s}]$ ($r,s > 1$). 
\end{lemma}
\begin{proof}

Suppose that $\Gamma(\alpha)= K_m[\bar{K}_n]$ and $\Gamma^*(\alpha)=K_r[\bar{K}_s]$. In this case, as 
$\Gamma(\alpha) \not\cong \Gamma^*(\alpha)$, $m\neq r$ and $n\neq s$. By reversing arcs if necessary, we may suppose $m<r$ and $n>s$.

Fix $\delta\in \Gamma(\alpha)$, and consider subsets $T$ of $\Gamma(\alpha)$ which contain $\delta$ and exactly 
$s$ vertices of each block in $\Gamma(\alpha)$. As such sets have isomorphic copies in $\Gamma^*(\alpha)$, each 
such $T\subset \Gamma^*(\mu)$ for some $\mu$, and $T\neq T'$ implies $\mu \neq \mu'$ (for $\Gamma^*(\mu)$ does not
contain a copy of $\overline{K_{s+1}}$). The number of such sets $T$ is ${\binom{n}{s}}^{m-1}{\binom{n-1}{s-1}}$ and as the 
corresponding $\mu$ lie in $\Gamma(\delta)$, we have
$$
{\binom{n}{s}}^{m-1}{\binom{n-1}{s-1}}\leq mn=rs.
$$
The only solution of this is $m=2$, $n=3$, $s=2$, $r=3$, so we consider this case. So $\Gamma(\alpha) \cong K_2[\overline{K_3}]$ and $\Gamma^*(\alpha) \cong K_3[\overline{K_2}]$. In this case the number of subdigraphs $T\cong K_{2,2}$ of $\Gamma(\alpha)$ is $3^2 = 9$,  and since $G_\alpha$ is transitive on the set of (unordered) $\Delta$-edges in $\Gamma(\alpha)$, the set $X$ consisting of those vertices
$\mu$ with $\Gamma^*(\mu)\cap \Gamma(\alpha)\cong K_{2,2}$ is a $G_\alpha$-orbit of size $3^2=9$, and contains $\Gamma(\delta)$. In particular $X\ne\Gamma(\alpha), \Gamma^*(\alpha)$ (as it has size $9$).

Now, $\Lambda$ is clearly self-paired, for if $\epsilon, \epsilon'$ are unrelated in $\Gamma^*(\alpha)$
 then an automorphism $(\alpha, \epsilon)\mapsto (\alpha, \epsilon')$ must swap them. Also, $\Delta$ is self-paired. For if 
$\delta\in \Gamma(\alpha)$ and $B,C$ are distinct 2-sets
 in the copy of $\overline{K_3}$ in $\Gamma(\alpha)$ which does not contain $\delta$ then there is $g\in G_{\alpha\delta}$ 
with
 $B^g=C$, so $G_{\alpha\delta}$ is transitive on $\overline{K_3}$. Now if $C=\{\gamma,\gamma'\}$ and $\delta'$ is in the same
 $\overline{K_3}$ as $\delta$, there is $g\in G_\alpha$ with $\{\gamma,\delta,\delta'\}^g=\{\delta,\gamma,\gamma'\}$. 
Combining the last two observations, there is $h\in G_\alpha$ with
 $(\gamma,\delta)^h=(\delta,\gamma)$, and $h$ flips a $\Delta$-edge.

Thus the set $X$ is either $\Delta(\alpha)$ or $\Lambda(\alpha)$. Let $\delta\in \Gamma(\alpha)$, $\epsilon\in \Gamma^*(\alpha)$. 
Since $\Gamma(\delta) \subseteq \Delta(\alpha)$ or $\Gamma(\delta) \subseteq \Lambda(\alpha)$,
we see by 3-set-homogeneity that $G_{\alpha\delta}$ is transitive on $\Gamma(\delta)$, so $G_{\epsilon\alpha}$ is transitive on $\Gamma(\alpha)$. 
If $X=\Delta(\alpha)$ then (by considering an automorphism taking $(\alpha, \delta)$ to $(\epsilon, \alpha)$) we see that
$\Delta(\epsilon)\cap \Gamma(\alpha)$ has size $6$, so equals $\Gamma(\alpha)$; that is, since $G_\alpha$ is transitive on $\Gamma^*(\alpha)$ and fixes $\Gamma(\alpha)$, every vertex of 
$\Gamma(\alpha)$ is $\sim$-related to each vertex in $\Gamma^*(\alpha)$. Now $\Delta(\delta)$, which has size 9 and is
 vertex transitive (as it is a $G_\delta$-orbit), contains three vertices from $\Gamma(\alpha)$ and six from $\Gamma^*(\alpha)$. The three 
have valency 6 in 
 $\Delta(\delta)$ and the six have valency $7$, contradicting vertex transitivity of $\Delta(\delta)$.

Thus $X=\Lambda(\alpha)$. As in the last paragraph, using an automorphism
 inducing$(\alpha,\delta)\mapsto (\epsilon,\alpha)$, we see that $\Lambda(\epsilon)\cap\Gamma(\alpha)$ contains six vertices, so equals 
$\Gamma(\alpha)$. Hence,   $\Gamma(\alpha)\subseteq \Lambda(\epsilon')$ for 
any $\epsilon'\in \Gamma^*(\alpha)$. Since $\Lambda$ is self-paired, it follows that
 $\Gamma^*(\alpha)\subseteq \Lambda(\delta)$. Thus, 
the digraph $\Lambda(\delta)$,
which is vertex transitive of valency $9$, contains two vertices $\delta',\delta''$ from $\Gamma(\alpha)$ (those in the same copy of $\bar{K_3}$) and six
 vertices from
 $\Gamma^*(\alpha)$, including $\epsilon$. Now $\delta'$ is independent from $\delta''$ and from all vertices in $\Gamma^*(\alpha)$, so has $\Delta$-valency at most $1$ in $\Lambda(\delta)$, but, from the structure of $\Gamma^*(\alpha)$,  $\epsilon$ 
has $\Delta$-valency at least $4$ in $\Lambda(\delta)$. This again contradicts vertex transitivity of $\Lambda(\delta)$.
\end{proof}

\begin{lemma}
If $M \in \mathcal{S}$ and $\Gamma(\alpha) \not\cong \Gamma^*(\alpha)$ then $\Gamma(\alpha)$ embeds an arc. 
\end{lemma}
\begin{proof}
Otherwise, from the lemmas above, and by Theorem~\ref{maintheorem}(i)
 the only possibilities for $\Gamma(\alpha)$ and $\Gamma^*(\alpha)$ are $C_5$ or $K_3 \times K_3$, and they have different sizes, a contradiction.  
\end{proof}

Now we know that $\Gamma(\alpha)$ embeds an arc, and the problem splits into consideration of the cases $|\Gamma(\alpha)^+| = 1$ and $|\Gamma(\alpha)^+| > 1$ (as defined before Lemma~\ref{DHom}).

\begin{lemma}\label{GammaPlusOne}
Suppose that $|\Gamma(\alpha)^+| = 1$. 
\begin{enumerate}[(a)]
\item For each $\gamma \in \Gamma^*(\alpha)$ there is unique $\delta \in \Gamma(\alpha)$ with $\gamma \to \delta$.

\item The permutation group $G_\alpha$ has isomorphic set-homogeneous actions on $\Gamma(\alpha)$ and on $\Gamma^*(\alpha)$.

\item One of the following holds.
\begin{enumerate}[(i)]
\item
  $\Gamma(\alpha) \cong E_6$ (or its complement) and $\Gamma^*(\alpha) \cong F_6$ (or its complement), or 
vice versa.

\item One of the pair $\Gamma(\alpha),\Gamma^*(\alpha)$ embeds $D_3$, and the other is isomorphic to  $ K_n[\bar{K}_m[D_3]]$ or its complement, with $n,m>1$.

\item Up to complementation, $\Gamma(\alpha)$ is isomorphic to one of the following with $n>1$: 
\[
K_n[D_3], \; K_n[D_4], \; K_n[D_5], \;  D_4, \;
 D_5, \;  E_6, \; E_7, \; F_6,  
\]
and $\Gamma^*(\alpha) \cong \overline{\Gamma(\alpha)}$.
\end{enumerate} 
\end{enumerate}
\end{lemma}

\begin{proof}

(a) We have  $|\Gamma(\alpha)^+|=1$. Now if $\alpha\rightarrow \beta$, $|\Gamma(\alpha)\cap \Gamma(\beta)|=1$,
 so if $\gamma \in \Gamma^*(\alpha)$, there is
a unique vertex $\delta\in\Gamma(\alpha)$ such that $\gamma\rightarrow \delta$.

(b) By (a), $\rightarrow$ gives
 a $G_\alpha$-invariant bijection from $\Gamma^*(\alpha)$ to $\Gamma(\alpha)$, so the permutation group
$G_\alpha^{\Gamma^*(\alpha)}$ acts isomorphically on 
$\Gamma(\alpha)$. 

(c) By induction and inspection of the list in Theorem~\ref{maintheorem}(iii), since $|\Gamma(\alpha)^+|=1$
the graph $\Gamma(\alpha)$, or its complement, must be one of the following, where $n,m>1$: 
$K_n[D_3]$, $K_n[D_4]$, $K_n[D_5]$, $K_n[\bar{K}_m[D_3]]$,
$C_5[D_3]$, $D_3$, $D_4$, $D_5$, $(K_3 \times K_3)[D_3]$, $E_6$, $E_7$ or $F_6$.

Suppose first that one of $\Gamma(\alpha)$ or $\Gamma^*(\alpha)$, say $\Gamma(\alpha)$,  is isomorphic to   $K_n[\bar{K_m}[D_3]]$ or its complement, with $n,m>1$.  Using (b), by inspection of the above list, we can check that $\Gamma^*(\alpha)$ then embeds $D_3$, so that (c)(ii) holds. Thus, we may rule out $K_n[\bar{K_m}[D_3]]$.

It can now be checked that, apart from $E_6$, $F_6$ (which give case (c)(i)), 
 for no two {\em distinct} digraphs in this list is there a group which acts isomorphically
 and set-homogeneously on both. (Note here that
 $\Aut(E_6) \cong \Aut(F_6) \cong \mathbb{Z}_2 \times \mathbb{Z}_3$.)  Thus, apart from 
cases (c)(i) and (c)(ii) , $\Gamma(\alpha)$ and $\Gamma^*(\alpha)$ must be, up to  complementation, a graph  and its complement from the list in (c)(iii) or from 
$D_3$, $(K_3 \times K_3)[D_3]$, $C_5[D_3]$.

Finally, the digraphs $D_3$, $(K_3 \times K_3)[D_3]$, $C_5[D_3]$ are all isomorphic to their complements. Thus, if 
one of them occurred as $\Gamma(\alpha)$, we would have
$\Gamma(\alpha) \cong \Gamma^*(\alpha)$,  
 contrary to assumption (III).
\end{proof}

Next we have to consider cases where one of $\Gamma(\alpha)$, $\Gamma^*(\alpha)$ is from the list in Lemma~\ref{GammaPlusOne}(c).  
Over the next four lemmas we eliminate each possibility. 

\begin{lemma}\label{SubcaseA}
If $M \in \mathcal{S}$ and $\Gamma(\alpha) \not\cong \Gamma^*(\alpha)$ then
 $\Gamma(\alpha)\not\cong K_n[D_r]$ (or its complement) 
for $r \in \{ 3,4,5 \}$ and $n\geq 1$.
\end{lemma}
\begin{proof}

Suppose that $\Gamma(\alpha)$ is isomorphic to ${K}_n[D_r]$, where  $r \in \{3,4,5\}$. Then by Lemma~\ref{GammaPlusOne}, $\Gamma^*(\alpha)$ is isomorphic 
to $\overline{\Gamma(\alpha)}$. Fix some $\epsilon \in \Gamma^*(\alpha)$. Then by Lemma \ref{GammaPlusOne}(a), there is a unique
 $\delta\in \Gamma(\alpha)$ with $\epsilon \to \delta$. Also,  all elements of the copy $T$ of $D_r$ in $\Gamma(\alpha)$ 
containing $\delta$ lie in singleton orbits of $G_{\epsilon\alpha}$. When $n=1$, as $\Gamma$ is connected, it follows that $G_{\epsilon\alpha}=1$. 

Suppose first that $r=3$. Since $\Gamma(\alpha) \not\cong \Gamma^*(\alpha)$, we have $n>1$ in this case. By applying 4-set-homogeneity to  configurations consisting of $\alpha$, one point from one copy 
of $D_3$ in $\Gamma(\alpha)$, and two $\Gamma$-connected points in another copy, we see that 
$G_\alpha$ induces a 2-transitive group on the set of copies of $D_3$ in $\Gamma(\alpha)$, and that $\Delta$ is self-paired. Likewise, considering $\alpha$ and similar configurations in $\Gamma^*(\alpha)$, $\Lambda$ is self-paired.

Now, since $n>1$, considering the structure of $\Gamma(\alpha)$ we observe that $\Delta(\delta)$ contains an arc, and hence so does $\Delta(\alpha)$. A similar argument given by considering $\Lambda(\epsilon)$ in $\Gamma^*(\alpha)$ shows that $\Lambda(\alpha)$ contains an arc. 
We claim that each copy of $D_3$ in $\Gamma(\alpha)$ dominates exactly one vertex of $M$: it dominates at most one vertex, since otherwise, 
if $\beta$ lies in the copy, then the subdigraph 
$\Gamma(\beta)$ has $\Gamma$-outvalency at least 2, a contradiction; and it dominates at least one vertex, since copies of $D_3$ in $\Gamma^*(\alpha)$ 
dominate $\alpha$. In view of the structure of $\Gamma^*(\alpha)$ we see that distinct copies of $D_3$ in $\Gamma(\alpha)$ dominate distinct vertices of $M$.
Thus, there is a set $\Sigma$ of size $n$ consisting of the vertices which are dominated by copies of $D_3$ in $\Gamma(\alpha)$. 
Clearly $\Sigma$ is a $G_\alpha$-orbit, and equals $\Lambda(\alpha)$ or $\Delta(\alpha)$ since
 $\Gamma(\alpha)$ and $\Gamma^*(\alpha)$ have size $3n$. It follows, as $\Lambda(\alpha)$ and $\Delta(\alpha)$ contain arcs, that 
the group induced by
 $G_\alpha$ on the copies of $D_3$ in $\Gamma(\alpha)$ is not 2-transitive, contradicting the previous paragraph. 
This deals with the case $r=3$. 

Next, suppose that $n>1$ and $r\in \{4,5\}$.    Then $G_{\epsilon\alpha}$ has at 
least 5 orbits on $\Gamma(\alpha)$, as the copy $T$ of $D_r$ containing
$\delta$ is fixed pointwise by $G_{\epsilon\alpha}$. 
If $\beta,\beta'\in \Gamma(\alpha)$ and $\epsilon \sim \beta$ and 
$\epsilon \sim \beta'$, then by 3-set-homogeneity and rigidity of $\{\epsilon,\alpha,\beta\}$,
the elements  $\beta,\beta'$ are in the same $G_{\epsilon\alpha}$-orbit; similarly if $\beta,\beta'$ are 
unrelated to $\epsilon$. 
Consider $\Gamma^*(\epsilon) \cap \Gamma(\alpha)$. We claim that $|\Gamma^*(\epsilon) \cap T| \leq 1$. Indeed, suppose not. Then there exist $\beta, \beta' \in T$ with $\beta, \beta' \rightarrow \epsilon$ and, fixing an arbitrary $\gamma \in \Gamma(\alpha) \setminus T$, regardless of the relationship between $\gamma$ and $\epsilon$, by set-homogeneity there is an automorphism $g \in G$ with $\{ \epsilon, \alpha, \beta, \gamma\}^{g} = \{\epsilon, \alpha, \beta', \gamma \}$, and by considering the substructure induced by $\{ \epsilon, \alpha, \beta, \gamma \}$ we conclude that $\epsilon^g = \epsilon$ and $\alpha^g = \alpha$. But $G_{\epsilon \alpha}$ fixes $\delta$ and $\beta$ and thus $\beta = \beta'$ or $\beta = \gamma$, in either case a contradiction. We conclude as claimed that $|\Gamma^*(\epsilon) \cap T| \leq 1$.
Thus, we have that each of $\Gamma(\epsilon) \cap T $, $\Gamma^*(\epsilon) \cap T$,
$\Lambda(\epsilon) \cap T$, $\Delta(\epsilon) \cap T$ have size at most one, which in particular forces $r=4$. Thus, 
 $(\epsilon,\mu) \in \Lambda$ and $(\epsilon,\mu') \in \Delta$ for some $\mu, \mu' \in T$, and by the observations above about the $G_{\epsilon \alpha}$-orbits of $\Gamma(\alpha)$, $\beta \rightarrow \epsilon$ for every  $\beta \in \Gamma(\alpha)\setminus T$. In particular, since $n>1$, there is a copy $T'$ of $D_4$ in $\Gamma(\alpha)$ disjoint from $T$ with $T' \subseteq \Gamma^*(\epsilon)$. But this is a contradiction since $\Gamma^*(\epsilon) \cong \Gamma^*(\alpha)$ does not embed $D_4$.

Thus, we have $n=1$, and $r\in \{4,5\}$. Recall that since $n=1$ we have $G_{\epsilon \alpha} = 1$. 

Suppose first $r=5$, so $\Gamma(\alpha) \cong D_5$ and $\Gamma^*(\alpha) \cong \overline{D_5}$.  The triple $(u, v, w)$ where $u \to v \to w$ and $u\sim w$ embeds in $\Gamma^*(\alpha)$, so there is
$\gamma \in \Gamma(\alpha)$ with $\epsilon \sim \gamma$. Similarly (considering a $2$-arc in $\Gamma^*(\alpha)$), there is $\gamma'\in \Gamma(\alpha)$ with $\epsilon$ independent from $\gamma'$. Since $G_{\epsilon\alpha}$ fixes $\Gamma(\alpha)$ pointwise, by 3-set-homogeneity $\gamma$ and $\gamma'$
 are unique.
Now let $\Gamma(\alpha)\setminus\{\delta,\gamma,\gamma'\}=\{\beta,\beta'\}$. Then we must have 
$\beta \to \epsilon$ and $\beta' \to \epsilon$.

Since $G_{\epsilon\alpha}$ fixes $\beta$ and $(\epsilon, \alpha, \beta) \cong (\epsilon, \alpha, \beta')$, it follows that the isomorphism $(\epsilon,\alpha,\beta \mapsto \epsilon,\alpha,\beta')$ between copies of $D_3$ does not extend to an automorphism, and so by Lemma~\ref{lem_basic2}, the group 
induced by $G$ on copies of $D_3$ is trivial. 
Thus, again by Lemma~\ref{lem_basic2}, it follows that $|\Gamma^*(\epsilon) \cap \Gamma(\alpha)| = 3$, a 
contradiction as $\Gamma^*(\epsilon) \cap \Gamma(\alpha) = \{ \beta, \beta' \} $. 

It remains only to consider the case $\Gamma(\alpha)\cong D_4$, and $\Gamma^*(\alpha)\cong \bar{D}_4$. The argument is in part similar to \cite[6.2]{lachlan}.
There is $\epsilon' \in \Gamma^*(\alpha)$  with $\epsilon \to \epsilon'$. Now $\Gamma(\epsilon) \cong \Gamma(\alpha) \cong D_4$, so as
$\epsilon $ dominates $\epsilon',\alpha,\delta$ and
$\epsilon' \to \alpha \to \delta$, there is $\mu \in \Lambda(\alpha)$ with $\epsilon \to \mu$ and  $\delta \to \mu \to \epsilon'$. Since $G_{\epsilon \alpha}=1$ we have $|G_{\alpha}|=4$ and hence 
$|\Lambda(\alpha)|$ divides 4. 

If $|\Lambda(\alpha)| = 1$ then $G_\alpha$ acts transitively on $\Gamma^*(\alpha)$ and fixes $\Lambda(\alpha) = \{ \mu \}$, and this is a contradiction since $\epsilon \rightarrow \mu$ while $\epsilon' \leftarrow \mu$.

Next suppose $|\Lambda(\alpha)| = 2$. Now $G_{\alpha}$ acts transitively on $\Gamma(\alpha)$ and fixes $\Lambda(\alpha) = \{ \mu, \mu' \}$ setwise, and it easily follows that there exist $\delta, \delta' \in \Gamma(\alpha)$ with $(\delta,\delta') \in \Lambda$ and $\{ \delta, \delta' \} \subseteq \Gamma^*(\mu)$. But this contradicts $\Gamma^*(\mu) \cong \Gamma^*(\alpha) \cong \overline{D_4}$.

Thus, $|\Lambda(\alpha)|=4$, and similarly $|\Delta(\alpha)|=4$. Now as $\Lambda(\epsilon)$ has size four and  contains just one point from
$\{\alpha\} \cup \Gamma(\alpha) \cup \Gamma^*(\alpha)$, it contains three points from $\Lambda(\alpha) \cup \Delta(\alpha)$, so contains two points
$\beta_1,\beta_2$ in $\Lambda(\alpha)$, or two points $\beta_1,\beta_2$ in $\Delta(\alpha)$. Either way, 3-set-homogeneity yields an
 automorphism inducing
$(\epsilon, \alpha, \beta_1) \mapsto (\epsilon, \alpha, \beta_2)$, contradicting the fact that $G_{\epsilon \alpha}=1$.

The corresponding cases when $\Gamma^*(\alpha)$ is isomorphic to ${K}_n[D_r]$ are also handled by the above (they arise by
 reversing all arcs). 
\end{proof}

\begin{lemma}\label{SubcaseB}
Case (c)(ii) of Lemma~\ref{GammaPlusOne} cannot occur.

\end{lemma}
\begin{proof}
We shall suppose for a contradiction that $\Gamma(\alpha) \cong K_n[\bar{K}_m[D_3]]$ (with $m,n > 1$), and $\Gamma^*(\alpha)$ embeds $D_3$. This suffices, since the other case arise through a combination of  complements and weak complements.

Let $\gamma \in \Gamma(\alpha)$. By considering the digraph $\Gamma(\alpha)$ 
we see that both $\Lambda(\gamma)$ and $\Delta(\gamma)$ contain arcs. Therefore by vertex transitivity both
 $\Lambda(\alpha)$ and $\Delta(\alpha)$ contains arcs. Also, each copy of $D_3$ in $\Gamma(\alpha)$ 
dominates exactly one vertex: it dominates at most one since otherwise, with $\beta$ in a copy of $D_3$ in $\Gamma(\alpha)$, we would have a vertex in $\Gamma(\beta)$ with 
out-degree at least $2$, contradicting the fact that $\Gamma(\beta) \cong K_n[\bar{K_m}[D_3]]$; and  each copy of $D_3$ in $\Gamma(\alpha)$ dominates at least one vertex since $D_3$ embeds into $\Gamma^*(\alpha)$. 

Thus there is a set $\Sigma$ of size $mn$ consisting of vertices dominated by copies of $D_3$
 in $\Gamma(\alpha)$. Clearly $\Sigma$ is a $G_{\alpha}$-orbit and equals $\Lambda(\alpha)$ or 
$\Delta(\alpha)$ since $|\Gamma(\alpha)| = |\Gamma^*(\alpha)| = 3mn > mn$. We claim that the group
 induced by $G_{\alpha}$ on the copies of $D_3$ in $\Gamma(\alpha)$ has all orbitals self-paired. 
Once established, this will be a contradiction since we have already seen that both $\Lambda(\alpha)$ 
and $\Delta(\alpha)$ contain arcs, and the group induced by $G_{\alpha}$ on the copies of $D_3$ in 
$\Gamma(\alpha)$ is the same as the group induced on either $\Lambda(\alpha)$ or $\Delta(\alpha)$. 
To prove the claim, let $T$ and $S$ be two copies of $D_3$ in $\Gamma(\alpha)$ with $t \in T$ and $s \in S$. First suppose that $s$ and $t$ are $\Delta$-related. Let $t' \in T$ with $t' \rightarrow t$, and $s' \in S$ with $s' \rightarrow s$. Then by set-homogeneity, and since the configuration is rigid, the isomorphism $(\alpha,t',t,s) \mapsto (\alpha,s',s,t)$ extends to an automorphism which interchanges $S$ and $T$. The case that $s$ and $t$ are $\Lambda$-related is dealt with similarly.  
 \end{proof}

\begin{lemma}
If $M \in \mathcal{S}$ and $\Gamma(\alpha) \not\cong \Gamma^*(\alpha)$ then we cannot have
$\Gamma(\alpha)$ isomorphic to $E_6,E_7, F_6$, or their complements.

\end{lemma}
\begin{proof}
Suppose that each of $\Gamma(\alpha)$, $\Gamma^*(\alpha)$ is isomorphic to one of $E_6$, $E_7$, $F_6$, or their complements.
Observe first that $\Aut(E_6)$, $\Aut(E_7)$,  and $\Aut(F_6)$ are each regular on vertices, and  
$\Aut(E_6)$ and $\Aut(E_7)$ are both cyclic. 
We have a bijection between 
$\Gamma^*(\alpha)$ and $\Gamma(\alpha)$ given by $\rightarrow$, by Lemma \ref{GammaPlusOne}. In particular, by connectedness of $\Gamma$, if
$\beta\in \Gamma^*(\alpha)$ then $G_{\beta\alpha}=1$.

Pick $\beta\in\Gamma^*(\alpha)$ and $\beta'\in \Gamma(\alpha)$ with $\beta \rightarrow \beta'$. By Lemma~\ref{GammaPlusOne}(a), for the 5 (or 6) elements
 $\gamma\in \Gamma(\alpha)\setminus \{\beta'\}$, 
there are three possible relations to $\beta$: $\gamma\rightarrow \beta$, $\gamma\sim \beta$, or $\gamma$, $\beta$ independent.
Since $G_{\beta\alpha}=1$, by 3-set-homogeneity only one element $\gamma$ of $\Gamma(\alpha)$ can be $\sim$-related to 
$\beta$, and only one can be independent from $\beta$, for the corresponding structures on $\{\alpha,\beta,\gamma\}$ are rigid. The group induced by $G$ on each copy of $D_3$ in $M$ either has size $1$ or $3$, and hence by Lemma~\ref{lem_basic2}, $\Gamma^*(\beta) \cap \Gamma(\alpha)$
has size 1 or 3, so by counting, has size 3. This  ensures $|\Gamma(\alpha)|=6$, so eliminates $E_7$ and its complement. It also eliminates $F_6$ and its complement, since 
a set-homogeneity argument shows that
any group acting set-homogeneously on $F_6$ (or on its complement) induces a transitive group on each $D_3$, contradicting Lemma~\ref{lem_basic2}. 
 
Thus, we may assume that one
 of $\Gamma(\alpha),\Gamma^*(\alpha)$ is $E_6$, and the other is its complement. One case is obtained from the other by
 reversing all arrows, so we may suppose that $\Gamma(\alpha)\cong E_6$. Note that, by the previous paragraph, each vertex of $\Gamma^*(\alpha)$ is dominated by
 exactly three vertices of $\Gamma(\alpha)$.
 
 Consider configurations $\beta_1,\beta_2,\beta_3\in \Gamma(\alpha)$ with 
$\beta_1\rightarrow \beta_2$, $\beta_2\sim \beta_3$, and $\beta_1||\beta_3$. Call $U_1$ the isomorphism type of 
such structures, and $U_2$
the corresponding isomorphism type when instead $\beta_2\rightarrow \beta_1$. Both $U_1$ and $U_2$ embed in both 
$\Gamma(\alpha)$ and 
$\Gamma^*(\alpha)$, with six copies of each, permuted transitively by $G_\alpha$. Thus, each copy of $U_1$ in $\Gamma(\alpha)$
 dominates
a vertex $u_1$, and no two copies of $U_1,U_1'$ in $\Gamma(\alpha)$ dominates the same vertex, as
 $\Gamma(\alpha),\Gamma^*(\alpha)$ 
do not share a 4-vertex isomorphism type. Thus, the vertex $u_1$ lies in a $G_\alpha$-orbit of size 6; likewise for any vertex
 $u_2$ dominated by a copy of $U_2$ in $\Gamma(\alpha)$.  
 
 Clearly we cannot have $u_1,u_2\in \Gamma(\alpha)$. If both $u_1,u_2\in \Gamma^*(\alpha)$ then a vertex in $\Gamma^*(\alpha)$ 
is dominated by at least four vertices of $\Gamma(\alpha)$, a contradiction. Thus  $u_i$, say, is not in $\Gamma(\alpha)\cup \Gamma^*(\alpha)$, and
hence lies in either $\Lambda(\alpha)$ or $\Delta(\alpha)$, say the former, so $|\Lambda(\alpha)|=6$ (the argument in the other case is similar). By $\Gamma$-connectedness $|G_\alpha|=6$. Pick 
$\beta\in\Gamma(\alpha)$ and recall that $G_{\alpha\beta}=1$. If there are $\gamma_1,\gamma_2\in \Lambda(\alpha)$ both 
dominated by $\beta$, or both independent from $\beta$, or both $\sim$-related to $\beta$, then by 3-set-homogeneity and rigidity of $\{\alpha,\beta,\gamma_i\}$ there is 
$g\in G_{\alpha\beta}$ with $\gamma_1^g=\gamma_2$, a contradiction. Thus, there are at least three vertices $\gamma_1,\gamma_2,\gamma_3 \in \Lambda(\alpha)$ all 
dominating $\beta$.
As the structure on $\alpha,\beta,\gamma_1$ admits just two automorphisms, there are distinct $\gamma_i,\gamma_j$
in the same $G_{\alpha\beta}$-orbit. This contradicts again that $G_{\alpha\beta}=1$.
\end{proof}

By Lemma~\ref{GammaPlusOne}, the last three lemmas have treated all cases in which 
$|\Gamma(\alpha)^+|=1$.
The remaining lemma eliminates the case $|\Gamma(\alpha)^+|>1$, which includes the case that $\Gamma(\alpha) \cong H_3$. 

\begin{lemma}
If $M \in \mathcal{S}$ and $\Gamma(\alpha) \not\cong \Gamma^*(\alpha)$ then we must have $|\Gamma(\alpha)^+| \leq 1$.
\end{lemma}
\begin{proof}
Suppose for a contradiction that $|\Gamma(\alpha)^+|>1$.
We consider the remaining cases of pairs $\Sigma_1$, $\Sigma_2$ of graphs from the list in 
Theorem~\ref{maintheorem}(iii)
which are possibilities for $\Gamma(\alpha)$, $\Gamma^*(\alpha)$. We require the following (see Lemma~\ref{DHom}):
$\Sigma_1, \Sigma_2$ are  non-isomorphic, have the same number of vertices, $\Sigma_1^- \cong\Sigma_2^+$ 
and has size greater than one.
In each case, there is no harm in supposing $\Gamma(\alpha)=\Sigma_1$ and $\Gamma^*(\alpha)=\Sigma_2$, and we do this; for we may replace $M$ by its weak complement if necessary, since this also will be set-homogeneous and satisfy the same assumptions.  Also,
 we may assume that $\Sigma_1$, rather than its complement,
comes from the list. 

\

\noindent \textbf{Case 1:} \textit{$\Sigma_1=K_n[H_0]$ with $n\geq1$.} 

\

\noindent Then $\Sigma_2= \overline{K_n[H_0]}$. Now  
 the tournaments $T=\{\beta,\gamma_0,\gamma_1,\gamma_2\}$ (where $\beta\rightarrow \gamma_i$ and 
$\gamma_i\rightarrow \gamma_{i+1 ({\rm mod 3})}$
for $i=0,1,2$), and $\bar{T}$, are both  configurations which are maximal subject to embedding
in both $\Gamma(\alpha)$ and $\Gamma^*(\alpha)$. Thus, given a copy of $T$ in $\Gamma(\alpha)$ labelled as above, there 
is $\delta$
with $T\subset \Gamma^*(\delta)$. Clearly $\delta\not\in \Gamma(\alpha)$. If $\delta\not\in \Gamma^*(\alpha)$, then with $\beta \in T$ the
 digraph induced on $\{\alpha,\beta,\delta\}$
is rigid, so by 3-set-homogeneity there is $g\in G$ with $(\alpha,\beta,\delta)^g=(\alpha,\gamma_1,\delta)$. However, because 
$g$ fixes
$\alpha$ and $\delta$ it induces an automorphism of  $T$, and no automorphism of $T$ moves $\beta$ to $\gamma_1$.

If $\delta\in \Gamma^*(\alpha)$, then the above argument applies provided we know that $G$ induces $\mathbb{Z}_3$ on copies of $D_3$. 
To see that the latter holds, note that $\{\alpha,\beta,\gamma_i\}$ is rigid, so there is
$g\in G_{\alpha\beta}$ with $\gamma_0^g=\gamma_1$. Such an element $g$ 
fixes $\Gamma(\alpha)\cap \Gamma(\beta)=\{\gamma_0,\gamma_1,\gamma_2\}$ and generates a copy of $\mathbb{Z}_3$
on this set.

\

\noindent \textbf{Case 2:} \textit{$\Sigma_1$ is a member of
\[
\begin{array}{c}
\{D_3[K_n],\; D_4[K_n],\; D_5[K_n],\;  H_0[K_n],\;  D_3[C_5],\;  D_3[K_3\times K_3], \\
 D_3[K_m[\bar{K}_n]], \; H_1, \; H_2, \; H_3, \; J_n, \; K_n[D_3[\bar{K}_m]] \}
\end{array}
\]
for $n,m>1$. }

\

\noindent In these cases there are no possibilities for $\Sigma_2$ with the required properties. 

All cases are easily eliminated except for the case $\Sigma_1\cong H_3$, so we consider this case. 
When $\Sigma_1 \cong H_3$, by Lemma~\ref{DHom}(ii) we know that $\Sigma_1^- \cong \Sigma_2^+$, 
and since $\Sigma_1 \cong H_3$ we have $\Sigma_1^- \cong H_0$. Hence $\Sigma_2^+ \cong H_0$ and this means 
that $\Sigma_2$ cannot be isomorphic to any of : $D_3[K_3 \times K_3]$, $D_3[K_9]$, $D_3[K_3[\bar{K}_3]]$,
 $J_9$, or $K_3[D_3[\bar{K}_3]]$. But then the only remaining possibility is that $\Sigma_2 \cong \bar{H_3}$.
 
In this case choose $\delta \in \Gamma(\alpha)$. Observe that (by inspection of $H_3$) the    pairwise unrelated sets in $H_3$ have size at most nine, and pairwise unrelated sets in $\Gamma^*(\alpha)$ have size at most 3. Consider a pairwise unrelated set $X$ in $\Gamma(\alpha)$ of size 9. For every 3-subset $A$ of $X$ containing $\delta$, there is $\beta$ with $A\subset \Gamma^*(\beta)$, and distinct 3-sets give distinct $\beta$. Since there are 
${8 \choose 2}=28$ possibilities for $A$, we find $|\Gamma(\delta)|\geq 28$, a contradiction as $|\Gamma(\alpha)|=|H_3|=27$.

\end{proof}

Combining the results of the last two sections, we now obtain a proof of our main result. 

\subsection*{Proof of Theorem~\ref{maintheorem}(iii)} Let  $M$ be a finite symmetric set-homogeneous digraph. By Lemma~\ref{disconnect}(ii) we may suppose that $M$ is $\Gamma$-connected. We now argue by induction on $|M|$. By Lemma~\ref{nhood} and the induction hypothesis, $\Gamma(\alpha)$ and $\Gamma^*(\alpha)$ belong to the list in (iii), and as $\Gamma$ and $\Gamma^*$ are paired orbitals $|\Gamma(\alpha)| = |\Gamma^*(\alpha)|$. Also, by Lemma~\ref{samenbours}(ii) we may suppose that for vertices $\alpha$, $\beta$, if $\alpha \neq \beta$ then $\Gamma(\alpha) \neq \Gamma(\beta)$ and $\Gamma^*(\alpha) \neq \Gamma^*(\beta)$. 

It follows from the results in Section~\ref{sec_5} that $\Gamma(\alpha) \cong \Gamma^*(\alpha)$. Then applying the results of Section~\ref{sec_CaseInNEqualsOutN} we deduce that $M$ is isomorphic to one of the digraphs in (iii), completing the proof of the theorem. \qed

\section{Infinite set-homogeneous digraphs}

In this section we begin a study of countably infinite set-homogeneous a-digraphs.
As before, there is an orbital $\Gamma$ such that $\alpha \to \beta$ if and only if 
$(\alpha,\beta)\in \Gamma$,
and there is another orbital $\Lambda$ such that if $(\alpha,\beta)\in \Lambda$ then 
$\alpha,\beta$ are distinct and unrelated by $\to$. We shall write $\alpha||\beta$ if $(\alpha,\beta)\in \Lambda \cup \Lambda^*$.
The main result is Theorem~\ref{infnot2hom},
a classification 
under 
the additional assumption that the a-digraph is not 2-homogeneous, or equivalently, that the orbital 
$\Lambda$ consisting of an orbit on unrelated pairs
is not self-paired.
As found also in \cite{dgms}
 for graphs, the problem of classifying
{\em all} set-homogeneous countable a-digraphs appears to be hard.

Let $T(4)$ be the a-digraph obtained by distributing countably many points densely around the unit circle, no two making an 
angle of $\pi/2$ or $\pi$ at the centre, such that $y\rightarrow x$ if and only if
$\pi/2<\arg(y/x)<\pi$. 

Also, let $2\leq n \leq \aleph_0$, and let $\{Q_i: 0\leq i<n\}$ be a partition of the rationals ${\mathbb Q}$ into $n$ dense codense sets. Define an 
a-digraph $R_n$ with domain ${\mathbb Q}$,
putting $a\rightarrow b$ if and only if $a<b$ and there is no $i<n$ such that $a,b\in Q_i$.

By standard back-and-forth arguments (see \cite[Section 2.5]{cameron2} or \cite[Ch. 9]{bmmn} for background on arguments of this kind), these constructions (for $T(4)$ and the $R_n$) determine unique a-digraphs up to isomorphism.

\begin{lemma}\label{not2hom}

\

\begin{enumerate}[(i)]
\item 
The a-digraphs $R_n$ (for $n\geq 2$) are set-homogeneous but not $2$-homogeneous, and have imprimitive automorphism groups.
\item 
The a-digraph $T(4)$  is set-homogeneous but not $2$-homogeneous, and has primitive automorphism group.
\item Let $T$ be a set-homogeneous tournament. Then the disjoint union $\bar{K}_n[T]$ of $n$ copies of $T$ (where $n\in \{\aleph_0\}\cup ({\mathbb N}\setminus\{0\})$) is set-homogeneous and $2$-homogeneous, and is homogeneous if and only if $T$ is homogeneous. Also, $T[\bar{K}_n]$ is a set-homogeneous a-digraph, and is $2$-homogeneous.
\end{enumerate}
\end{lemma}

\begin{proof}
(i)  First, observe, by a back-and-forth argument, that there is a homogeneous expansion $R_n'$ of $R_n$ to a language with an 
equivalence relation whose classes are the $Q_i$ (which are the maximal independent sets in $R_n$), 
and with a binary relation symbol for the natural ordering on ${\mathbb Q}$, such 
that $\Aut(R_n)=\Aut(R_n')$. Note here that the induced ordering on each $Q_i$ is invariant under $\Aut(R_n)$, since if $x,y\in Q_i$ are distinct we have $x<y\Leftrightarrow \Gamma(x)\supset \Gamma(y)$.

Let $\theta:U\rightarrow V$ be an isomorphism between finite subsets of $R_n$. The relation $||$ (for unrelated pairs)
is an equivalence relation on $U$ and $V$, preserved by $\theta$. We may deform $\theta$ to a map 
$\theta':U\rightarrow V$ which induces the same map as $\theta$ on $U/||$ but
respects the ordering on each $||$-class induced from ${\mathbb Q}$. Then $\theta':U\to V$ is an isomorphism; for if $E$ is the equivalence relation on $U={\rm dom}(\theta)$ given by
$$Exy\Leftrightarrow x||y \wedge (\Gamma(x)\cap U=\Gamma(y)\cap U) \wedge (\Gamma^*(x)\cap U=\Gamma^*(y)\cap U),$$
then there is an automorphism $\phi$ of ${\rm dom}(\theta)$ fixing each $E$-class setwise such that $\theta'=\phi\circ \theta$. Now $\theta'$ is an isomorphism of finite substructures 
of the expansion $R_n'$, so by homogeneity, $\theta'$ extends to an automorphism $\tilde{\theta}$ of $R_n'$,
 and this is 
also an automorphism of $R_n$,
with $\tilde{\theta}(U)=V$. The a-digraph $R_n$  is not 2-homogeneous, for if $x,y\in Q_2$ with $x<y$ then 
$\{x,y\} \cong\{y,x\}$, but
there is $z$ with
$x\rightarrow z \rightarrow y$ but no $z$ with $y\rightarrow z \rightarrow x$. 

(ii) Consider the binary structure $S(4)$ introduced by Cameron \cite{cameron1} and mentioned in \cite{dgms}.
 The domain is, as  for $T(4)$, a countable set $M$ of points on the unit circle, distributed densely,
 no two making an angle $2k\pi/4$ at the centre ($k\in {\mathbb Z}$). 
There are
 binary relations $\sigma_0,\sigma_1,\sigma_2,\sigma_3$ with $\sigma_i(a,b)$ if and only if $2\pi i/4<\arg(a/b)
<2\pi(i+1)/4$.
The structure $S(4)=(M,\sigma_0,\sigma_1,\sigma_2,\sigma_3)$ is homogeneous, as noted in \cite[Section 8]{cameron1}. (Formally, Cameron also has the circular ordering in the language, but since it is definable without quantifiers from the $\sigma_i$, it is not needed.) Now $T(4)$ is a `reduct' of $S(4)$;
 that is
$\alpha\rightarrow \beta$ if and only if $\sigma_1(\alpha,\beta)$ holds. In particular, $\Aut( T(4))\geq \Aut(S(4))$ (in fact, we have equality here).

To show $T(4)$ is set-homogeneous, suppose that $\theta:U\rightarrow V$ is an isomorphism between finite subdigraphs of $T(4)$.
Define an equivalence relation $\equiv$ on $U$, putting $x\equiv y$ if and only if 
$\Gamma(x) \cap U=\Gamma(y)\cap U$ and $\Gamma^*(x) \cap U=\Gamma^*(y) \cap U$, and define $\equiv$ similarly on $V$. Then 
$\theta$ respects $\equiv$, so induces a bijection
$U/\equiv ~\longrightarrow V/\equiv$. Each $\equiv$-class is an independent set, so any pair of distinct $\equiv$-related elements from $U$ or 
$V$ satisfies $\sigma_0$ or $\sigma_3$ (in $S(4))$). Suppose $\theta$ maps the $\equiv$-class
$\{\alpha_1,\ldots,\alpha_r\}$ of $U$ to $\{\beta_1,\ldots,\beta_r\}$ of $V$.  We may suppose that the enumerations are such that
$\sigma_0(\alpha_i,\alpha_{i+1})$ and $\sigma_0(\beta_i,\beta_{i+1})$ for each $i=1,\ldots,r-1$. We may deform $\theta$ 
to $\theta'$ 
such that for all such $\equiv$-classes,
$\theta'(\alpha_i)=\beta_i$, that is, $\theta'$ respects $\sigma_0$ on each $\equiv$-class. 
The map $\theta'$ is also an isomorphism between substructures of $T(4)$, since any permutation of $U$ which fixes each  $\equiv$-class setwise is an automorphism of $U$.
We claim that $\theta'$ 
extends to an automorphism of $T(4)$. To see this,
by the homogeneity of $S(4)$, it suffices to show that, on $U$, we can reconstruct the $\sigma_i$ from knowledge of 
$\rightarrow$ and of $\sigma_0$ on each $\equiv$-class. So suppose that $\alpha,\beta\in U$ are independent and in
 different $\equiv$-classes. 
If there is $x\in U$ such that $\alpha\rightarrow x$ and $\beta||x$, or $\beta\rightarrow x$ and 
$x\rightarrow \alpha$, or $x||\alpha$ and $x\rightarrow \beta$, then $\sigma_0(\alpha,\beta)$. Otherwise, 
$\sigma_0(\beta,\alpha)$ holds. Similarly we may recover $\sigma_3$, and $\sigma_1$ and $\sigma_2$ are given (as $\Gamma$ and $\Gamma^*$).  

Clearly $T(4)$ is not $2$-homogeneous because there exist independent pairs that 
cannot be swapped by any automorphism. Indeed, if $x,y\in T(4)$ with $0<\arg(y/x)<\pi/2$, then
$\exists z(z\rightarrow y\wedge x\rightarrow z)$, but $\neg \exists z(z\rightarrow x \wedge y \rightarrow z)$.

For the primitivity assertion, it suffices to show that $\Aut(S(4))$ is primitive. For this, it is enough to verify that each $\sigma_i$ generates the universal equivalence relation on $M$. This is straightforward.

(iii) This is easy, and is omitted.
\end{proof}

The proof of Theorem~\ref{infnot2hom} breaks into two parts, depending on whether or not $M$ has primitive automorphism group. The theorem will follow immediately from Propositions~\ref{recover1} (the primitive case) and \ref{recover2} (the imprimitive case), the subject of the rest of the section.

\subsection*{The primitive case}

\begin{lemma}\label{(1)and(2)}
Let $M$ be a set-homogeneous countably infinite a-digraph, and let $G := \Aut(M)$ be its automorphism group.   
\begin{enumerate}[(i)]
\item \label{(1)} If $G$ is primitive then each of the suborbits $\Gamma(\alpha)$, $\Gamma^*(\alpha)$, $\Lambda(\alpha)$, and $\Lambda^*(\alpha)$ is infinite.
\item \label{(2)} If there exist vertices $\alpha$, $\beta$, $\alpha'$ and $\beta'$ such that $\alpha||\beta$, $\alpha'||\beta$, $\beta'||\alpha$, and either
$\alpha\rightarrow \alpha'$ and $\beta\rightarrow \beta'$, or $\alpha'\rightarrow \alpha$ and $\beta'\rightarrow \beta$, then $M$ is $2$-homogeneous. 
\end{enumerate}
\end{lemma}
\begin{proof}
(i) This is standard -- otherwise the union of the finite $G_\alpha$-orbits would be a 
non-trivial block of imprimitivity (see e.g. Proposition 2.3 of \cite{mp}).

(ii) If such $\alpha',\beta'$ exist, then by 3-set-homogeneity the isomorphism 
$(\alpha,\alpha',\beta)\mapsto (\beta, \beta',\alpha)$
must extend to an automorphism of $M$ interchanging $\alpha$ and $\beta$, and hence $M$ is $2$-homogeneous.
\end{proof}

For what remains of this subsection $M$ will denote a set-homogeneous but not 2-homogeneous countably 
infinite a-digraph, with a primitive automorphism group $G:=\Aut(M)$. We shall prove several lemmas about 
$M$ which will then be used in the proof of Proposition~\ref{recover1} to show that $M \cong T(4)$. 

Since $M$ is not 2-homogeneous, there are two distinct paired orbitals $\Lambda$ and $\Lambda^*$ on unrelated pairs. 
We shall write $\alpha\Rightarrow \beta$
to mean that $\beta\in \Lambda(\alpha)$. Also, we write $\alpha||\beta$ to mean that $\alpha,\beta$ are independent, and
$\alpha||B$
 to mean that  there is no arc between $\alpha$ and any member of the set $B$ of vertices. We use the following notation to 
denote certain configurations on three vertices (see Table~\ref{table_2}).

\begin{table}\label{configurations2}
\begin{center}
\begin{tabular}{|c|c||c|c||c|c|}\hline
\raisebox{-3ex}[0pt]{$L_1$}  &  \xymatrix{
&  					     & \node \rulab{\gamma} &     \\
& \node \darcc{ur} \ldlab{\alpha}     &   \node  \darcc{l} \arcc{u} \rdlab{\beta} & \\
} &
\raisebox{-3ex}[0pt]{$L_6$}  &  \xymatrix{
&  					     & \node \rulab{\gamma} &     \\
& \node  \darcc{ur}  \ldlab{\alpha}     &   \node  \bdarcc{l} \darcc{u} \rdlab{\beta} & \\
} &
\raisebox{-3ex}[0pt]{$L_{11}$}  & \xymatrix{
&  					     & \node \rulab{\gamma} &     \\
& \node \arcc{ur} \ldlab{\alpha}     &   \node  \barcc{l} \arcc{u} \rdlab{\beta} & \\
}  \\ \hline
\raisebox{-3ex}[0pt]{$L_2$}  &  \xymatrix{
&  					     & \node \rulab{\gamma} &     \\
& \node \arcc{ur} \ldlab{\alpha}     &   \node  \darcc{l} \arcc{u} \rdlab{\beta} & \\
} &
\raisebox{-3ex}[0pt]{$L_7$}  &  \xymatrix{
&  					     & \node \barcc{d} \arcc{dl} \rulab{\gamma} &     \\
& \node   \ldlab{\alpha}     &   \node  \darcc{l}  \rdlab{\beta} & \\
} &
\raisebox{-3ex}[0pt]{$L_{12}$}  &  \xymatrix{
&  					     & \node \rulab{\gamma} &     \\
& \node  \ldlab{\alpha}  \bdarcc{ur}   &   \node  \darcc{l} \arcc{u} \rdlab{\beta} & \\
}  \\ \hline
\raisebox{-3ex}[0pt]{$L_3$}  &  \xymatrix{
&  					     & \node \rulab{\gamma} &     \\
& \node \barcc{ur} \ldlab{\alpha}     &   \node  \darcc{l} \barcc{u} \rdlab{\beta} & \\
} &
\raisebox{-3ex}[0pt]{$L_8$}  &  \xymatrix{
&  					     & \node \arcc{d} \rulab{\gamma} &     \\
& \node  \ldlab{\alpha}     &   \node  \darcc{l}  \rdlab{\beta} & \\
} &
\raisebox{-3ex}[0pt]{$L_{13}$}  &  \xymatrix{
&  					     & \node \bdarcc{d} \arcc{dl} \rulab{\gamma} &     \\
& \node  \ldlab{\alpha}     &   \node  \darcc{l}  \rdlab{\beta} & \\
} \\ \hline
\raisebox{-3ex}[0pt]{$L_4$}  &  \xymatrix{
&  					     & \node \rulab{\gamma} &     \\
& \node \arcc{ur} \ldlab{\alpha}     &   \node  \darcc{l} \barcc{u} \rdlab{\beta} & \\
} &
\raisebox{-3ex}[0pt]{$L_9$}  &  \xymatrix{
&  					     & \node \rulab{\gamma} &     \\
& \node  \arcc{ur} \ldlab{\alpha}     &   \node  \darcc{l}  \rdlab{\beta} & \\
} &
\raisebox{-3ex}[0pt]{$L_{14}$}  &  \xymatrix{
&  					     & \node \bdarcc{dl} \arcc{d} \rulab{\gamma} &     \\
& \node  \ldlab{\alpha}     &   \node  \darcc{l} \rdlab{\beta} & \\
} \\ \hline
\raisebox{-3ex}[0pt]{$L_5$}  &  \xymatrix{
&  					     & \node \rulab{\gamma} &     \\
& \node \barcc{ur} \ldlab{\alpha}     &   \node  \barcc{l} \arcc{u} \rdlab{\beta} & \\
} &
\raisebox{-3ex}[0pt]{$L_{10}$}  &  \xymatrix{
&  					     & \node \rulab{\gamma} &     \\
& \node  \darcc{ur}  \ldlab{\alpha}     &   \node  \darcc{l} \bdarcc{u} \rdlab{\beta} & \\
} &
&  \\ \hline
\end{tabular}
\end{center}
\caption{Configurations on $3$ vertices.}
 \label{table_2}
\end{table}

\

\begin{tabular}{ll}
  $L_1$: $\beta\Rightarrow \alpha$, $\beta\rightarrow \gamma$, $\alpha\Rightarrow \gamma$. & $L_8$: $\beta\Rightarrow \alpha$, $\gamma\rightarrow \beta$, $\gamma ||\alpha$ (isomorphic to $L_{12}$ or $L_{14}$).   \\
$L_2$: $\beta\Rightarrow \alpha$, $\alpha,\beta \rightarrow \gamma$. &
$L_{9}$: $\beta\Rightarrow \alpha$, $\alpha\rightarrow \gamma$, $\gamma||\beta$ (isomorphic to $L_{13}$ or $L_{14}$). \\
$L_3$: $\beta\Rightarrow \alpha$, $\gamma \rightarrow \alpha,\beta$. &  $L_{10}$: $\beta\Rightarrow \alpha$, $\alpha \Rightarrow \gamma$, $\gamma \Rightarrow \beta$.  \\
$L_4$: $\beta\Rightarrow \alpha$, $\alpha\rightarrow\gamma\rightarrow \beta$.  & $L_{11}$: $\alpha\rightarrow \beta\rightarrow\gamma$, $\alpha\rightarrow \gamma$.
   \\
 $L_5$: $\alpha\rightarrow\beta\rightarrow\gamma\rightarrow\alpha$. & $L_{12}$: $\beta\Rightarrow \alpha$, $\beta\rightarrow \gamma$, $\gamma\Rightarrow\alpha$.
   \\
 $L_6$:  $\alpha\Rightarrow\beta\Rightarrow\gamma$, $\alpha\Rightarrow\gamma$.  & $L_{13}$: $\beta\Rightarrow \alpha$, $\gamma \rightarrow \alpha$, $\beta\Rightarrow \gamma$.
  \\
$L_7$: $\beta\Rightarrow \alpha$, $\beta\rightarrow \gamma\rightarrow\alpha$. & $L_{14}$: $\beta\Rightarrow\alpha\Rightarrow \gamma\rightarrow \beta$. 
\end{tabular}

\

\begin{lemma}\label{(3)and(4)}
If $\alpha \Rightarrow \beta$ then:
\begin{enumerate}[(i)]
\item \label{(3)} $\Gamma(\alpha)\neq \Gamma(\beta)$ and $\Gamma^*(\alpha)\neq \Gamma^*(\beta)$; and
\item \label{(4)} each of $\Gamma(\alpha)\setminus \Gamma(\beta)$,
$\Gamma(\beta)\setminus\Gamma(\alpha)$, $\Gamma^*(\alpha)\setminus \Gamma^*(\beta)$ and $\Gamma^*(\beta)\setminus \Gamma^*(\alpha)$ is
 non-empty.
\end{enumerate}
\end{lemma}
\begin{proof}
(i) Define the $G$-congruence $\equiv$ on $M$, putting $u\equiv v$ if and only if $\Gamma(u)=\Gamma(v)$. 
If $\alpha\equiv \beta$,
then, by primitivity, $u\equiv v$ for any $u,v\in M$, and it follows that $M$ has no arcs, so is 2-homogeneous, contrary to assumption. 
The same argument applies for $\Gamma^*$.

(ii) Suppose that $\alpha\Rightarrow \beta$ implies that $\Gamma(\beta)\supset \Gamma(\alpha)$ (the other cases are similar). There is 
a partial order
$<$ on $M$ given by
$\alpha<\beta$ if and only if $\Gamma(\beta)\supset \Gamma(\alpha)$, and we have that $\alpha\Rightarrow \beta$ implies 
$\alpha<\beta$. 
In particular, $L_{10}$ does not embed. Likewise, $L_{12}$ does not embed in $M$; for if $\beta\Rightarrow \alpha$ then 
$\Gamma(\alpha)\supset \Gamma(\beta)$ so 
there is no vertex $\gamma$ in $\Gamma(\beta)\setminus \Gamma(\alpha)$.

\

\noindent  \textbf{Case 1:} \emph{$\Rightarrow$ is transitive.}

\

\noindent In this case,
 $\Rightarrow$ is a 
partial order $<$ on $M$ such that the `principal ideal' $\{y: y \Rightarrow \alpha\}$  generated by 
every element $\alpha$ is linearly ordered; for if $y_1\Rightarrow \alpha$ and $y_2\Rightarrow \alpha$ with $y_1\neq y_2$, we cannot have $y_1 \to y_2$ for then $y_2\in \Gamma(y_1)\setminus \Gamma(\alpha)$, and similarly $y_2\to y_1$ is impossible, so $y_1\Rightarrow y_2$ or $y_2 \Rightarrow y_1$. Furthermore, we claim that any two elements of $M$ have a common lower bound. For let  $\sim$ be the adjacency relation for the comparability graph of $(P,<)$ (so $x\sim y \Leftrightarrow x<y \mbox{~or~} y<x$). Since $\Aut(M)$ is primitive and there are comparable pairs, this graph is connected.
Suppose for a contradiction that $x,y\in M$ have no common lower bound, and let $x=x_0\sim x_1\sim... \sim x_n=y$ be a path of minimal length between them. If $x_1>x_0$ then by minimality of $n$, $x_2<x_1$, a contradiction as then  the principal ideal below $x_1$ is not linearly ordered. So $x_1<x_0$, and then by minimality of $n$, $x_2>x_1$. We cannot have $x_2=y$, and by minimality of $n$ must have $x_3<x_2$, in which case the principal ideal below $x_2$ is not linearly ordered, again a contradiction.

Thus, 
the structure $(M,\Rightarrow)$ (so just with the relation $\Rightarrow$)
is a {\em semilinear order} (a partial order in which each principal ideal is totally ordered  and any two elements have a common lower bound). Also, $(M,\Rightarrow)$ is 2-set-homogeneous. 
Countable 2-set-homogeneous semilinear orders are classified in
 in Droste \cite{droste}, under slightly different terminology (they are called countable 2-transitive
trees). 
 We recall that a poset is called Dedekind--MacNeille complete if each non-empty upper bounded subset has a supremum, or equivalently, each non-empty lower bounded subset has an infimum. For any poset $M$ there is a unique (up to isomorphism fixing $M$ pointwise) extension $\overline{M}$ of $M$ which is Dedekind--MacNeille complete, called the \emph{Dedekind--MacNeille} completion of $M$; see \cite{droste} for more details.
According to Droste's classification, the isomorphism type of $M = (M,\Rightarrow)$ is determined by whether or not the 
`ramification points'
(greatest lower bounds of incomparable pairs) are in $M$ or $\bar{M}\setminus M$, and the `ramification order', or number
 of `cones' at each ramification point.
 
Choose
$\{\alpha_i:i\in {\mathbb N}\}$ with $\alpha_i\Rightarrow \alpha_j$ whenever $i<j$, and $\beta_i$ such that
$\alpha_i\Rightarrow \beta_{i}$ and $\beta_i$ is incomparable to $\alpha_{i+1}$ in the semilinear order. Then $(\{\beta_i:i\in {\mathbb N}\},\to)$ is a tournament. 
 By Ramsey's Theorem, there are $i_1<i_2<i_3<i_4 \in {\mathbb N}$ such that
$B:=\{\beta_{i_1},\beta_{i_2},\beta_{i_3},\beta_{i_4}\}$ is linearly ordered by $\rightarrow$  (we colour a 2-subset $\{i,j\}$ of ${\mathbb N}$ with $i<j$ red if $\beta_i \to \beta_j$, and green otherwise, and pick a 4-element monochromatic subset of ${\mathbb N}$). By structural properties of the semilinear order under consideration, we may choose 
$C:=\{\gamma_i:i\in {\mathbb N}\}$
and
$D:=\{\delta_i:i\in {\mathbb N}\}$ such that $C\cup D$ is an  antichain of $(M,\Rightarrow)$, and for some ramification point $\alpha$ in the 
 completion $\bar{M}$ of $M$, any pair
 $\gamma_i,\delta_j$
has greatest lower bound $\alpha$, but any two of the $\gamma_i$, or of the $\delta_i$, have greatest lower bound greater than 
$\alpha$. Since any distinct 
$\gamma,\delta \in C \cup D$
are related by $\rightarrow$, it is now possible to find a 4-set $E$, containing 2 elements from $C$
 and 2 elements from $D$, 
which is linearly ordered by $\rightarrow$. To see this, observe
first that, after  replacing D by an infinite subset if necessary, without loss there is $c_1\in C$ such that $c_1 \rightarrow y$ for all $y\in D$. If now there is $c_2 \neq c_1$ dominating two elements $d_1,d_2$ of $D$, then $\{c_1,c_2,d_1,d_2\}$ is totally ordered by $\rightarrow$. If  there is no such $c_2$, then there are distinct $d_1,d_2\in D$ dominating two distinct elements $c_1,c_2$ of $C$, and again $\{c_1,c_2,d_1,d_2\}$ is totally ordered by $\rightarrow$. 

Finally,  by 4-set-homogeneity, there is $g\in G$ with $B^g=E$. This is  impossible, for $g$ induces an automorphism of $\bar{M}$, and the closures in $\bar{M}$ of $B$ and $E$ (where we add infima of $\Rightarrow$-incomparable pairs) are non-isomorphic. 

\

\noindent  \textbf{Case 2:} \emph{$\Rightarrow$ is not transitive.} 

\

\noindent Now $L_1$ does not embed in $M$, for in the notation of  $L_1$, 
$\gamma\in \Gamma(\beta)\setminus \Gamma(\alpha)$, contradicting $\beta\Rightarrow \alpha$.
 Hence, since $L_{10}$ does not embed, if
 $\beta\Rightarrow \alpha\Rightarrow\gamma$ then
$\beta\Rightarrow \gamma$ or $\gamma\rightarrow\beta$, and since $\Rightarrow$ is not transitive the latter case ($L_{14}$) does occur. It follows that 
$\beta\rightarrow\alpha$ implies $\alpha<\beta$, for by 2-set-homogeneity there is $\gamma$ so that $\alpha\Rightarrow \gamma\Rightarrow \beta$, so $\alpha<\gamma<\beta$, so $\alpha<\beta$. 
Thus, $\alpha\Rightarrow \beta$ and $\beta \to \alpha$ each imply $\alpha<\beta$. Since any two distinct elements of $M$ are related by $\Rightarrow$ or by $\to$, it follows that the partial order
$<$ on $M$ is
  a total order, with
$\alpha<\beta$ if and only if $\alpha\Rightarrow \beta$ or $\beta\rightarrow \alpha$.

Given $\alpha$, let $C_\alpha:=\{\beta:\alpha\Rightarrow \beta\}$, and
$D_\alpha:=C_\alpha \cup \bigcup_{\beta\in C_\alpha} C_\beta$. 
First observe that $C_\alpha$ is an initial segment of $\{x\in M: \alpha<x\}$.
 Indeed, 
suppose $\beta\in C_\alpha$
and $\alpha<\gamma<\beta$. If $\gamma\not\in C_\alpha$ then $\gamma\rightarrow \alpha$, so $\alpha\in \Gamma(\gamma)$; 
thus, by the definition of $<$, $\alpha\in \Gamma(\beta)$, a contradiction. 
It follows that if $\beta\in C_\alpha$ then $C_\beta$ is an initial segment of $\{x\in M:\beta<x\}$, so $C_\alpha \cup C_\beta$ is an initial segment of $\{x\in M: \alpha<x\}$, and in fact $D_\alpha$ is an initial segment of $\{x\in M:\alpha<x\}$. Also, $C_\alpha$ is a {\em proper} initial segment of $\{x\in M:\alpha<x\}$, for there is $\gamma \in M$ with $\gamma\rightarrow \alpha$, and then $\alpha<\gamma$ with $\gamma\not\in C_\alpha$.

Next, $D_\alpha=\{x\in M: \alpha<x\}$. For if $\gamma\in M$ with $\gamma>C_\alpha$ then $\gamma\rightarrow \alpha$, so 
as $L_{14}$ embeds in $M$,
by 2-set-homogeneity there is $\beta\in C_\alpha$ with $\beta\Rightarrow \gamma$, so $\gamma \in C_\beta\subset D_\alpha$.
Observe also that for such $\alpha,\beta,\gamma$ we have $\Sup C_\alpha<\Sup C_\beta$, so whenever we have 
$\beta_1\Rightarrow \beta_2$ we have $\Sup C_{\beta_1} <\Sup C_{\beta_2}$.
It follows that if $\beta_1,\beta_2\in C_\alpha$ with $\beta_1<\beta_2$ then $\Sup C_\alpha<\Sup C_{\beta_1}$ so
 $\beta_1\Rightarrow \beta_2$ so
$\Sup C_{\beta_1}<\Sup C_{\beta_2}$. 

Fix $\beta,\gamma$ as in the last paragraph, so $\alpha\Rightarrow \beta$, $\beta\Rightarrow \gamma$,  and $\gamma>C_\alpha$.
If $\alpha<\beta'<\beta$ then $\Sup C_{\beta'}<\Sup C_\beta<\Sup C_\gamma$, and if $\beta<\beta'\in C_\alpha$ then 
$\Sup C_\beta<\Sup C_{\beta'}$, so $\beta'\Rightarrow \gamma$, so $\Sup C_{\beta'}<\Sup C_\gamma$.
Together, using the fact that $C_\alpha$ is an initial segment of $\{x\in M:\alpha<x\}$, these give that $x< \Sup C_{\gamma}$ for all $x\in  C_\alpha$. Hence, if $\delta\in M$ with $\delta\rightarrow \gamma$, then 
$\delta>\gamma>\alpha$ and $\delta\not\in C_\gamma$, so $\delta\not\in D_\alpha$, contradicting the first assertion of
 the last paragraph.
\end{proof}

\begin{lemma}\label{FromNowOn}
If $\beta \Rightarrow \alpha$ then either there exists $\gamma$ with $\beta\rightarrow \gamma$
 and $\alpha||\gamma$, or there exists $\delta$ with $\alpha \rightarrow \delta$ and $\beta \parallel \delta$. 
\end{lemma}
\begin{proof}
Suppose not. Then by Lemma~\ref{(3)and(4)}
 there are $\gamma,\delta$ with
$\beta \rightarrow \gamma\rightarrow \alpha$ and $\alpha\rightarrow\delta\rightarrow\beta$,
and there is an automorphism taking $(\beta,\gamma,\alpha)$ to $(\alpha,\delta,\beta)$, swapping $\alpha$ and $\beta$, which is impossible.
\end{proof}

So from now on, we may suppose without loss of generality that if $\beta\Rightarrow \alpha$ then 
there is $\gamma$ with $\beta\rightarrow \gamma$ and $\alpha||\gamma$. For,
 if the other possibility arising from Lemma~\ref{FromNowOn} held, then we would continue with the argument below, but with the orientation of every $\Rightarrow$-arc reversed. 

\begin{lemma}\label{(5)}
Suppose $\beta\Rightarrow \alpha$. Then 
\begin{enumerate}[(i)]
\item there is $\delta$
such that
$\alpha\rightarrow \delta\rightarrow \beta$;
\item there is no directed $2$-path $\beta\rightarrow \eta\rightarrow \alpha$, that is, $L_7$ does not embed;
\item there is $\epsilon||\beta$ with $\epsilon \rightarrow \alpha$;
\item there is no $\eta$ with $\eta\rightarrow \beta$ and $\eta||\alpha$; that is, $L_{12}$ and $L_{14}$ do not embed.
 \end{enumerate}
\end{lemma}
\begin{proof}
By our assumption there is $\gamma$ with $\beta \to \gamma$ and $\alpha||\gamma$.

(i) By Lemma~\ref{(3)and(4)}(ii) there is $\delta\in \Gamma(\alpha)\setminus \Gamma(\beta)$, and by Lemma~\ref{(1)and(2)}(ii) (applied with $\alpha'=\delta$ and $\beta'=\gamma$) 
 we cannot have $\delta||\beta$. Hence $\delta \to \beta$. 

(ii) If such $\eta$ exists, then by 3-set-homogeneity there is $g\in G$ mapping $(\alpha,\delta,\beta)$ to $(\beta,\eta,\alpha)$ (where $\delta$ is as in (i)), and such 
$g$ swaps
$\alpha$ and $\beta$, a contradiction.

(iii) By Lemma~\ref{(3)and(4)}(ii) 
there is $\epsilon\in \Gamma^*(\alpha)\setminus \Gamma^*(\beta)$, and by (ii) we must have $\epsilon||\beta$. 

(iv) Suppose such $\eta$ exists. Then there is $g\in G$ inducing $(\epsilon,\alpha,\beta)\mapsto (\eta,\beta,\alpha)$ (with $\epsilon$ as in (iii))
and so swapping $\alpha$ and $\beta$, a contradiction.
\end{proof}

\begin{lemma}\label{(6)}
Suppose $\beta\Rightarrow \alpha$. Then 
\begin{enumerate}[(i)]
\item for any $\gamma\in \Gamma(\beta)\setminus \Gamma(\alpha)$, $\alpha\Rightarrow \gamma$;
\item for any $\gamma\in \Gamma^*(\alpha)\setminus\Gamma^*(\beta)$, $\gamma\Rightarrow \beta$ (so $L_{13}$ does not embed);
\item $\Gamma(\beta)\setminus \Gamma(\alpha)$ is an independent set.
\end{enumerate}
\end{lemma}
\begin{proof}
(i) We cannot have $\gamma\rightarrow \alpha$ by Lemma~\ref{(5)}(ii). By Lemma~\ref{(5)}(iv), $\gamma\Rightarrow \alpha$ is impossible.
 The only remaining possibility is
$\alpha\Rightarrow \gamma$.

(ii) By Lemma~\ref{(5)}(ii), we cannot have $\beta \to \gamma$. If $\beta \Rightarrow \gamma$, then, by
 considering a map $(\beta, \gamma) \mapsto (\beta, \alpha)$,
there is $\delta \in \Gamma(\alpha)$ with $\delta ||\beta$, contrary to Lemma~\ref{(1)and(2)}(ii) (applied with ($\alpha'=\delta$). Thus,
 $\gamma \Rightarrow \beta$.

(iii) Suppose $\gamma,\gamma'\in \Gamma(\beta)\setminus\Gamma(\alpha)$, and $\gamma \rightarrow \gamma'$. By (i), $\alpha\Rightarrow\gamma$ and
$\alpha\Rightarrow \gamma'$. But now the configuration $\{\alpha,\gamma,\gamma'\}$ is a copy of $L_{13}$, contrary to (ii). \end{proof}

\begin{lemma}\label{(7)}
\
\begin{enumerate}[(i)]
\item There is no $\Gamma$-$3$-chain $\alpha\rightarrow\beta\rightarrow \gamma$ with $\alpha\rightarrow \gamma$.
\item  $\Gamma(\alpha)$ and $\Gamma^*(\alpha)$ are each independent sets.
\item  $\Lambda(\alpha)$ and $\Lambda^*(\alpha)$ are independent sets.
\item $L_6$ embeds in $M$, but $L_{10}$ does not embed.
\item  Each of $\Gamma(\alpha)$, $\Gamma^*(\alpha)$, $\Lambda(\alpha)$ and $\Lambda^*(\alpha)$ is linearly ordered 
by $\Rightarrow$, densely and without endpoints.
\end{enumerate}
\end{lemma}
\begin{proof}
(i) Suppose such $\alpha,\beta,\gamma$ exist. Since $\alpha\rightarrow\gamma$, and in 
view of the triple $(\alpha,\beta,\gamma)$ of Lemma~\ref{(6)}(i),
there is $\delta$ with $\alpha\Rightarrow\delta\Rightarrow \gamma$. We cannot have 
$\delta||\beta$, since 
$\delta\Rightarrow \beta$ implies $(\delta,\beta,\gamma)$ carries $L_{13}$, and $\beta\Rightarrow \delta$ implies 
$(\alpha,\beta,\delta)$ carries $L_{12}$.
  Now $\beta\rightarrow\delta$ is impossible 
since  $\{\alpha,\beta,\delta\}$ would be isomorphic to $L_7$, contrary to Lemma~\ref{(5)}(ii), and likewise $\delta\rightarrow \beta$ 
is impossible because otherwise 
$\{\gamma,\delta,\beta\}$ carries $L_7$.

(ii) This is immediate from (i).

(iii) It suffices to observe that the configurations $L_{12}$ and $L_{13}$ do not embed. For these, see Lemmas~\ref{(5)}(iv) and \ref{(6)}(ii).

(iv) By Lemma~\ref{(1)and(2)}(i), each of $\Gamma(\alpha)$, $\Gamma^*(\alpha)$, $\Lambda(\alpha)$ and $\Lambda^*(\alpha)$ is infinite. Thus, by (iii) above, $L_6$ embeds in $M$, so by $3$-set-homogeneity, as $L_6$ and $L_{10}$ are digraph-isomorphic, $L_{10}$ does not embed.

(v)  It follows from (iv) that each of
$\Gamma(\alpha)$, $\Gamma^*(\alpha)$, $\Lambda(\alpha)$ and $\Lambda^*(\alpha)$ is linearly ordered by $\Rightarrow$. By $3$-set-homogeneity, it follows that $G_\alpha$ acts 2-homogeneously
on each of the four sets, so in each case,  the linear order is dense and without endpoints.
\end{proof}

Now fix $\alpha \in M$. It follows from Lemma~\ref{(7)}(v) that there is a unique linear ordering $<_\alpha$ on $M$, with $\alpha$ as the first point, with convex subsets
$\Lambda(\alpha)$, $\Lambda^*(\alpha)$, $\Gamma(\alpha)$, $\Gamma^*(\alpha)$ on which $<_\alpha$ extends $\Rightarrow$, and with
$\Lambda(\alpha)<\Gamma(\alpha)<\Gamma^*(\alpha)<\Lambda^*(\alpha)$. There is a corresponding 
\emph{circular ordering} (in the sense, say,  of \cite[11.3.3]{bmmn})
$K_\alpha$ on $M$ defined from $<_\alpha$, with $K_\alpha(x,y,z)$ holding if and only if $x<_\alpha y<_\alpha z$
 or $y<_\alpha z<_\alpha x$ or $z<_\alpha x<_\alpha y$. We shall show that $K_\alpha$ is $G$-invariant (not just 
$G_\alpha$-invariant). 

\begin{lemma}\label{(8)}
The structures $L_1-L_4$ and $L_6$ embed in $M$, but $L_7$--$L_{14}$ do not embed in $M$.
\end{lemma}
We shall show in the next lemma that $L_5$ also embeds in $M$. 

\begin{proof}
The structure $L_1$ embeds by Lemma~\ref{(6)}(i); $L_2$ and $L_3$ embed since $\Gamma(\alpha)$ and $\Gamma^*(\alpha)$ are
 infinite (by Lemma~\ref{(1)and(2)}(i))
and are independent (by Lemma~\ref{(7)}(ii));
$L_4$ embeds by Lemma~\ref{(5)}(i), and $L_6$ by Lemma~\ref{(7)}(iv). 

The structure $L_7$, $L_{10}$, and $L_{11}$  do not embed, by Lemmas~\ref{(5)}(ii), \ref{(7)}(iv), and \ref{(7)}(i), respectively. Also,
 $L_{12}$, $L_{13}$ 
and $L_{14}$
do not embed, by 
 Lemma~\ref{(5)}(iv), Lemma~\ref{(6)}(ii),  and  Lemma~\ref{(5)}(iv),  respectively. Since $L_8$ is `$L_{12}$ or $L_{14}$', and $L_9$ is `$L_{13}$ or $L_{14}$', it follows that these do not embed too.
\end{proof}

 Suppose now that $\Sigma_1,\Sigma_2$ are distinct members of 
$\{\Lambda(\alpha),\Gamma(\alpha),\Gamma^*(\alpha),\Lambda^*(\alpha)\}$, and are viewed as sets totally ordered 
by $<_\alpha$.
 We say that \emph{$(\Sigma_1,\Sigma_2)$ is of $(\Phi,\Psi)$-type} if $\Phi,\Psi$ are distinct elements 
of $\Gamma,\Gamma^*,\Lambda,\Lambda^*$ and

\begin{enumerate}[(A)] 
\item for any $\beta\in \Sigma_1$, the set $\Phi(\beta)\cap\Sigma_2(\alpha)$ is a proper non-empty final segment of
 $\Sigma_2$ with no least element, and if $\beta_1,\beta_2\in \Sigma_1$ are distinct, then
 $\Phi(\beta_1)\cap \Sigma_2\neq \Phi(\beta_2)\cap \Sigma_2$;
\item  for any $\beta\in \Sigma_1$, we have $\Sigma_2\setminus \Phi(\beta)\subseteq \Psi(\beta)$;
\item  for any $\gamma\in \Sigma_2$, the set $\Phi^*(\gamma)\cap\Sigma_1$ is a proper non-empty initial segment of $\Sigma_1$ with no
 greatest element, 
and if $\gamma_1,\gamma_2\in \Sigma_2$ are distinct then $\Phi^*(\gamma_1)\cap \Sigma_1\neq \Phi^*(\gamma_2) \cap \Sigma_1$;
\item for any $\gamma\in \Sigma_2$, we have $\Sigma_1\setminus\Phi^*(\gamma)\subseteq \Psi^*(\gamma)$.
\end{enumerate}

We now determine the types for all possible pairs $\Sigma_1,\Sigma_2$. 

\begin{lemma}\label{(9)}
\
\begin{enumerate}[(i)]
\item $(\Lambda(\alpha),\Gamma(\alpha))$ is of type $(\Gamma,\Lambda)$.
\item $(\Lambda(\alpha),\Gamma^*(\alpha))$ has type $(\Gamma^*,\Gamma)$.
\item $(\Lambda(\alpha),\Lambda^*(\alpha))$ has type $(\Lambda^*,\Gamma^*)$.
\item $(\Gamma(\alpha),\Lambda^*(\alpha))$ has type  $(\Gamma^*,\Gamma)$.
\item $(\Gamma^*(\alpha),\Lambda^*(\alpha))$ has type $(\Gamma,\Lambda)$.
\item $(\Gamma(\alpha),\Gamma^*(\alpha))$ has type $(\Gamma,\Lambda)$.
\item $L_5$ embeds in $M$.
\end{enumerate}
\end{lemma}
\begin{proof}
We just prove (i), 
as parts (ii)-(v) are dealt with using similar arguments to (i). We also comment 
on the proof of (vi) and (vii). We make repeated use of Lemma~\ref{(8)} throughout. 

(A) Let $\beta\in \Lambda(\alpha)$. To see that $I:=\Gamma(\beta)\cap \Gamma(\alpha)$ is a final segment of $\Gamma(\alpha)$, suppose that
  $\gamma,\gamma'\in \Gamma(\alpha)$ with
$\gamma\Rightarrow \gamma'$, and $\beta\rightarrow \gamma$. As $\{\beta, \gamma, \gamma'\}$ cannot be $L_8$, $\beta||\gamma'$ is 
impossible, 
and as $\{\alpha, \beta,\gamma'\}\not\cong L_7$ we cannot have
$\gamma'\rightarrow \beta$, so $\beta\rightarrow \gamma'$. To see that $I\neq \varnothing$, observe that by 2-set-homogeneity,
$\{\alpha,\beta\}$ lies in a copy of $L_2$. Also, $I\neq \Gamma(\alpha)$ as $\{\alpha,\beta\}$ lies in a copy of $L_1$. If $I$ has a least 
element, $\gamma$ say, and $\gamma'\in I\setminus\{\gamma\}$,
then by 3-set-homogeneity there is $g\in G$ with $(\alpha,\beta,\gamma)^g=(\alpha,\beta,\gamma')$, and this is clearly impossible.

If $\beta_1,\beta_2\in \Lambda(\alpha)$ are distinct and $\Gamma(\beta_1) \cap \Gamma(\alpha)=\Gamma(\beta_2)\cap \Gamma(\alpha)$,
 there is a $G_\alpha$-congruence on $\Lambda(\alpha)$ given by
$u\equiv v$ if and only if
$\Gamma(u)\cap \Gamma(\alpha)=\Gamma(v) \cap \Gamma(\alpha)$. By 2-homogeneity of 
$G_\alpha$ on $\Lambda(\alpha)$, there is a single $\equiv$-class on $\Lambda(\alpha)$. In this case, if
 $\beta\in\Lambda(\alpha)$ and $\gamma\in \Gamma(\beta)\cap\Gamma(\alpha)$, then for all 
$\beta'\in \Lambda(\alpha)$ we have $\beta'\rightarrow\gamma$. This contradicts that $\{\alpha,\gamma\}$ lies in a copy of $L_1$.

(B) To see that if $\beta\in \Lambda(\alpha)$ then $\Gamma(\alpha)\setminus\Gamma(\beta)\subseteq \Lambda(\beta)$, observe that
$L_7$ and $L_{12}$ do not embed in $M$.

(C) Let $\gamma\in \Gamma(\alpha)$. To see that $J:= \Gamma^*(\gamma)\cap\Lambda(\alpha)$ is an initial segment of
 $\Lambda(\alpha)$, observe that 
if $\beta,\beta'\in \Lambda(\alpha)$ with $\beta\Rightarrow\beta'$ and $\beta'\rightarrow \gamma$, then $\beta,\gamma$ are not
 independent as otherwise $\{\beta,\beta',\gamma\}\cong L_{9}$, and 
 $\gamma\not\rightarrow\beta$ as otherwise $\{\alpha,\beta,\gamma\} \cong L_7$. Now $J\neq \varnothing$ as $\{\alpha,\gamma\}$ lies in some
 $L_2$, and $J\neq \Lambda(\alpha)$ as $\{\alpha,\gamma\}$ lies in some $L_1$. The facts that $J$ has no greatest element,
 and that distinct elements $\gamma$ determine distinct sets $J$, are much as in (A).
 
(D) Suppose $\beta\in \Lambda(\alpha)\setminus \Gamma^*(\gamma)$. Then $\gamma\not\rightarrow \beta$ as 
$\{\alpha,\beta,\gamma\}\not\cong
L_7$, and $\gamma\not\Rightarrow \beta$ as $\{\alpha,\gamma,\beta\}\not\cong L_{12}$. Thus, $\beta\Rightarrow\gamma$.

This proves (i). We now make some comments on (vi) and (vii). 

Suppose $\beta\in\Gamma(\alpha)$ and $\gamma\in \Gamma^*(\alpha)$. Then $\gamma\not\Rightarrow \beta$ as
$\{\alpha,\beta,\gamma\}\not\cong L_7$, and $\gamma\not\rightarrow\beta$ as $\{\alpha,\beta,\gamma\}\not\cong L_{11}$. Thus,
 the only possibilities are that
$\beta\Rightarrow\gamma$ and $\beta\rightarrow\gamma$. Clearly, for any $\beta\in \Gamma(\alpha)$ there is some 
$\gamma\in \Gamma^*(\alpha)$ with 
$\beta\Rightarrow\gamma$, as $\{\alpha,\beta\}$ lies in some
$L_4$. The key point is to check that $\{\alpha,\beta\}$ also lies in some $L_5$, i.e., that $L_5$ embeds in $M$ (so (vii) holds); for then there is $\gamma\in \Gamma^*(\alpha)$ with $\beta \to \gamma$, and the remaining details of (vi) follow easily as in (i).

So suppose for a contradiction that $\beta\Rightarrow\gamma$ for all $\beta\in \Gamma(\alpha)$ and $\gamma\in \Gamma^*(\alpha)$.
By (iv), there are $\delta_1,\delta_2\in \Lambda^*(\alpha)$ with $\beta\rightarrow\delta_1$ and $\delta_2\rightarrow\beta$. 
Then $\delta_1\Rightarrow\delta_2$, as $\{\beta,\delta_1,\delta_2\} \not\cong L_7$. Using (v), pick $\gamma_1,\gamma_2\in \Gamma^*(\alpha)$, with $\gamma_1\rightarrow\delta_1$ 
and $\gamma_2\Rightarrow\delta_2$. 
By 3-set-homogeneity, there is $g\in G_{\alpha\beta}$ with $\gamma_1^g=\gamma_2$. Then 
$\gamma_2 \rightarrow \delta_1^g$, so as $\gamma_2 \Rightarrow \delta_2$ and $\delta_1^g,\delta_2\in \Lambda^*(\alpha)$, by (v) we have 
$\delta_2\Rightarrow \delta_1^g$. Thus, as  $\delta_2\to \beta$, by (iv) we also have
$\delta_1^g\rightarrow \beta$, a contradiction as $\beta\rightarrow\delta_1$ and $\beta^g=\beta$. This yields (vii), and hence (vi).
\end{proof}

There is a natural notion of {\em convex subset} of a circular ordering (in the sense of \cite[11.3.3]{bmmn}) $K$ on $M$: here, $I\subset M$ is convex if, given
 distinct
$\beta,\gamma\in I$, either $K(\beta,\gamma,\alpha)$ for all $\alpha\in M\setminus I$, or $K(\beta,\alpha,\gamma)$ for all 
$\alpha \in M\setminus I$. With $<_\alpha$ defined as in the paragraph before Lemma~\ref{(8)}, a subset of $M$ will be convex in the sense of the circular ordering on $M$, if and only if it is convex in the sense of the linear order, or is the union of an initial and a final segment of $(M,<_\alpha)$.

\begin{lemma}\label{(10)}
\
\begin{enumerate}[(i)]
\item The circular ordering $K_\alpha$ is independent of the choice of $\alpha$, so can be denoted by $K$.
\item For any $\beta$, each $G_\beta$-orbit is a convex subset of $M$ with respect to $K$.
\item Suppose that $A=\{\alpha_1,\ldots,\alpha_n\}\subset M$ is finite, and $A_i$ is an infinite $G_{\alpha_i}$-orbit for each $i=1,\ldots,n$, and that $I:=A_1\cap \ldots\cap A_n \neq \emptyset$. Then $I$ is convex. 
\end{enumerate}
\end{lemma}
\begin{proof}
This follows easily from Lemma~\ref{(9)}.

(i) Consider any $\alpha'\neq \alpha$; there are four cases to consider -- the four $G_\alpha$
 orbits on $M\setminus\{\alpha\}$. In each case, patch together the circular ordering $K_{\alpha'}$ from
$\Lambda(\alpha')$, $\Gamma(\alpha')$, $\Gamma^*(\alpha')$, $\Lambda^*(\alpha')$, and check that 
$K_{\alpha'}$ agrees with $K_\alpha$. 

(ii) This follows immediately from (i), since clearly, by construction of $K_\alpha$,
 each $G_\alpha$-orbit is a convex subset of $K_\alpha$.
 
 (iii) It suffices to observe that if $I$ is one of $\Lambda(\alpha)$, $\Gamma(\alpha)$, $\Gamma^*(\alpha)$, $\Lambda^*(\alpha)$, and $\beta\in M$ with $\beta\neq \alpha$, then there is no $G_\beta$-orbit $J$ which intersects $I$ properly in the union of an initial segment and a final segment of $I$. This follows by a case analysis from Lemma~\ref{(9)}. The main point is that  each $G_\beta$-orbit $J$ is convex and, by Lemma~\ref{(9)}, lies 
in the union of $\{\alpha\}$ and {\em two} $G_\alpha$-orbits. This suffices, for if $J$ intersected $I$ properly in the union of an initial and final segment of $I$, then $M$ would equal
$I\cup J$ and would be the union of $\{\alpha\}$ and three $G_\alpha$-orbits, which is impossible.
\end{proof}

\begin{proposition} \label{recover1} Let $M$ be a set-homogeneous but not $2$-homogeneous countably infinite a-digraph, and suppose that
$G:=\Aut(M)$ is primitive. Then $M\cong T(4)$.
\end{proposition}
\begin{proof}
Since $T(4)$ has the required properties (by Lemma~\ref{not2hom}), it suffices to show that if $M,M'$ are 
countably infinite set-homogeneous 
but not 2-homogeneous a-digraphs with
primitive automorphism group, then $M\cong M'$.  By the above analysis, $M'$ carries the same relation $\Rightarrow$, so 
that Lemmas~\ref{(1)and(2)}--\ref{(10)} hold for $M'$ as well as for $M$. 
We shall show by a back and forth argument that $M\cong M'$, where $M$, $M'$ are viewed as structures in the expanded 
language with relation symbols
 $\rightarrow$, $\Rightarrow$. So suppose that $A\subset M$, $A'\subset M'$ are finite, $f:A\rightarrow A'$ is an 
isomorphism,
 and $\alpha\in M\setminus A$. We must find $\beta\in M'$ so that $f$ extends to a partial isomorphism 
(of $\{\rightarrow,\Rightarrow\}$-structures)
 with $f(\alpha)=\beta$.

Let $I:=\{y\in M: (x,\alpha)\cong (x,y)\mbox{~for all~} x\in A\}$.  Put
$A:=\{\alpha_1,\ldots,\alpha_m\}$, and $\beta_i:=f(\alpha_i)$ for each $i=1,\ldots,m$. For each $i=1,\ldots,m$,
let $A_i:=\{y\in M: (\alpha_i,y)\cong (\alpha_i,\alpha)\}$. Thus, $I:= A_1\cap \ldots \cap A_m$. 
By Lemma~\ref{(10)}(iii), $I$ is a convex subset of $M$
under the circular 
ordering $K$.
Also put 
$B_i:= \{y\in M': (\beta_i,y)\cong (\alpha_i,\alpha)\}$. It suffices to show that $J:=B_1 \cap \ldots \cap B_m\neq \varnothing$,
 for then, if
$\beta\in J$, then $f$ extends to an isomorphism $f':A\cup\{\alpha\}\to A' \cup\{\beta\}$ with $f'(\alpha)=\beta$.

By considering how convex sets in circular orderings intersect (bearing in mind Lemma~\ref{(10)}(iii)), we see that there are $r,s\in \{1,\ldots,m\}$ such that $I=A_r \cap A_s$. We suppose $r\neq s$, this being the harder case.
 Let $I':= B_r \cap B_s$. Then 
$B_r\cap B_s \neq \varnothing$. For whether 
or not an intersection $B_i\cap B_j$ is empty
 is determined by which configurations $L_1-L_{14}$ embed in $M'$, and the same members, namely $L_1$--$L_6$, embed in both 
$M$ and $M'$. 

We claim that $I'=J$. To see this, we must show that for each
 $i\in \{1,\ldots,m\}\setminus \{r,s\}$, $B_r \cap B_s \subseteq B_i$. For this we view $\alpha_i,\beta_i$ 
(in the structures $M,M'$ respectively) as playing the role of $\alpha$ in Lemma~\ref{(10)} above. Thus, $M\setminus\{\alpha_i\}$ is 
partitioned into convex subsets
$\Lambda(\alpha_i),\Gamma(\alpha_i),\Gamma^*(\alpha_i),\Lambda^*(\alpha_i)$, as is $M'$ over $\beta_i$. There are orderings
$<_{\alpha_i}$, 
$<_{\beta_i}$ on $M$ and $M'$ respectively such that $\alpha_r$ and $\beta_r$ lie in corresponding
 convex subsets of $M$ and $M'$ (i.e. corresponding orbits over $\alpha_i$ and $\beta_i$ respectively), and
$\alpha_r<_{\alpha_i}\alpha_s$ if and only if $\beta_r<_{\beta_i} \beta_s$. 
We may suppose that $I$ is an initial segment of $A_r$ and a final segment of $A_s$ (with respect to $<_{\alpha_i}$). 
Now the isomorphism type 
of $\{\alpha_r,\alpha_i\}$ 
ensures that
an initial segment of $A_r$ lies in $A_i$, so  an initial segment of $B_r$ lies in $B_i$ (with respect to $<_{\beta_i}$). Likewise, the isomorphism type of 
$\{\alpha_s,\alpha_i\}$ ensures that
$A_s$ has a final segment in $A_i$, so $B_s$ has a final segment in $B_i$ (essentially, these assertions are because 
$M$ and $M'$ both satisfy Lemma~\ref{(10)}). Now it is an easy consequence of Lemmas ~\ref{(9)} and ~\ref{(10)} that $M$ cannot be the union of  three convex sets of the form 
$B_r$, $B_s$, $B_i$. As $B_r \cap B_s \neq \varnothing$, 
 it 
follows that 
$B_r \cap B_s \subseteq B_i$, as required.
\end{proof}

\subsection*{The imprimitive case}

We begin with a lemma used in the next proposition, and relevant to the discussion in  Section
7.

\begin{lemma} \label{extra}
Let $M$ be a countably infinite connected set-homogeneous a-digraph and suppose that $G=\Aut(M)$ preserves a nontrivial block system $\{B_n : n\in I\}$. Then one of the following holds, for some set-homogeneous tournament $T$ and $n\leq \aleph_0$.
\begin{enumerate}[(i)]
\item $M\cong \bar K_n[T]$, $M$ is $2$-homogeneous, and each $B_n$ induces $T$;
\item $M\cong T[\bar K_n]$, $M$ is $2$-homogeneous, and each $B_n \cong\bar K_n$;
\item Each $B_n$ is an independent set, and for distinct $x,y\in B_n$, $\Gamma(x)\ne \Gamma(y)$; also, for distinct $i,j\in I$, there are arcs in both directions between $B_i$ and $B_j$.
\end{enumerate}
\end{lemma}

\begin{proof}
First note that in cases (i) and (ii), $2$-homogeneity of $M$ follows from Lemma~\ref{not2hom}(iii).

Suppose first that some $B=B_n$ contains an arc. Then by $2$-set-homogeneity, two elements in distinct blocks must be independent, and two elements within any block must have an arc between them. Thus, $B$ induces a tournament $T$, and $T$ must be set-homogeneous, and (i) above holds.

Thus, we may assume that there are no arcs within any block. It follows by $2$-set-homogeneity that any two elements in distinct blocks are related by an arc, and hence that $G$ acts $2$-homogeneously on $\{B_n:n\in I\}$. If there are distinct blocks $B_i, B_j$ such that there is an arc from every vertex in $B_i$ to every vertex in $B_j$, then for any two distinct blocks, all arcs go from one to the other. There is then an induced tournament structure on the block system, and  (ii) holds. 

In the remaining case, for any two distinct blocks there are arcs in both directions between them. It follows that there are two vertices $x,y$ in the same block such that $\Gamma(x)\neq \Gamma(y)$. Hence, by $2$-set-homogeneity, this holds for any two vertices in the same block, and (iii) holds.
\end{proof}

\begin{proposition}\label{recover2} Let $M$ be a countably infinite connected set-homogeneous but not $2$-homogeneous a-digraph, and suppose that 
$G:=\Aut(M)$ is imprimitive. Then $M$ is isomorphic to $R_n$ for some $n\geq 2$.
\end{proposition} 

\begin{proof} Let $\{B_n: n \in I\}$ be a non-trivial block system  for $G$. Then since $M$ is not $2$-homogeneous, case (iii) of Lemma~\ref{extra} holds.
It follows that the $G$-action induced on $B_i$ is highly homogeneous (since $M$ is set-homogeneous) but not 2-transitive (since $M$ is not 2-homogeneous).

We claim that the $B_i$ are infinite. Indeed, otherwise, by Fact~\ref{livwag}, $|B_i|=3$ for each $i$. 
Then by 2-set-homogeneity if $i\neq j$ there is $g\in G$ interchanging $B_i$ and $B_j$, but the number of arcs between $B_i$ and $B_j$ is 9, so the number of arcs from $B_1$ to $B_2$ is not equal to the number for $B_2$ to $B_1$, a contradiction.

By the last paragraph and 
an 
unpublished result of
J. P. J. MacDermott (see p.63 of Cameron \cite{cameron2}),  $G$ preserves 
a dense linear order without endpoints $<_i$ on $B_i$. Let  $x,y\in B_i$ 
with $x<_i y$. If there are $z,w\not\in B_i$ such that $x\rightarrow z\rightarrow y$ and $y\rightarrow w \rightarrow x$, then (by 3-set-homogeneity and rigidity of $\{x,y,z\}$)
there is $g\in G$ with $(x,y,z)^g=(y,x,w)$, so $M$ is 2-homogeneous, a contradiction.  Also, 
given such $x,y$, as $\Gamma(x)\neq \Gamma(y)$, either there is $z$ with $x\to z\to y$ or there is $w$ with $y\to w\to x$. 
Hence, reversing some of the $<_i$ if necessary,  we may suppose that
$x<_i y$ if and only if $\Gamma(x) \supset \Gamma(y)$.

It follows that if $B_i$, $B_j$ are distinct then every element $x$ of $B_i$ dominates
 a final segment of $B_j$ and is dominated by an initial segment of $B_j$. For the first statement note that if $x\in B_i$ and $u,v\in B_j$ with $u<_j v$, and $x \to u$, then $x\not\in \Gamma(v)$ as $x\not\in \Gamma(u)$ and $\Gamma(u)\supset \Gamma(v)$, so $x\to v$; the argument is similar for the second statement.  These segments are both non-empty. 
For example, if $x$ is not dominated by any element of $B_j$, then there is a proper non-empty initial segment $I$ of $B_i$ which dominates
 all of $B_j$, and $B_i\setminus I$ is dominated by all $B_j$. Then there is no automorphism 
swapping $B_i$ and $B_j$, which is impossible as there are arcs in both directions.

We will show that $<$ is a total order on $M$,
where
$x<y$ if and only if 
$x\rightarrow y$ or $x,y$ are unrelated and
$\Gamma(x) \supset \Gamma(y)$. 
All points follow from the definition (and the last paragraph) except the case where $x \to y$
and $y \to z$ with $x,z$ in distinct blocks, in which case we need to prove that $x \to z$.  Thus it is sufficient to prove 
that $M$ does not embed a  directed triangle $x,y,z$
with $x \rightarrow y \rightarrow z \rightarrow x$.

\

\noindent {\em Claim.}
For any $3$ blocks $B_1, B_2, B_3$ and for any $a \in
B_2$ there exist $u \in B_1$ and $v \in B_3$ such that
$u \rightarrow a \rightarrow v$ and $u \rightarrow v$.

\

\begin{proof}[Proof of Claim.]

Let $x \in \Gamma^*(a) \cap B_1$ and $y \in \Gamma(a)
\cap B_3$ be arbitrary. Such vertices exist by the fact that the initial and final segments mentioned above are non-empty.
Now if $x \rightarrow y$ then we are done by setting
$u=x$ and $v=y$. 

The other possibility is that $y \rightarrow x$, so suppose this holds.
Recall that
$\Gamma(y) \cap B_1$ is a final segment of $B_1$. Let
$x' \in B_1 \setminus \Gamma(y)$ so that $x' \in
\Gamma^*(y) \cap B_1$ and $x'<x$. By definition of
$<_i$ we have $x' \rightarrow a$ since $x' <_1 x$. But
now if we set $u=x'$ and $v=y$ then $u \rightarrow a
\rightarrow v$ and $u \rightarrow v$, proving the
claim.    
\end{proof}

Now, seeking a contradiction, suppose that $M$ embeds
a directed triangle $b \rightarrow a \rightarrow c
\rightarrow b$ with $b \in B_1$, $a \in B_2$, and $c
\in B_3$. By applying the claim twice we see
that there exist $u, v' \in B_1$ and $v, u' \in B_3$
such that $u \rightarrow a \rightarrow v$, $u
\rightarrow v$, $u' \rightarrow a \rightarrow v'$, and
$u' \rightarrow v'$. By 3-set-homogeneity there is an
automorphism $g \in G$ extending the isomorphism
$(u,a,v) \mapsto (u',a,v')$. Now $a^g=a$ and
$(B_1,B_3)^g = (B_3,B_1)$. Thus $(\Gamma(a) \cap
B_1)^g = \Gamma(a) \cap B_3$, $(\Gamma(a) \cap
B_3)^g = \Gamma(a) \cap B_1$ and the corresponding
statements for $\Gamma^*(a)$ also hold. It follows
that $b,c^g \in B_1$ with $b<c^g$, and $b^g, c \in
B_3$ with $b^g < c$, while $c^g \rightarrow b^g$ and
$c \rightarrow b$. Now from the definition of $<_i$,
$c^g \rightarrow b^g$ and $b^g < c$ implies $c^g
\rightarrow c$, while $c \rightarrow b$ and $b<c^g$
implies $c \rightarrow c^g$. This is a contradiction.
We conclude that $M$ does not embed a directed
triangle and so $<$ is an ordering on $M$.   

\

In order to recover $R_n$ completely, we now have to show the following, for any distinct $i,j\in I$: 

\begin{enumerate}[(i)]
\item 
if  $x,y\in B_i$ with $x<_i y$, then 
there is $z\in B_j$ with $x\rightarrow z \rightarrow y$;
\item 
if $x\in B_i$, $y\in B_j$ with $x\rightarrow y$, then 
there is $x'\in B_i$ with $x<_ix'\rightarrow y$ and there is $y'\in B_j$ with $x\rightarrow y'<_j y$;
\item 
if $x\in B_i$, $y\in B_j$ with $x \to y$, and $k\in I\setminus \{i,j\}$, there is $z\in B_k$ with $x\rightarrow z\rightarrow y$.
\end{enumerate}
From these properties it follows that $(M,<)$ is dense 
without endpoints,  the $B_i$ are dense codense, and hence  $M \cong R_n$ for some $n$
(see the comment above Lemma~\ref{not2hom}).

{\em Proofs of (i)--(iii).}
Given distinct $x,y \in B_i$ there is $z\in B_j$ with $x,y\rightarrow z$.  Hence, 
using 3-set-homogeneity,
$G_{\{B_i\},\{B_j\}}$ acts 2-homogeneously and hence primitively on $B_i$. 

Suppose that (i) does not hold. Then $x,y$ dominate the same elements of $B_j$, so are equivalent under the $G^{B_i}_{\{B_i\},\{B_j\}}$-invariant equivalence relation `dominate the same vertices of $B_j$'. Thus, 
 by primitivity of $G^{B_i}_{\{B_i\},\{B_j\}}$, any two elements of $B_i$ dominate the same elements of $B_j$.
 This is impossible, since case (iii) of Lemma~\ref{extra} holds.
 
 Part (ii) is an easy consequence of transitivity on arcs. Likewise, part (iii) follows from the claim and arc-transitivity. 
 This completes the proof of the proposition. \end{proof}

{\bf Proof of Theorem~\ref{infnot2hom}.} This follows immediately from Propositions~\ref{recover1}  and \ref{recover2}. \qed

\section{Concluding remarks}

The problem of classifying all countably infinite set-homogeneous a-digraphs appears to be hard. 
As a starting point for the investigation of this problem, suppose $M$ is a set-homogeneous a-digraph. By Theorem~\ref{infnot2hom}, we may suppose that $M$ is 2-homogeneous, so the orbital $\Lambda$ is self-paired. 

In the particular case when $M$ is a countably infinite set-homogeneous a-digraph with {\em imprimitive} automorphism group, by Lemma~\ref{extra}
 there is a block system $\{B_i:i\in I\}$ such that one of the following holds.
 
\begin{enumerate}[(i)] 
\item 
The blocks are isomorphic set-homogeneous tournaments, and there are no arcs between the blocks.
\item
The blocks are independent sets, and on each block the group induced by $\Aut(M)$ is highly homogeneous, so (assuming the blocks are infinite)
 is either highly transitive,
or, by a theorem of Cameron \cite{cameron}, preserves or reverses  a linear or circular order.
\end{enumerate}
We shall say that $M$ is {\em connected} if the graph obtained by forgetting the arc orientation is connected. By the last remark, if $M$ is disconnected, then  it is a disjoint union of set-homogeneous tournaments, in which case Proposition~\ref{tourninf} below provides some information. Thus, we assume $M$ is connected.

We shall suppose first that $M$ is not a tournament.
If 
$\alpha \in M$ then $G_\alpha$ has three orbits on $M\setminus\{\alpha\}$, namely $\Gamma(\alpha)$, $\Gamma^*(\alpha)$, and $\Lambda(\alpha)$ (which equals $\Lambda^*(\alpha)$, by assumption).

Suppose that
$(\Gamma\circ \Gamma) \subseteq \Gamma$. 
Then $\to$ is transitive, so $(M, \to)$ is a set-homogeneous countably infinite partial order, so is homogeneous, by \cite[Theorem 8.13]{droste}. These are classified in \cite{Schmerl1}.

Thus, we may suppose $(\Gamma\circ \Gamma)\not\subseteq \Gamma$.
 In this case $M$ embeds either $D_3$ or $D_4$.
 For if there is an arc from $\Gamma(\alpha)$ to
 $\Gamma^*(\alpha)$ then $M$ embeds $D_3$.
 And if not, then there are $\alpha\rightarrow \beta\rightarrow \gamma$ with $(\alpha,\gamma)\in \Lambda$, and as $\Lambda$ 
is self-paired, there is $\delta$ with $\gamma\rightarrow\delta\rightarrow \alpha$, yielding a $D_4$. 

Let $\diam(M)$ be the smallest $d$ such that for any two distinct vertices $\alpha,\beta$
there is a directed path from $\alpha$ to $\beta$ of length no greater than $d$.
An easy analysis now
yields the following
 three possible cases, under the given assumptions (that $M$ is a countably infinite connected 2-homogeneous and set-homogeneous a-digraph which is not a tournament).
\begin{enumerate}[(i)]
\item 
$M$ embeds $D_3$, $\diam(M)=3$, in which case $(\Gamma\circ \Gamma) \cap \Lambda =\emptyset$ and $(\Gamma \circ \Gamma)\supseteq \Gamma^*$.
\item
$M$ embeds $D_3$, $\diam(M)=2$, in which case $(\Gamma\circ \Gamma)\supseteq \Lambda\cup\Gamma^*$.
\item
$M$ embeds $D_4$ but not $D_3$, $\diam(M)=3$, in which case $(\Gamma\circ \Gamma) \supseteq \Lambda$ and $(\Gamma\circ \Gamma)\cap \Gamma^*=\emptyset$.
\end{enumerate}

For tournaments,  we have the following result. 

\begin{proposition}\label{tourninf}
Let $T$ be a  set-homogeneous tournament. Then $T$ is $k$-homogeneous for all $k\leq 4$.
\end{proposition}

\begin{proof} 
We first show that $T$ is $3$-homogeneous.

From Lemma~\ref{sethomtourn}, we may suppose that $T$ is infinite. By Ramsey's Theorem, $T$ embeds a chain of 
length 3. Thus,
by 2-set-homogeneity, for each arc $x\rightarrow y$, there is $z$ such that $x\rightarrow z$, and $z\rightarrow y$. 

We may suppose that $T$ embeds a 3-cycle, since otherwise it is a dense linear order, so is 3-homogeneous.
 Thus, pick a 3-cycle
$x_1\rightarrow x_2\rightarrow x_3\rightarrow x_1$. By the previous paragraph, there are $y_1,y_2,y_3$ such that 
$x_i\rightarrow y_i\rightarrow x_{i+1~({\rm mod}~ 3)}$ for each $i=1,2,3$. Observe that $y_1,y_2,y_3$ are distinct. Either at least two of the $y_i$ each dominate at least
 two of the $x_i$, or at least two of the $y_i$ are each dominated by at least two of the $x_i$. Without loss we suppose 
 the former, with say $y_1\rightarrow x_3$, $y_2\rightarrow x_1$. Now there is $g\in G$ with
 $\{x_1,x_2,x_3,y_1\}^g=\{x_2,x_3,x_1,y_2\}$. Furthermore, the tournament induced on $\{x_1,x_2,x_3,y_1\}$ admits
 no automorphisms, since it contains just two copies of $D_3$ 
 ($\{x_1,x_2,x_3\}$ and $\{x_1,y_1,x_3\}$)  so any automorphism must fix their intersection $\{x_1,x_3\}$ and hence fixes the tournament pointwise.
 It follows that $(x_1,x_2,x_3)^g=(x_2,x_3,x_1)$, so the cyclic group $Z_3$ is induced on each triangle. 
This ensures 3-homogeneity. Note that this argument
 applies to any $G\leq \Aut(T)$ which acts set-homogeneously.
 
Recall that the automorphism group of any finite tournament has odd order, since any involution would have to reverse some pair of distinct vertices.
To see that $T$ is 4-homogeneous, note that any non-trivial odd order subgroup of $S_4$ is cyclic of order 3, so any non-rigid 4-vertex
 tournament 
has a vertex dominating the other three, or dominated by the other three. We must show that $\Aut(T)$ induces $\mathbb{Z}_3$ on such a 
tournament.
So consider $\alpha,\beta_1,\beta_2,\beta_3$ where $\{\beta_1,\beta_2,\beta_3\} \subset \Gamma(\alpha)$ carries a copy of $D_3$.
 Clearly $\Aut(T)_\alpha$ acts set-homogeneously on the tournament $\Gamma(\alpha)$, so, by the last paragraph, induces ${\mathbb Z}_3$ on copies of
$D_3$ in $\Gamma(\alpha)$. Thus, $\Aut(T)$ induces the full group of automorphisms on
$\{\alpha,\beta_1,\beta_2,\beta_3\}$. \end{proof}

\begin{problem}
Does there exist a countably infinite tournament that is set-homogeneous but not homogeneous?
\end{problem}
 
 \medskip

We conclude with some remarks about set-homogeneous  subgroups of full automorphism groups.  It is convenient to use the language of topological groups. Recall (see \cite[Section 2.4]{cameron2}) that if $M$ is a countably infinite set, then there is a topology on $\Sym(M)$, for which the basic open sets are the cosets of pointwise stabilisers of finite sets. A subgroup $G$ of $\Sym(M)$ is then closed in $\Sym(M)$ if and only if $G$ is the automorphism group of some relational first order structure  with domain $M$. If $G\leq \Sym(M)$ is closed, then $H\leq G$ is dense in $G$ if and only if $H$ has the same orbits as $G$ on $M^n$ for all $n\geq 1$. 

If $M$ is a countably infinite homogeneous structure, we say that $G\leq \Aut(M)$ {\em acts set-homogeneously} on $M$, or is a {\em set-homogeneous subgroup} ~ of $\Aut(M)$, if, whenever $U,V$ are isomorphic finite substructures of $M$, there is $g\in G$ with $U^g=V$.

\begin{problem}
Which countably infinite homogeneous structures $M$ have the property that every set-homogeneous subgroup of $\Aut(M)$ is dense in $\Aut(M)$?
\end{problem}

We remark that if $M$ has no structure (that is, $\Aut(M)=\Sym(M)$) then $G\leq \Aut(M)$ acts set-homogeneously if and only if $G$ acts highly homogeneously on $M$, that is, $G$ is $k$-homogeneous on $M$ for all $k\geq 1$. By the theorem of Cameron \cite{cameron} mentioned above, it follows
 that $\Aut(M)$ has (up to conjugacy in $\Sym(M)$) four proper non-trivial set-homogeneous closed subgroups, namely the automorphism groups of a linear order, a circular order, a linear betweenness relation, or an
(arity four) separation relation.

Following \cite[Section 6]{dgms}, we shall say that a structure $M$ is {\em locally rigid} if 
for every finite substructure $U$ of $M$, there is a finite substructure $V$ of $M$ containing $U$ such that every automorphism of $V$ fixes $U$ pointwise.

\begin{proposition} \label{rigid} Let $M$ be a locally rigid countably infinite homogeneous relational structure. Then any set-homogeneous subgroup of $\Aut(M)$ is dense in $\Aut(M)$.
\end{proposition}

\begin{proof} Let $G\leq \Aut(M)$ be set-homogeneous, and let $f:U_1\to U_2$ be an isomorphism between finite substructures of $M$. By local rigidity, there is a substructure $V_1$ of $M$ containing $U_1$ so that any automorphism of $V_1$ fixes $U_1$ pointwise. As $M$ is homogeneous, there is $h\in \Aut(M)$ extending $f$. Put $V_2:=V_1^h$. Then $V_1 \cong V_2$, so as $G$ acts set-homogeneously, there is $g\in G$ with $V_1^g=V_2$. Then by choice of the $V_i$, $g$ extends $f$, as required.
\end{proof}

\begin{remark} \rm By \cite[Proposition 6.1]{dgms}, if $\Gamma$ is an infinite graph, and $\Gamma$ has the property that for any two distinct vertices $x$ and $y$, there are infinitely many vertices joined to $x$ and not $y$, and infinitely many joined to $y$ but not $x$, then $\Gamma$ is locally rigid. It follows from this and Proposition~\ref{rigid} that if $\Gamma$ is the random graph, or the universal homogeneous $K_n$-free graph (see \cite[4.10]{cameron2}), then any set-homogeneous subgroup of $\Aut(\Gamma)$ is dense in $\Aut(\Gamma)$.
\end{remark}

\end{document}